\documentclass[a4paper,12pt]{article}

\usepackage[T1]{fontenc}
\usepackage{amsmath}
\usepackage{amssymb}
\usepackage{amsfonts}
\usepackage{amsthm}
\usepackage{indentfirst}
\usepackage{psfrag}
\usepackage{amscd,stmaryrd,latexsym}

\newtheorem{thm}{Theorem}[section]
\newtheorem{lem}{Lemma}[section]
\newtheorem{cor}{Corollary}[section]

\newtheorem{Prop}{Proposition}

\theoremstyle{remark}

\theoremstyle{definition}

\theoremstyle{remark}
\newtheorem{oss}{Remark}[section]

\def\eps{\mathop{\varepsilon}}
\def\partone{\mathop{\sharp_1}}
\def\parttwo{\mathop{\sharp_2}}
\def\psv{\mathop{\Psi}}
\def\fiav{\mathop{\tilde{\varphi}}}
\def\aav{\mathop{\tilde{\alpha}}}

\def\ovr{\mathop{\overline{r}}}
\newcommand{\be}{\begin{equation}}
\newcommand{\ee}{\end{equation}}
\newcommand{\R}{\mathbb{R}}
\newcommand{\N}{\mathbb{N}}
\newcommand{\C}{\mathbb{C}}
\newcommand{\CP}{\mathbb{C}\mathbb{P}}
\newcommand{\Z}{\mathbb{Z}}
\newcommand{\Q}{\mathbb{Q}}

\newcommand\res{\mathop{\hbox{\vrule height 7pt width .5pt depth 0pt
\vrule height .5pt width 6pt depth 0pt}}\nolimits}

\def\ti{\tilde}
\def\lf{\left}
\def\rg{\right}

\def\al{\alpha}
\def\la{\lambda}

\def\ep{\varepsilon}
\def\ds{\displaystyle}
\def\ov{\overline}

\def\om{\omega}
\def\p{\partial}

\begin{document}

\title{\textbf{The regularity of Special Legendrian Integral Cycles }}
\author{\textit{Costante Bellettini} \\ ETH Z\"urich \and \textit{Tristan Rivi\`ere}\\ ETH Z\"urich}
\date{}
\maketitle

\textbf{Abstract}. \textit{Special Legendrian Integral Cycles in $S^5$} are the links of the tangent cones to \textit{Special Lagrangian integer multiplicity rectifiable currents} in Calabi-Yau 3-folds. We show that such \textit{Special Legendrian Cycles} are smooth except at isolated points.

\section{Introduction}

Some years ago, in a survey paper \cite{DoT}, S.K. Donaldson and R.P. Thomas gave a fresh
boost to the analysis of non-linear \textit{gauge theories} in geometry by
exhibiting heuristically links between some invariants in complex
geometry and spaces of solutions to Yang-Mills equations in dimensions
higher than the usual conformal 4 dimensions for these equations.  In
\cite{Ti} G. Tian described the loss of compactness of sequences of some \textit{Yang-Mills Fields} in dimension larger than 4. This loss of compactness arises along $(n-4)$-rectifiable objects, called the blow-up sets. It plays a crucial role in the compactification procedure of the space of the solutions of $\Omega$-anti-self-dual instantons (the generalisation of the usual 4-dimensional instantons to dimensions larger than 4).

Can one expect the blow-up set to be more than just 
rectifiable?  What is its exact nature?

At such a level of generality this question is wide open and difficult. The situation is better understood for some sub-classes of solutions: one example is given by the so-called $SU(4)$-Instantons in a Calabi-Yau 4-fold.
The concentration set is, in this case, the carrier of a \textit{calibrated} rectifiable cycle. Among these cycles we find for instance the {\it Special Lagrangian Integral Currents}. This provides one possible field of application for \textit{Special Lagrangian Geometry} or \textit{calibrated geometries} in general.

\medskip

Further reasons for studying \textit{Special Lagrangians} come from \textit{String Theory}, more precisely from \textit{Mirror Symmetry}. According to this model, our universe is a product of the standard Minkowsky space $\R^4$ with a Calabi-Yau 3-fold $Y$. Based on physical grounds, the so called \textit{SYZ-conjecture} (named after Strominger, Yau and Zaslov) expects, roughly speaking, that this Calabi-Yau 3-fold can be fibrated by (possibly singular) Special Lagrangians, whence the interest in understanding the singularities of a \textit{Special Lagrangian current}. The compactification of the \textit{dual fibration} should lead to the mirror partner of $Y$. See the survey paper by Joyce \cite{Joyce} for a more thorough explanation.

\medskip
We remark also that, as all calibrated geometries (see \cite{HL} or \cite{Haskins}), \textit{Special Lagrangian Geometry} provides examples of \textit{volume-minimizing} submanifolds or currents; Special Lagrangians are a particularly large family. Having such examples helps the understanding of the possible singular behaviour of such minimizers.

\medskip

\textbf{General description of the problem: setting and results}.
In the complex euclidean space $\C^3$ with the standard coordinates $z=(z_1, z_2, z_3)$, $z_i=x_i+i y_i$, consider the constant differential 3-form 
\[\Omega = Re(dz^1 \wedge dz^2 \wedge dz^3).\]
This is the so called \textit{Special Lagrangian calibration}, introduced and analysed in \cite{HL}. We recall some notions from calibrated geometry, referring to the quoted paper for a broader exposition. Given a $p$-form $\phi$ on a Riemannian manifold $(M,g)$, the comass of $\phi$ is defined to be
\[||\phi||^*:=  \sup \{\langle \phi_x, \xi_x \rangle: x \in M, \xi_x \text{ is a unit simple $p$-vector at } x\}.\]
A form $\phi$ of comass one is called a \textit{calibration} if it is closed ($d \phi = 0$); when it is non-closed it is referred to as a \textit{semi-calibration}. 

Let $\phi$ be a \textit{calibration} or a \textit{semi-calibration}; among the oriented $p$-dimensional planes that constitute the Grassmannians $G(p, T_xM)$, we pick those that (represented as unit simple $p$-vectors) realize $\langle \phi_x, \xi_x \rangle = 1$ and define the set $\mathcal{G}(\phi)$ of  \textit{$m$-planes calibrated by $\phi$}:
\[\mathcal{G}(\phi) = \cup_{x \in M} \{\xi_x \in G(p, T_xM): \langle \phi_x, \xi_x \rangle = 1 \}.\]

\medskip

We recall now the notion of calibrated cycle. For definitions and notations from Geometric Measure Theory we refer to \cite{F} or \cite{Giaquinta}. 

\medskip

An \underline{integer $m$-cycle} $C$ in $M$ is an integer multiplicity rectifiable current of dimension $m$ without boundary, i.e.
\begin{description}
	\item[(i)] \textit{Rectifiability}: there is a countable family of oriented $C^1$ submanifolds $N_i$ of dimension $m$ in $M$; in each of them we take a $\mathcal{H}^m$-measurable subset $\mathcal{N}_i$, so that the $\mathcal{N}_i$-s are disjoint; the union $\mathcal{C}=\cup_i \mathcal{N}_i$ is a so-called \textit{oriented rectifiable set}. 
	
	$\mathcal{C}$ possesses an \textit{oriented approximate tangent plane} $\mathcal{H}^m$-a.e. (see \cite{F} or \cite{Giaquinta}). On $\mathcal{C}$ an integer valued and locally summable \textit{multiplicity function} $\theta$ is given, $\theta \in L^1_{\text{loc}}(\mathcal{C}; \Z)$; the action of the current $C$ on any $m$-form $\psi$ which is smooth and compactly supported in $M$, is given by
	\[C(\psi)= \int_{\mathcal{C}} \theta(x) \langle \psi_x, \xi_x \rangle d \mathcal{H}^m(x),\]
where $\xi_x$ is the tangent at $x$ represented as a unit simple vector.

	\item[(ii)] \textit{Closedness}: the boundary $\partial C$ of the current is $0$ \footnote{The term \textit{cycle} refers to the absence of boundary.}. Explicitly: for any smooth $(m-1)$-form $\alpha$, which is compactly supported in $M$,
	\[(\partial C)(\alpha):=C(d \alpha) = 0.\]
\end{description}

The class of integer-multiplicity, rectifiable currents of dimension $m$ in $M$ is denoted by $\mathcal{R}_m(M)$. The support $spt(C)$ of the current is defined as the complement of the open set 
\[\cup \{A: C(\psi)=0 \text{ for all $m$-forms } \psi \text{ compactly supported in } A\}.\]

The underlying rectifiable set $\mathcal{C}$ is sometimes referred to as the \textit{carrier} of the current $C$.

\medskip

We recall the notions of \underline{Smooth Points and Singular Points}. A point $x \in \mathcal{C}$ is said to be a \textit{smooth point} if there is a ball $B_r(x)$ in which the current acts as a smooth $m$-submanifold $\mathcal{V}$, i.e. if there is some constant $N \in \mathbb{N}$ such that for any smooth $m$-form $\psi$ compactly supported in $B_r(x)$ 
\[C(\psi)=N \int_{\mathcal{V}} \psi. \]

The set of smooth points is open in $\mathcal{C}$; its complement in $\mathcal{C}$ is called the \textit{singular set of $C$}, denoted by \textit{Sing C}.

\medskip

For a current in $\mathcal{R}_m(M)$, at $\mathcal{H}^m$-almost every point $x \in \mathcal{C}$ denote by $T_x \mathcal{C}$ the $m$-dimensional oriented approximate \textit{tangent plane} to the underlying rectifiable set $\mathcal{C}$; given a (semi)-calibration $\phi$, $C$ is said to be \textit{calibrated by $\phi$} if 
\[\text{for } \mathcal{H}^m \text{-almost every } x, \;T_x \mathcal{C} \in \mathcal{G}(\phi).\]
When $\phi$ is a closed form, then a current calibrated by $\phi$ is locally homologically volume-minimizing; (closed) calibrations were introduced in the foundational paper \cite{HL}.

\medskip

Returning to our case, being $\Omega$ constant, it is obviously closed; as shown in \cite{HL} it has comass one. Currents in $\mathcal{R}_3(\C^3)$, calibrated by $\Omega$, are called \textit{Special Lagrangians}.

Let $N$ denote the radial vector field $N:=r \frac{\partial}{\partial r}$ in $\mathbb{C}^3$ and define the \emph{normal part} of $\Omega$ by
\[\Omega_N := \iota_N \Omega,\]
where $\iota$ is the interior product. We will work in the sphere $S^5 \subset \mathbb{C}^3$, with the induced metric. Consider the pull back of $\Omega_N$ on the sphere via the canonical inclusion map $\mathcal{E}: S^5 \hookrightarrow \mathbb{C}^3$:
\[\omega := \mathcal{E}^\ast \Omega_N.\]
An easy computation shows that
\[\omega= Re(z_1 dz^2 \wedge dz^3 + z_2 dz^3 \wedge dz^1 + z_3 dz^1 \wedge dz^2).\]
$\omega$ is a 2-form on $S^5$ of comass one. Indeed, $|N|=1$ on $S^5$ and for any simple 2-vector $\xi$ in $T S^5$
\[|\omega(\xi)|=|\Omega(N\wedge\xi)| \leq \|N\wedge\xi\|=\|\xi\|.\]
Equality is surely reached when $N\wedge\xi$ is a Special Lagrangian $3$-plane, compare Proposition \ref{Prop:leglagr}.
We remark that both $\Omega$ and $\omega$ are $SU(3)$-invariant. As explained in \cite{HL} (Section II.5) or \cite{Haskins} (Section 2.2), $\omega$ is non-closed. 

$\omega$ is referred to as the \textit{Special Legendrian semi-calibration}.
Rectifiable currents in $S^5$ calibrated by $\omega$ are called \textit{Special Legendrians}.

Our main result is the following:

\begin{thm}
\label{thm:main}
An integer multiplicity rectifiable current $C$ without boundary calibrated by $\omega$ (this is called a Special Legendrian integral cycle) in $S^5$ can only have isolated singularities (therefore finite). 

In other words: $C$ is, out of isolated points, the current of integration along a smooth Special Legendrian submanifold with smooth integer multiplicity.
\end{thm}

\begin{oss}
This result is optimal. We will provide an example in the next section, see remark \ref{oss:example}.
\end{oss}

Still from \cite{HL} (Section II.5) or \cite{Haskins} (Section 2.2), the 2-currents of $S^5$ on which $\omega$ restricts to the area form are exactly those such that the cone built on them is calibrated by $\Omega$:

\begin{Prop}
\label{Prop:leglagr}
(\cite{HL} or \cite{Haskins}) A rectifiable current $T$ in $S^5$ is a Special Legendrian if and only if the cone on $T$ 
\[C(T)= \{ tx \in \mathbb{R}^6 : x \in T , t>0\}\]
is Special Lagrangian.
\end{Prop}

We know that Special Lagrangian currents (as a particular case of currents calibrated by a closed form) are (locally) homologically area-minimizing in $\mathbb{C}^3$; from \cite{Alm} we know that volume-minimizing $3$-cycles are smooth outside a set of Hausdorff dimension $1$. In the case of a cone, this roughly translates into having radial lines of singularities, possibly accumulating onto each other. We establish here that there can only by a finite number of such lines.

We remark here that Special Lagrangians can be defined in general \textit{Calabi-Yau n-folds}, see \cite{Joyce}; Special Lagrangians are known to possess tangent cones at all points (see \cite{HL} sect. II.5), and such cones are Special Lagrangian cones in $\C^n$. Thanks to Proposition \ref{Prop:leglagr}, our result can be restated as follows:

\begin{cor}
Tangent cones to a Special Lagrangian in a Calabi-Yau $3$-fold have a singular set made of at most finitely many lines passing through the vertex.
\end{cor}

From \cite{Simons} (Prop. 6.1.1), $T$ in $S^5$ is minimal, in the sense of vanishing mean curvature, if and only if $C(T) \subset \mathbb{C}^3$ is minimal. Therefore, Special Legendrians are minimal currents in $S^5$ (although not necessarily area-minimizing).

Relying on \cite{Alm}, Chang proved in \cite{Chang} the corresponding regularity result for area-minimizing $2$-dimensional currents.
 
One advantage coming from the existence of the calibration, as will be seen, is the fact that the current can locally be described as integration along a multi-valued graph satisfying a first order elliptic PDE; the general problem of volume-minimizing currents, instead, requires an elliptic problem of order two, see \cite{Alm} or  \cite{Chang}. It is also remarkable that the general regularity theory for mass-minimizing currents developed by Almgren is extremely hard; his Big Regularity Paper \cite{Alm} comprises a thousand pages and it is therefore helpful to have shorter (and relatively easier) self-contained proofs of regularity results for some sub-classes of minimizing currents, such as Special Lagrangian currents or $J$-holomorphic currents (see \cite{Taubes}, \cite{RT1}, \cite{RT2}).

\medskip

\textbf{The proof}. We are now giving the sketch of our proof. We are basically following the same structure as \cite{Taubes} and \cite{RT1}, where the regularity of $J$-holomorphic cycles in a $4$-dimensional ambient manifold was shown. In our case we have a fifth coordinate to deal with, which introduces new challenging difficulties, as will be seen.

A standard blow-up analysis tells us that at any point $x$ of $S^5$ the multiplicity function $\displaystyle  \theta(x)= \lim_{r\to 0}\frac{M(C \res B_r(x))}{\pi r^2}$  is\footnote{For general integral cycles, the limit $\lim_{r\to 0}\frac{M(C \res B_r(x))}{\pi r^2}$ exists a.e. and coincides with the absolute value $|\theta|$ of the multiplicity  assigned in the definition of integer cycle. In our case $\lim_{r\to 0}\frac{M(C \res B_r(x))}{\pi r^2}$ is well-defined everywhere, therefore we can choose (everywhere) this natural representative for $\theta$, after having chosen the correct orientation for the approximate tangent plane.} an integer $Q$. The monotonicity formula (see \cite{PR} or \cite{Sim}) tells us that, at any $x_0$, $\displaystyle \frac{M(C \res B_r(x_0))}{r^2}$ is monotonically non-increasing as $r \downarrow 0$, whence we get that $\theta$ is upper semi-continuous, therefore the set 
\[\mathcal{C}^Q :=\{x \in S^5: \theta(x) \leq Q\}\] 
is open in $S^5$; this allows a proof by induction of our result, indeed we can restrict the current to $\mathcal{C}^Q$ and consider increasing integers $Q$ (see section \ref{PDEaverage}).

One key ingredient is the construction of families of $3$-dimensional surfaces $\Sigma$ which locally foliate $S^5$ and that have the property of intersecting positively the Special Legendrian ones. As in \cite{RT1}, this algebraic property can be exploited to provide a self-contained proof of the uniqueness of tangent cones for our current. This result was proved for general semi-calibrated cycles in \cite{PR} and for general area-minimizing ones in \cite{White} using a completely different approach\footnote{The proof in \cite{White} relies however on the area-minimality property which is not generally true for Special Legendrians.}. Further, the positiveness of intersection allows us to describe our current as a multivalued graph from a disk of $\C$ into $\mathbb{\R}^3$. 

Currents of integration along multivalued graphs constitute one of the important objects of interest in Geometric Measure Theory. Multivalued graphs were introduced by Almgren in \cite{Alm} for the study of Dirichlet-minimizing and volume-minimizing currents and were lately revisited in a new flavour in \cite{DeLelSpa}.

The inductive step is divided into two parts: in the first one we show that there is no possibility for an accumulation of singularities of multiplicity $Q$ to a singularity of the same multiplicity. In the second part we exclude accumulation of lower order singularities to a singularity of higher order.

For the first, we introduce the first order PDEs that describe the calibrating condition. These equations turn out to be, in appropriate coordinates, perturbations of the classical Cauchy-Riemann equations, with three real functions and two real variables, however. Using these PDEs we prove a $W^{1,2}$ estimate for the average of the branches of our multivalued graph. We remark here that in theorem \ref{thm:w12} we give a proof of the $W^{1,2}$-estimate different than the one in \cite{RT1}, where the authors had the further hypothesis that \textit{Sing C} was $\mathcal{H}^2$-negligible. Then, in section \ref{uniquecont}, by a suitable adaptation of the unique continuation argument used in \cite{Taubes}, we prove that the multivalued graph obtaining by subtracting the average from each branch cannot have accumulation of zeros, thereby concluding the first part of the inductive step. The techniques we employ to show the partial integration formulas for multivalued graphs are more typical of geometric measure theory; we also provide in lemma \ref{lem:partint} a step that was incomplete in \cite{Taubes}.

For the second part of the inductive step we use an homological argument inspired by the one used in \cite{Taubes}, where the same statement was proved in the case of $J$-holomorphic cycles in a $4$-manifold, although in our case the existence of the fifth coordinate induces new difficulties and a more involved argument.

\medskip

\textbf{Acknowledgments}: the authors are very grateful to Gang Tian for having suggested this problem to them and for very fruitful discussions.

\section{Preliminaries: the construction of positively intersecting foliations}
\label{preliminaries}

In this section we are going to construct in a generic way a smooth 3-surface $\Sigma$ in $S^5$ with the property that, anytime $\Sigma$ intersects a special Legendrian $L$ transversally, this intersection is positive, i.e., the orientation of $T_p L \wedge T_p \Sigma$ agrees with that of $T_p S^5$ ($S^5$ being oriented according to the outward normal). Then we will construct foliations made with families of 3-surfaces of this kind.

\medskip
\textbf{Contact structure}.
Now we recall some basic facts on the geometry of the contact structure associated to the Special Legendrian calibration in $S^5$, see \cite{Haskins} for more details. 

$S^5$ inherits from the symplectic manifold $\displaystyle (\mathbb{C}^3, \sum_{i=1}^3 dz^i \wedge d\overline{z}^i)$ the contact structure given by the form 
\[\gamma := \mathcal{E}^\ast \iota_N (\sum_{i=1}^3 dz^i \wedge d\overline{z}^i).\]
This is a 1-form with the contact property saying that $\gamma \wedge (d \gamma) ^2 \neq 0$ everywhere; the associated distribution of hyperplanes is $ker (\gamma (p)) \subset T_p S^5$. In the sequel the hyperplane of the distribution at $p$ will be denoted by $H_p^4$, where $H$ stands for horizontal \footnote{This is nothing else but the universal horizontal connection associated to the Hopf projection $S^5 \to \CP^2$ sending $(z_1, z_2, z_3)  \to [z_1, z_2, z_3]$. The fibers $e^{i\theta}p$, $\theta \in [0,2\pi]$ and $p \in S^5$, are great circles in $S^5$ and the hyperplanes $H^4_p$ of the horizontal distribution are everywhere orthogonal to the fibers. This structure is $SU(3)$-invariant.}. The condition on $\gamma$ is equivalent to the non-integrability of this distribution, i.e. it is impossible (even locally) to find a 4-surface in $S^5$ which is everywhere tangent to the $H^4$. The vectors $v$ orthogonal to $H^4$ are called vertical; they are everywhere tangent to the Hopf fibers $e^{i \theta}(z_1,z_2,z_3) \subset S^5$.

\medskip

\textbf{Special Legendrians are tangent to the horizontal distribution}.
The Special Legendrian calibration $\omega$ has the property that any calibrated 2-plane in $T S^5$ must be contained in $H^4$. Therefore, Special Legendrian submanifolds are everywhere tangent to the horizontal distribution and they are a particular case of the so called Legendrian curves, which are the maximal dimensional integral submanifolds of the contact distribution. We can shortly justify this as follows: recall that $\omega$ and the horizontal distribution are invariant under the action of $SU(3)$. At the point $(1,0,0) \in S^5$ the Special Legendrian semi-calibration is easily\footnote{Recall that we are using standard coordinates $z_j=x_j+iy_j$, $j=1,2,3$ on $\C^3$.} computed: $\omega_{(1,0,0)}= dx^2 \wedge dx^3 - dy^2 \wedge dy^3$. Then if a unit simple 2-vector in $T_{(1,0,0)} S^5$ is calibrated, it must lie in the 4-plane spanned by the coordinates $x_2, y_2, x_3, y_3$, which is the horizontal hyperplane $H^4_{(1,0,0)}$ orthogonal to the Hopf fiber $e^{i \theta} (1,0,0)$. The $SU(3)$-invariance of $\omega$ and of $\{H^4\}$ implies that, at all points on the sphere, Special Legendrians are tangent to the horizontal distribution.

\medskip

\textbf{J-structure and J-invariance}.
We introduce now a further structure: on each hyperplane $H_p^4$, $\omega$ restricts to a non-degenerate 2-form, so we get a symplectic structure and we can we define the (unique) linear map
\[J_p : H_p^4 \to H_p^4\]
characterized by the properties that $J_p^2=-Id$ and, for $v,w \in H_p^4$,
\begin{equation}
\label{eq:sljh}
\omega(p)(v,w)= \omega(p)(J_p v,J_p w),\;\;\; \langle v, w \rangle_{T_p S^5} = \omega(p)(v,J_p w).
\end{equation}
This is a standard construction from symplectic geometry and the uniqueness of the $J_p$ at each point implies that we get a smooth endomorphism of the horizontal bundle; in our case the setting is simple enough to allow an explicit expression of $J_p$ in coordinates, as follows.

$\omega_{(1,0,0)}= dx^2 \wedge dx^3 - dy^2 \wedge dy^3$ and recall that $H^4_{(1,0,0)}$ is spanned by the coordinates $x_2, y_2, x_3, y_3$. Then choose 
\[ J_{(1,0,0)} := \left \{ \begin{array}{lll}
\frac{\partial}{\partial x_2} & \to &  \frac{\partial}{\partial x_3}   \\ 
\frac{\partial}{\partial y_2} & \to & - \frac{\partial}{\partial y_3} 
\end{array} \right .\]
The conditions in (\ref{eq:sljh}) hold true at this point.

For any $p \in S^5$, take $g \in {SU(3)}/{SU(2)}$ sending $p$ to $(1,0,0)$. The $SU(2)$ in the quotient is the stabilizer of $H^4_{(1,0,0)}$. This stabilizer leaves $J_{(1,0,0)}$ invariant (any element of $SU(2)$ commutes with $J_{(1,0,0)}$) and we can define, for $v \in H^4_p$,
\[J_p(v):= dg^{-1}(J_{(1,0,0)}(dg(v))).\]
Thus we get a smooth $J$-structure on the horizontal bundle.

From the properties in (\ref{eq:sljh}), if a simple unit 2-vector $v \wedge w$ in $H^4_p$ is calibrated by $\omega$, then
\[1=\omega_p(v,w)= \omega_p(J_p v,J_p w)= \langle J_p v, w \rangle_{T_p S^5} \]
so
\[v \wedge w \text{ is a Special Legendrian plane }\Leftrightarrow J_p(v \wedge w):= J_p v \wedge J_p w= v \wedge w,\]
i.e. 

\begin{Prop}
\label{Prop:slc}
A 2-plane in $T_p S^5$ is Special Legendrian if and only if it lies in $H_p^4$ (horizontal for the Hopf connection) and it is $J_p$-invariant for the $J$-structure above.
\end{Prop}

\medskip
Since all the above introduced objects are invariant under the action of $SU(3)$, we can afford to work at a given point of $S^5$; from now on we will focus on a neighbourhood of the point $(1,0,0) \in S^5$, where we are using the complex coordinates $(z_1, z_2, z_3)=(x_1, y_1,x_2, y_2, x_3, y_3)$ of $\mathbb{C}^3$.
\medskip

\textbf{Positive 3-surface}.
We are now ready for the construction of a 3-surface with the property of positive intersection.

Oriented $m$-planes in $\C^3$ will be identified with unit simple $m$-vectors in $\C^3$. In particular, $T S^5$ is oriented so that $T S^5 \wedge \frac{\partial}{\partial r}=\C^3$.

Writing down the Special Lagrangian calibration explicitly
\[\Omega=  dx^1 \wedge dx^2 \wedge dx^3 - dx^1 \wedge dy^2 \wedge dy^3 - dy^1 \wedge dx^2 \wedge dy^3 - dy^1 \wedge dy^2 \wedge dx^3,\]
it is straightforward to see that
\[\mathcal{L}_0 = \frac{\partial}{\partial x_1} \wedge \frac{\partial}{\partial x_2} \wedge \frac{\partial}{\partial x_3}\]
is a  Special Lagrangian 3-plane passing through the origin of $\C^3$ and through the point $(1,0,0)$.
We now consider the following family $\{\mathcal{L}_\theta\}_{\theta \in (-\eps, \eps)}$ of Special Lagrangian planes, where $\{(e^{i \theta},0,0)\}_{\theta \in (-\eps, \eps)}$ is the fiber containing $(1,0,0)$ and $\mathcal{L}_\theta$ goes through the point $(e^{i \theta},0,0)$:

\[\mathcal{L}_\theta = \left( \begin{array}{ccc}
e^{i \theta} & 0 & 0\\
0 & e^{-i \theta} & 0\\
0 & 0 & 1 \end{array} \right)_* \mathcal{L}_0 = \]\[=(\cos \theta \frac{\partial}{\partial x_1} + \sin \theta \frac{\partial}{\partial y_1} )\wedge (\cos \theta \frac{\partial}{\partial x_2} - \sin \theta \frac{\partial}{\partial y_2} ) \wedge \frac{\partial}{\partial x_3},\]
which is Special Lagrangian since it has been obtained by pushing forward $\mathcal{L}_0$ by an element in $SU(3)$.

We introduce the 4-surface $\Sigma^4$ in $\mathbb{C}^3$ obtained by attaching the $\mathcal{L}_\theta$-planes along the fiber $\{(e^{i \theta},0,0)\}_{\theta \in (-\eps, \eps)}$: this 4-surface can be expressed as 
\[\Sigma^4 = (a e^{i \theta}, b e^{-i \theta}, c)\]
parametrized with $(a,b,c)\in \mathbb{R}^3 \setminus \{0\}, \;\; \theta \in (-\eps, \eps)$. Then define
\[\Sigma = \Sigma^4 \cap S^5.\]
As stated in the coming lemma \ref{lem:positiveintersectionplanes}, this 3-surface has the desired property of intersecting Special Legendrians positively.
\medskip

We can make the equivalent construction starting from the form $\omega$ restricted to the fiber $\{(e^{i \theta},0,0)\}_{\theta \in (-\eps, \eps)}$:
\[\omega= \cos \theta(dx^2 \wedge dx^3 - dy^2 \wedge dy^3) + \sin \theta (-dx^2 \wedge dy^3 -dy^2 \wedge dx^3 )\]
and explicitly writing down the J-structure on $H_{(e^{i \theta},0,0)}^4$ introduced above. On $H_{(e^{i \theta},0,0)}^4$ we can use coordinates $(x_2, y_2, x_3, y_3)$ since $H^4 \wedge v = T S^5$, $T S^5 \wedge \frac{\partial}{\partial r} = \mathbb{C}^3$ and $v = i \frac{\partial}{\partial r}$, so $H^4 = \frac{\partial}{\partial x_2} \wedge \frac{\partial}{\partial y_2} \wedge \frac{\partial}{\partial x_3} \wedge \frac{\partial}{\partial y_3}$.
\[ J_\theta = J_{(e^{i \theta},0,0)}:= \left \{ \begin{array}{lll}
\frac{\partial}{\partial x_2} & \to & \cos \theta \frac{\partial}{\partial x_3} - \sin \theta \frac{\partial}{\partial y_3}  \\ 
\frac{\partial}{\partial y_2} & \to & -\cos \theta \frac{\partial}{\partial y_3} - \sin \theta \frac{\partial}{\partial x_3} \\ 
\frac{\partial}{\partial x_3} & \to & -\cos \theta \frac{\partial}{\partial x_2} + \sin \theta \frac{\partial}{\partial y_2} \\ 
\frac{\partial}{\partial y_3} & \to & \cos \theta \frac{\partial}{\partial y_3} + \sin \theta \frac{\partial}{\partial x_2}  \end{array} \right .\]
So $J_\theta$ is represented by the matrix $J_0 A_\theta$, where \footnote{In complex notation, looking at $H_{(e^{i \theta},0,0)}^4$ as $\C^2_{z_2, z_3}$, we can write  \[A_\theta = \left ( \begin{array}{cc}
e^{i \theta} & 0  \\ 
0 & e^{-i \theta} \end{array} \right ).\]}

\[J_0 = \left ( \begin{array}{cccc}
0 & 0 & -1 & 0  \\ 
0 & 0 & 0 & 1  \\
1 & 0 & 0 & 0  \\
0 & -1 & 0 & 0 \end{array} \right ),
A_\theta = \left ( \begin{array}{cccc}
\cos \theta & \sin \theta & 0 & 0  \\ 
-\sin \theta & \cos \theta & 0 & 0  \\
0 & 0 & \cos \theta & -\sin \theta  \\
0 & 0 & \sin \theta & \cos \theta \end{array} \right ).\]

If $v \wedge w$ is a J-invariant 2-plane in $H_0^4$, with $w=J_0 v$, then $A_\theta ^{-1} v \wedge w$ is $J_\theta$ invariant, in fact $J_\theta (A_\theta ^{-1} v) = J_0 A_\theta A_\theta ^{-1} v = J_0 v = w$. Take the geodesic 2-sphere $L_0$ tangent to the $J_0$-holomorphic plane 
\[\frac{\partial}{\partial x_2} \wedge \frac{\partial}{\partial x_3}.\]
This Special Legendrian 2-sphere $L_0$ coincides with $\mathcal{L}_0 \cap S^5$ introduced above.
The 2-plane 
\[A_\theta ^{-1}\frac{\partial}{\partial x_2} \wedge \frac{\partial}{\partial x_3} =  (\cos \theta \frac{\partial}{\partial x_2} - \sin \theta \frac{\partial}{\partial y_2}) \wedge \frac{\partial}{\partial x_3}\]
is therefore $J_\theta$ holomorphic and the geodesic 2-sphere tangent to it is $\mathcal{L}_\theta \cap S^5$. $\Sigma$ is the 3-surface obtained from the union of those Special Legendrian spheres as $\theta \in (-\eps, \eps)$.

\begin{lem}
\label{lem:positiveintersectionplanes}
There is an $\eps_0>0$ small enough such that for any $\eps <\eps_0$ the following holds:

let $S$ be any Special Legendrian current in $B_{\eps} (1,0,0) \subset S^5$; then, at any point $p$ where $T_p S$ is defined and transversal to $T_p \Sigma$,  $S$ and $\Sigma$ intersect each other in a positive way, i.e.
\[T_p S \wedge T_p \Sigma = T_p S^5.\]
\end{lem}

\begin{proof}[\bf{proof of lemma \ref{lem:positiveintersectionplanes}}]

\[T_{(e^{i \theta},0,0)} \Sigma = A_\theta ^{-1}\frac{\partial}{\partial x_2} \wedge \frac{\partial}{\partial x_3} \wedge v_\theta ,\]
so, along the fiber, the tangent space to $\Sigma$ is spanned by two vectors $l^1, l^2$ such that $l^1 \wedge l^2$ is Special Legendrian and by the vertical vector $v_\theta$. At any other point $p$ of $\Sigma$, the tangent space always contains two directions $l_p^1, l_p^2$ such that $l_p^1 \wedge l_p^2$ is Special Legendrian (from the construction of $\Sigma$). The third vector $w$, orthogonal to these two and such that $l_p^1 \wedge l_p^2 \wedge w = T \Sigma$, drifts from the vertical direction as the point moves away from the fiber, but by continuity, for a small neighbourhood $B_{\eps} (1,0,0) $, we still have that 
\[H_p^4 \wedge w_p = T_p S^5.\]
On the other hand, it is a general fact that, given a 4-plane with a J-structure, two transversal J-invariant planes always intersect positively. Therefore 
\[T_p S \wedge l_p^1 \wedge l_p^2 = H_p^4\]
at any point $p$, so
\[T_p S \wedge T_p \Sigma = T_p S \wedge (l_p^1 \wedge l_p^2 \wedge w_p) =(T_p S \wedge l_p^1 \wedge l_p^2 )\wedge w_p= T_p S^5.\]

\end{proof}

\medskip

\textbf{First parallel foliation}.
Now we are going to exhibit a 2-parameter family of 3-surfaces that foliate $B_{\eps} (1,0,0)$ and have the property of positive intersection. Consider the Special Legendrian 2-sphere
\[L = (- \frac{\partial}{\partial x_1} \wedge \frac{\partial}{\partial y_2} \wedge \frac{\partial}{\partial y_3}) \cap S^5.\]
This is going to be the space of parameters. Consider $SO(3)$ and let it act on the 3-space $- \frac{\partial}{\partial x_1} \wedge \frac{\partial}{\partial y_2} \wedge \frac{\partial}{\partial y_3}$. We are only interested in the subgroup of rotations having axis in the plane $\frac{\partial}{\partial y_2} \wedge \frac{\partial}{\partial y_3}$. This subgroup is isomorphic to $SO(3)/ S$, where $S$ is the stabilizer of a point, in our case the point $(1,0,0) \in - \frac{\partial}{\partial x_1} \wedge\frac{\partial}{\partial y_2} \wedge \frac{\partial}{\partial y_3}$. Thus the rotations in this subgroup can be parametrized over the points of $L =  \left(- \frac{\partial}{\partial x_1} \wedge\frac{\partial}{\partial y_2} \wedge \frac{\partial}{\partial y_3} \right)\cap S^5$ and we will write $A_q$ for the rotation sending $(1,0,0)$ to $q \in L$. We extend $A_q$ to a rotation of the whole $S^5$ by letting it act diagonally on $\mathbb{R}^3 \oplus \mathbb{R}^3 = \mathbb{C}^3$. Then define
\[\Sigma_q = A_q (\Sigma),\]
for $q \in L$. Since $A_q \in SU(3)$, Special Legendrian spheres are invariant and $A_q(e^{i \theta}(1,0,0))= e^{i \theta}A_q((1,0,0))=e^{i \theta} q$, so the fiber through $(1,0,0)$ is sent into the fiber through $q$. 
Therefore, for a fixed $q$, $\Sigma_q$ is a 3-surface of the same type as $\Sigma$, that is, it contains the fiber through $q$ and is made of the union of Special Legendrian spheres smoothly attached along the fiber. By the $SU(3)$-invariance of $\omega$, from lemma \ref{lem:positiveintersectionplanes} we get that $\Sigma_q$ has the property of intersecting positively any transversal Special Legendrian $S$.

For the sequel define $L_{\eps}=L \cap B_{\eps} (1,0,0)$.

\begin{lem}
\label{lem:firstfoliation}
The 3-surfaces $\Sigma_q$, as $q \in L_{\eps}$, foliate a neighbourhood of  $(1,0,0)$ in $S^5$.
\end{lem}

\begin{proof}[\bf{proof of lemma \ref{lem:firstfoliation}}]
Parametrize $L_{\eps}$ with normal coordinates $(s,t)$, with $\frac{\partial}{\partial s}= \frac{\partial}{\partial y_2}, \frac{\partial}{\partial t} =-\frac{\partial}{\partial y_3} $ and $\Sigma= \Sigma_0$ with $(a,b,c,\theta)\in (S^2 \cap B_{\eps}(1,0,0)) \times (-\eps, \eps)$, with $(a,b,c)\in S^2 \subset \mathbb{R}^3 = \frac{\partial}{\partial x_1} \wedge \frac{\partial}{\partial x_2} \wedge \frac{\partial}{\partial x_3}$ and $\theta \in (-\eps, \eps)$ as done during the construction (we set $a=(1-b^2-c^2)^{1/2}$).
Consider the function $\psi: \Sigma \times L_{\eps} \rightarrow S^5$ defined as
\[\psi (p,q) = A_q (p)\]
for $p=(b,c,\theta) \in \Sigma$, $q=(s,t) \in L_{\eps}$. Analysing the action of the differential $d \psi$ on the basis vectors at $(0,0) \in  \Sigma \times L_{\eps}$ we get:
\[ \frac{\partial \psi}{\partial b} = \frac{\partial}{\partial x_2}, \frac{\partial \psi}{\partial c} = \frac{\partial}{\partial x_3}, \frac{\partial \psi}{\partial \theta} = \frac{\partial}{\partial y_1}, \frac{\partial \psi}{\partial s} = \frac{\partial}{\partial y_2}, \frac{\partial \psi}{\partial t} = -\frac{\partial}{\partial y_3} .\]
so the Jacobian determinant at $0$ is $1$ and $\psi$ is a diffeomorphism in some neighbourhood of $(1,0,0)$ where we can introduce the new set of coordinates $(b,c,\theta,s,t)$.
Therefore, the family $\{\Sigma_{s,t}\}_{(s,t)\in L}$ foliates an open set that we can assume to be $\psi(\Sigma \times L_{\eps})$ if both $\Sigma$ and $L_{\eps}$ were taken small enough.
\end{proof}

\medskip

\textbf{Coordinates induced by the first parallel foliation}. Recall that, in each $H_p^4$ we are interested in the possible calibrated 2-planes, which, as shown above, must be $J_p$-invariant. The set of these 2-planes is parametrized by the complex lines in $\mathbb{C}^2$ and is therefore diffeomorphic to $\mathbb{CP}^1$. We are often going to identify $H^4$ with $\mathbb{C}^2$ (respectively $\mathbb{CP}^1$, if we are interested in the complex lines) with the following coordinates: on $H_{(1,0,0)}^4$ we set $H_{(1,0,0)}^4= T L \oplus T (L_0) = \mathbb{C}_{s+it} \oplus \mathbb{C}_{b+ic}$, where $L, L_0$ are the Special Legendrians introduced above; $T L,T L_0$ are $\mathbb{C}$-orthogonal complex lines in $H_{(1,0,0)}^4$, $T L=\frac{\partial}{\partial s} \wedge \frac{\partial}{\partial t}$ and $T L_0 = \frac{\partial}{\partial b} \wedge \frac{\partial}{\partial c}$. Then the complex line $L$ will be represented by $[1,0]$ in $\mathbb{CP}^1$ and $L_0$ by $[0,1]$. Extend these coordinates to the other hyperplanes $H^4$ via the rotations $A_q$ as above, so, at any $H_p^4$ we have that, for the unique $\Sigma$ containing $p$:
\begin{equation}
\label{eq:coordindparfol}
T_p L = [1,0], \;\;\; T_p \Sigma \cap H_p^4 =[0,1].
\end{equation}

\medskip

\textbf{Families of parallel foliations}.
We will often need to use not only the foliation constructed, but a family of foliations. Keeping as base coordinates the coordinates that we just introduced, we can perform a similar construction. The foliation we constructed is parametrized by $q \in L$ with the property that $T_q \Sigma_q \cap H_q^4 =[0,1]\in \mathbb{CP}^1$. For $X$ in a neighbourhood of $[0,1]\in \mathbb{CP}^1$, e.g. $\{X=[Z,W] \in \mathbb{CP}^1, :|Z| \leq |W|\}$, we start from the 3-surface $\Sigma_0^X$ built as follows: the Special Legendrian spheres that we attach to the fiber should have tangent planes in the direction $X \in \mathbb{CP}^1$. Then, for any such fixed $X$, we still have a foliation of a neighbourhood of $(1,0,0)$, parametrized on $L$ and made of the 3-surfaces
\begin{equation}
\label{eq:parallelfoliations}
\Sigma_q^X := A_q (\Sigma_0^X), \;\;\; q\in L.
\end{equation}
We will refer to $\Sigma_q^X$ as to the 3-surface born at $q$ in the direction $X$. The original surfaces we built will be denoted $\Sigma^{[0,1]}$. By the $SU(3)$-invariance of $\omega$, from lemma \ref{lem:positiveintersectionplanes} we get the positiveness property for $\Sigma_q^X$:

\begin{cor}
For any $q$, $\Sigma_q^X$ has the property of intersecting positively any transversal Special Legendrian $S$, i.e. at any point $p$ where $T_p S$ is defined and transversal to $T_p \Sigma_q^X$,
\[T_p S \wedge T_p \Sigma_q^X = T_p S^5.\]
\end{cor}

For a fixed $X$, a parallel foliation $\{\Sigma_p^X\}$ (as $p$ ranges over $L_{\eps}$) gives rise in a neighbourhood of $(1,0,0)$ to a system of five real coordinates. The adjective \textit{parallel} is reminiscent of this resemblance to a cartesian system of coordinates in the chosen neighbourhood. There are several reasons why we produced parallel foliations keeping freedom on the "direction" $X$; they will be clear later on.

\medskip

\textbf{Families of polar foliations}.
So far we have been dealing with "parallel" foliations. We turn now to "polar" foliations\footnote{The term \textit{polar} is used as reminiscent of the standard polar coordinates in the plane.}. 

Notice that, a point in $L$ being fixed, say $0$, we have that, as $X$ runs over a neighbourhood of $[0,1]\in \CP^1$, the family $\{\Sigma_0^X\}$ foliates a conic neighbourhood of $\Sigma_0^{[0,1]}$. Observe that the rotations in $SO(3) \subset SU(3)$ fixing the fiber through $q\in S^5$ have for differentials exactly the rotations in $SU(2)$ on $H_q^4$. Denoting $R_{X,Y}$ the rotation whose differential sends $X$ to $Y\in H_q^4$, we have $R_{X,Y}(\Sigma_q^X)= \Sigma_q^Y$.
 
\begin{lem}
\label{lem:polarfoliations}
With the above notations, let $U$ be a small enough neighbourhood of $Y \in \mathbb{CP}^1$ and consider $\Sigma_q^Y$ for some point $q$. Let $L^Y \subset \Sigma_q^Y$ be the special legendrian sphere tangent to $Y$ at $q$. Then
\[(\cup_{X \in U} \Sigma_q^X) - \{e^{i \theta}q\}\]
is a neighbourhood of $L^Y - \{q\}$.
\end{lem} 

\begin{proof}[\bf{proof of lemma \ref{lem:polarfoliations}}]
Introduce the function $\psi: \Sigma \times U \rightarrow S^5$ sending $(p,X)$, $p \in \Sigma = \Sigma_q^Y$, $X \in U$, to the point $R_{Y,X}(p)\in \Sigma_q^X$. Observe that, in a neighbourhood of $(1,0,0)$, the differential of $\psi$ is different from zero except at the points of the Hopf fiber through $(1,0,0)$. Indeed, on this fiber, $d \psi$ restricted to the 3-space $T \Sigma$ has rank 3 and $T \Sigma \cong Y_1 \wedge Y_2 \wedge v$, with $Y_1 \wedge Y_2$ the 2-plane in $\C^2$ represented by $Y$. At any point $F$ among these, $d \psi$ is zero on the tangent space to $U$ at $Y$, since the image $\psi(F,X)$ is constantly equal to $F$ for any $X$. For any fixed point $p$ not on the fiber and for $X$ on a curve in $U$ through $Y$,  $\psi(p,X)$ is a curve transversal to $\Sigma_p^Y$, since we are moving $p$ by the rotation $R_{Y,X}$. Therefore the differential $d \psi(p, Y)$ has rank 2 when restricted to the tangent to $U$ at $Y$, while on the complementary 3-space $d \psi$ still has rank 3 by smoothness. Therefore we get the desired result.
\end{proof}
 
\begin{oss}
\label{oss:coincidesphere}
We remark here that a $3$-surface $\Sigma$ of the type just exhibited above, is foliated by Special Legendrian spheres, so the Special Legendrian structure restricted to $\Sigma$ is integrable; a Special Legendrian integral cycle contained in such a $\Sigma$ must locally be one of these spheres.
\end{oss} 
 
\begin{oss}
\label{oss:example}
With the above notations, $L_0 + L$ is a Special Legendrian cycle with isolated singularities at the points $(1,0,0)$ and $(-1,0,0)$. This example shows that our regularity result is optimal. The reader may consult \cite{Haskins} for further explicit examples of Special Legendrian surfaces.
\end{oss}
 
\section{Tools from intersection theory}

In this section we recall some basic facts about the blowing-up of the current at a point and about the Kronecker intersection index (for the related issues in geometric measure theory we refer to \cite{Giaquinta}); then we show that this index is preserved when we send a blown-up sequence to the limit. 

\medskip
Let $C$ be the Special Legendrian cycle that we are studying.
The blow-up analysis of the current $C$ around a point $x_0$ is performed as follows: consider a dilation of $C$ around $x_0$ of factor $r$ which, in normal coordinates around $x_0$, is expressed by the push-forward of $C$ under the action of the map  $\displaystyle \frac{x-x_0}{r}$:
\[C_{x_0,r}(\psi)=\left( \frac{x-x_0}{r}\right)_\ast C (\psi) = C \left( \left(\frac{x-x_0}{r}\right)^\ast \psi\right).\]
From \cite{PR} or \cite{Sim} we have the monotonicity formula\footnote{This formula is proved in \cite{PR} for semi-calibrated currents and in \cite{Sim} for currents of vanishing mean curvature; both cases apply here.} which states that, for any $x_0$, the function
\[\frac{M(C \res B_r(x_0))}{r^2}\] 
is monotonically non-increasing as $r \downarrow 0$, therefore the limit 
\[ \theta(x):= \lim_{r\to 0}\frac{M(C \res B_r(x))}{\pi r^2}\] 
exists for any point $x\in S^5$. This limit coincides (a.e.) with the multiplicity $\theta$ assigned in the definition\footnote{The multiplicity $\theta$ can be assumed to be positive by choosing the right orientation for the approximate tangent planes to the current.} of integer cycle, whence the use of the same notation. We can therefore speak of the \textit{multiplicity function $\theta$} as a (everywhere) well-defined function on $\mathcal{C}$.

\medskip

We recall the definitions of \underline{weak-convergence} and \underline{flat-convergence} for a sequence $T_n$ of currents in $\mathcal{R}_m$ to $T \in \mathcal{R}_m$. We remark, however, that the notions of weak-convergence and flat-convergence turn out to be equivalent for integral currents of equibounded mass and boundary mass (as it is in our case), see \cite{Sim} 31.2 or \cite{Giaquinta} page 516.

We say that $T_n \rightharpoonup T$ \textit{weakly} when we look at the dual pairing with $m$-forms, i.e. if $T_n(\psi) \to T(\psi)$ for any smooth and compactly supported $m$-form $\psi$.

$T_n \to T$ \textit{in the Flat-norm} if the quantity $\mathcal{F}(T-T_n):=\inf\{M(A)+M(B): T-T_n=A+\p B, A \in \mathcal{R}_m, B \in \mathcal{R}_{m+1}\}$ goes to $0$ as $n\to \infty$.

\medskip

The fact that $\displaystyle \frac{M(C \res B_r(x_0))}{r^2}$ is monotonically non-increasing as $r \downarrow 0$ gives that, for $r\leq r_0$, we are dealing with a family of currents $\{C_{x_0,r}\}$ which are boundaryless and equibounded in mass; by Federer-Fleming's compactness theorem\footnote{See \cite{Giaquinta} page 141. }, there exist a sequence $r_n \to 0$ and a rectifiable boundaryless current $C_\infty$ such that
\[C_{x_0,r_n} \to C_\infty \text { in Flat-norm.}\]
$C_\infty$ turns out to be a cone (a so called tangent cone to $C$ at $x_0$) with density at the origin the same as the density of $C$ at $x_0$ and calibrated by $\omega_{x_0}$ (see \cite{HL} section II.5); being $J_{x_0}$-holomorphic this cone must be a sum of $J_{x_0}$-holomorphic planes, so $C_\infty= \oplus_{i=1}^Q D_i$, where the $D_i$-s are (possibly coinciding) Special Legendrian disks.
An important question for regularity issues is to know whether this tangent cone is unique or not, or, in other words, if $C_\infty$ is independent of the chosen $\{r_n\}$: the answer happens to be positive in our situation. We are going to give a self-contained proof of it in the next section (theorem \ref{thm:uniqtangcone}) based on the tools from this section.

\medskip

What kind of geometric information can we draw from the existence of a tangent cone? The following lemma shows that, considering a blown-up sequence $C_{r_n, x_0}$ tending to one possible tangent cone $C_\infty$, we can fix a conic neighbourhood of $C_\infty$, as narrow as we want, and if we neglect a ball around zero of any radius $R<1$ the restrictions of $C_{r_n, x_0}$ to the annulus $B_1 \setminus B_R$ are supported in the chosen neighbourhood for $n$ large enough.

We can assume without loss of generality to perform the blow-up analysis about the point $0$.

\begin{oss}
It is a standard fact that two distinct sequences $C_{r_n, 0}$ and $C_{\rho_n, 0}$ must tend to the same tangent cone if $a \leq \frac{r_n}{\rho_n} \leq b$ for some positive numbers $a$ and $b$. See \cite{Morgan}.
\end{oss}

\begin{lem}
\label{lem:cono0}
Let $C$ be a Special Legendrian cycle with  $0 \in C$ and let $0<R<1$. With the above notations, let ${\rho_n}\to 0$ be such that $C_{\rho_n, 0} \rightharpoonup C_\infty = \oplus_{i=1}^Q D_i $. Denote by $A_R$ the annulus $\{x\in B_1, |x|\geq R\}$ and by $E_{\eps}$ the set $\{x\in B_1, dist(x, C_\infty) \leq \eps |x| \}$. Then, for any ${\eps} >0$, there is $n_0 \in \N$ large enough such that
\[spt C_{\rho_n, 0} \cap A_R \subset E_{\eps} \]
for $n\geq n_0$.
\end{lem}

\begin{proof}[\bf{proof of lemma \ref{lem:cono0}}]
Arguing by contradiction, we assume the existence of $\eps_0 >0$ such that 
\[\forall n \;\;\exists x_n \in spt C_{\rho_n, 0} \cap E_{{\eps}_0} ^c \cap A_R.\]
Recall that the sequence $C_{\rho_n |x_n|, 0}$ also converges weakly to the same tangent cone $C_\infty$ since $\displaystyle R \leq \frac{\rho_n |x_n|}{\rho_n} \leq 1$. From the monotonicity formula we have
\[M\left (C_{\rho_n |x_n|, 0}\; \res \; B_{\frac{{\eps}_0}{2}}\left(\frac{x_n}{|x_n|}\right )\right) \geq \frac{\pi{\eps}_0^2}{4}.\]
By compactness, modulo extraction of a subsequence, we can assume that $\displaystyle \frac{x_n}{|x_n|} \to x_\infty \in \partial B_1 \cap \overline{E_{{\eps}_0} ^c}.$ Then, since for $n$ large enough $B_{\frac{3{\eps}_0}{4}}(x_\infty) \supset B_{\frac{{\eps}_0}{2}}\left(\frac{x_n}{|x_n|}\right)$, we get
\[M\left (C_{\rho_n |x_n|, 0} \; \res \; B_{\frac{3{\eps}_0}{4}}(x_\infty)\right) \geq \frac{\pi{\eps}_0^2}{4}.\]
Recall that, from the semi-calibration property, we have
\[M\left (C_{\rho_n |x_n|, 0} \; \res \; B_{\frac{3{\eps}_0}{4}}(x_\infty)\right) = \left (C_{\rho_n |x_n|, 0} \; \res \; B_{\frac{3{\eps}_0}{4}}(x_\infty)\right)\left ( \frac{id}{\rho_n |x_n|}^* \omega\right );\]
besides
\[\frac{id}{\rho_n |x_n|}^* \omega \stackrel{C^\infty (B_1)}{\longrightarrow} \omega_0\]
as $n \to \infty$, where $\omega_0$ is the constant $2$-form $\omega(0)$.
Putting all together, we can write (the first equality expresses the fact that $\omega_0$ is a calibration for $C_\infty$)
\[M\left (C_\infty \; \res \; B_{\frac{3{\eps}_0}{4}}(x_\infty)\right) = \left (C_\infty \; \res \; B_{\frac{3{\eps}_0}{4}}(x_\infty)\right)\left (\omega_0 \right ) =\]\[= \lim_n \left (C_{\rho_n |x_n|, 0} \; \res \; B_{\frac{3{\eps}_0}{4}}(x_\infty)\right) (\omega_0) = \lim_n \left (C_{\rho_n |x_n|, 0} \; \res \; B_{\frac{3{\eps}_0}{4}}(x_\infty)\right)\left ( \frac{id}{\rho_n |x_n|}^* \omega\right ) =\]

\begin{equation}
\label{eq:computationcono0}
= \lim_n M\left (C_{\rho_n |x_n|, 0} \; \res \; B_{\frac{3{\eps}_0}{4}}(x_\infty)\right) \geq \frac{\pi{\eps}_0^2}{4},
\end{equation}

which contradicts the fact that $\displaystyle spt C_\infty \cap B_{\frac{3{\eps}_0}{4}}(x_\infty) = \emptyset$.
\end{proof}

\medskip
We need some more tools from intersection theory. For the theory of intersection and of the Kronecker index we refer to \cite{Giaquinta}, chap.5, sect. 3.4. We recall the definition of the index relevant to our case. 

Let $f:\mathbb{R}^5 \times \mathbb{R}^5 \to \mathbb{R}^5$ be the function $f(x,y)=x-y$. The \textit{Kronecker intersection index} $k(S,T)$ for two currents of complementary dimensions $S \in \mathcal{R}_k(\R^5), T \in \mathcal{R}_{5-k}(\R^5)$ is defined under the following conditions:
\begin{equation}
\label{eq:noncrossingboundaries}
spt S \cap spt(\p T)=\emptyset \text{ and } spt T \cap spt(\p S)=\emptyset,
\end{equation}
which imply
\[0 \notin f(spt(\partial(S \times T))).\] 
Then there is an $\eps>0$ such that $B_{\eps}(0) \cap f(spt(\partial(S \times T))) = \emptyset$. By the constancy theorem (\cite{Giaquinta} page 130) we can define the index $k(S, T)$ as the only number such that\footnote{We are using $f_*$ to denote the push-forward under $f$; in \cite{Giaquinta} the notation is $f_\sharp$. 

The brackets $\llbracket B_{\eps}(0) \rrbracket$ denote the current of integration on $B_{\eps}(0)$.}
\[f_*(S \times T) \res B_{\eps}(0)= k(S, T)\llbracket B_{\eps}(0) \rrbracket.\]
$k(S, T)$ turns out to be an integer; when $S$ and $T$ are standard submanifolds $k(S, T)$ just counts intersections with signs.

\medskip

In the following lemma we focus on a chosen sequence $C_{\rho_n, 0}$ converging to a possible cone $C_\infty = \oplus_{i=1}^Q D_i $. For notational convenience we rename the sequence $C_n$ and the limit $C$.

\begin{lem}
\label{lem:limcon}
Let $C_n \rightharpoonup C$ in $B_1$. Take $\Sigma$ to be any 3-surface such that $\Sigma \cap C \cap \partial B_1 =\emptyset$. Then, for all $n$ large enough, $k(C_n, \Sigma) =k(C, \Sigma)$, where $k$ is the Kronecker index just defined. 
\end{lem}

\begin{proof}[\bf{proof of lemma \ref{lem:limcon}}]
Define $T_n := C-C_n$. $T_n \to 0$ in the Flat-norm of $B_1$, so we can write $T_n = S_n + \partial R_n$, with $M(T_n)+M(S_n) \to 0$, where $S_n \in \mathcal{R}_2$ and $R_n \in \mathcal{R}_3$.
From the hypothesis on $\Sigma$ we can choose $\eps >0$ small enough to ensure that $\Sigma \cap E_{\eps} \cap A_R = \emptyset$, where $E_{\eps} \cap A_R = \{x\in B_1, |x|\geq R, dist(x, C) \leq \eps |x| \}$, for some suitable $0<R<1$. For all $n$ big enough, from lemma \ref{lem:cono0} , we get that $spt \; T_n \cap A_R \subset E_{\eps}$; in particular, the condition on the boundaries of $\Sigma$ and $C$ is fulfilled and the intersection index $k(T_n, \Sigma)$ is well-defined.

Denote by $\tau_a \Sigma$, as in \cite{Giaquinta}, the push-forward $(\tau_a)_*[\Sigma]$ of $\Sigma$ by the translation map $\tau_a$, where $a$ is a vector. The Kronecker index is invariant by homotopies keeping the boundaries condition, so we can assume that all the intersections we will deal with are well defined as integer 0-dim rectifiable currents: in fact, for a fixed $n$, the intersection $T_n \cap \tau_a \Sigma$ exists for a.e. $a$, and $n$ runs over a countable set.
Obviously $$k(C_n - C, \Sigma)=k(S_n, \Sigma)+ k(\partial R_n, \Sigma);$$ we are going to show that both terms on the r.h.s. are zero.

From \cite{Giaquinta} we have that ($k$ counts the points of intersection with signs)
\[k(\partial R_n, \Sigma) = (\partial R_n \cap \Sigma )(1).\] On the other hand,
\[\partial R_n \cap \Sigma = R_n \cap \partial \Sigma - \partial (R_n \cap \Sigma) = - \partial (R_n \cap \Sigma)\]
since $\partial \Sigma = 0$ in $B_1$. So
\[\partial (R_n \cap \Sigma) (1) = (R_n \cap \Sigma) (d 1)=0,\] which implies $k(\partial R_n, \Sigma)=0$.
\newline Consider now $k(S_n, \Sigma)$ and recall that $\partial S_n = \partial T_n$. We have that $spt \partial S_n \cap \Sigma = \emptyset$ and $spt  S_n \cap \partial\Sigma = \emptyset$, so $0 \notin f(spt(\partial(S_n \times \Sigma)))$ and this index is well-defined and given by
\[f_*(S_n \times \Sigma) = k(S_n, \Sigma)\llbracket B_{\eps}(0) \rrbracket, \]
where $f:\mathbb{R}^5 \times \mathbb{R}^5 \to \mathbb{R}^5$ is $f(x,y)=x-y$ and $\eps$ is such that $B_{\eps} \cap f(spt(\partial(S_n \times \Sigma))) = \emptyset$; thanks to lemma \ref{lem:cono0}, $\eps$ can be chosen independently of $n$. So, for a fixed $\eps$, we have that

\begin{equation}
\label{eq:indextozero}
f_*(S_n \times \Sigma) = k(S_n, \Sigma)\llbracket B_{\eps}(0) \rrbracket 
\end{equation}
holds for all $n$ large enough. By assumption we know that $M(S_n) \to 0$, therefore $M(S_n \times \Sigma) \to 0$ and $M(f_*(S_n \times \Sigma)) \to 0$ since $f$ is Lipschitz; but then, for $\eps$ fixed and $k \in \mathbb{N}$, the only possibility for the r.h.s. of (\ref{eq:indextozero}) to go to zero in mass-norm is that eventually $k(S_n, \Sigma) = 0$ . So we can conclude that $k(T_n, \Sigma)=0$ for all large enough $n$.
\end{proof}

\begin{oss}
\label{oss:remmult}
If $Q$ is the multiplicity at $0$ and $\Sigma = \Sigma_0$ such that $\Sigma_0$ is transversal to all $D_i$ that constitute the tangent cone $C$, then $k(C_i)=Q$ for $i$ greater than some $i_0$. By homotopy, this also holds for small translations $\tau_a\Sigma_0$ (the condition of non-intersection at the boundaries must be kept during the homotopy).
\end{oss}

\section{Uniqueness of the tangent cone - easy case of non-accumulation - Lipschitz estimate}

The uniqueness of the tangent cone at an arbitrary point of the Special Legendrian follows from the more general result proved in \cite{PR} for general semi-calibrated integral 2-cycles. In this section, using the tools developed in the previous sections, we will give a self-contained proof of this uniqueness in our situation. The section then continues with proofs in the same flavour of the two other results quoted in the title of the section.

\medskip

We shall use in the following lemma a common notation for both families of foliations $(\Sigma_p^X)_p$ ($X$ fixed) and $(\Sigma_p^X)_X$ ($p$ fixed), respectively the parallel and polar families constructed in (\ref{eq:parallelfoliations}) and in lemma \ref{lem:polarfoliations}. We shall therefore denote by $\Sigma^a$ a family of 3-surfaces foliating and open set of $\mathbb{R}^5$, parametrized by $a$ which belongs to a 2-dimensional open ball $A$. 

$C$ is our Special Legendrian current in $\mathbb{R}^5$.

\begin{lem}
\label{lem:technical}
Let $A$ be an open ball in $\mathbb{R}^2$ (or any regular open set), and consider an open set $W$ in $\mathbb{R}^5$ of the form ($0<r<R$)
\[ W= (\cup_{a \in A} \Sigma^a ) \cap (B_R - \overline{B_r}).\] 
Assume that $C \res W \neq 0$ and that $\partial (C \res \overline{W}) \subset \cup_{a \in \partial A} \Sigma^a$. 
Then $k(C \res W, \Sigma^a)\geq 1$ for any $a\in A$.
\end{lem}

\begin{proof}[\bf{Sketch of the proof of lemma \ref{lem:technical}}]
There are some technicalities that we are going to skip, we will give the following sketch.
The main remark here is that there exists $\overline {a} \in A$ s.t. $(C \res W ) \cap \Sigma^ {\overline{a}}$ exists, is transversal and non-zero. Indeed, from general intersection theory, $(C \res W ) \cap \Sigma^a$ exists and is transversal for almost any $a \in A$; now, if all these $a$-s would lead to a zero-intersection, then each tangent plane to $C \res W$ (tangent planes exist $\mathcal{H}^2$-a.e.) should be contained in one $\Sigma^a$, and this would imply that $C$ is locally contained in one $\Sigma^a$, for an $a$ in the interior of $A$ due to the assumption $C\res W \neq 0$; but the structure of $\Sigma$ is made in such a way that the Special Legendrian $C$ should coincide with one of the Special Legendrian spheres that build $\Sigma$ up; then $C \res W$ must have boundary on $\p B_r$ and $\p B_R$, which is a contradiction.
\newline Once we have the desired $\overline{a}$, we can write $k(C \res W, \Sigma^{\overline{a}})\geq 1$ thanks to lemma \ref{lem:positiveintersectionplanes}; but then by homotopy the same holds for $k(C \res W, \Sigma^a)\geq 1$ independently of $a$, since the boundaries of $\Sigma^a$ and $C\res W$ do not cross during the homotopy.
\end{proof}

\textbf{Uniqueness of the tangent cone}. We start with the following:
\begin{lem}
\label{lem:uniqtangconeI}
Take any point $x_0$ of a Special Legendrian cycle $C$ and be $Q$ its multiplicity. Then there exists a unique choice of $n$ distinct Special Legendrian disks $D_1, ... D_n$ going through $x_0$ such that any tangent cone at $x_0$ must be of the form $T_{x_0} C = \oplus_{k=1}^n N_k D_k$, for some $N_k\in \N \setminus{0}$ satisfying $\sum_{k=1}^n N_k=Q$.
\end{lem}

\begin{oss}
This result "almost" gives the uniqueness of the tangent cone. What still is missing, is the fact that the multiplicities $N_k$ are also uniquely determined. This will be achieved in theorem \ref{thm:uniqtangcone}.
\end{oss}

\begin{proof}[\bf{proof of lemma \ref{lem:uniqtangconeI}}]
Assume, without loss of generality, that the point at which we are working is $0$ and be $Q$ its multiplicity. Argue by contradiction: take two tangent cones $C_\infty ^{(1)}$ and $C_\infty ^{(2)}$ having distinct supports, and two blown-up sequences $\{C_{r_i,0}\}$ and $\{C_{\rho_i,0}\}$ converging to each of them (we drop the $0$ from the notation $C_{r,0}$):
\[C_{r_i} \rightharpoonup C_\infty ^{(1)},  \;\;\;\;  C_{\rho_i} \rightharpoonup C_\infty ^{(2)}.\]
Take a positive $\delta$ much smaller than the angular distance 
$$\displaystyle \widehat {C_\infty ^{(1)}, C_\infty ^{(2)}}= \min_{D_i \neq D_j} \widehat {D_i ^{(1)}, D_j^{(2)}}>>\delta >0$$
(the distance is given by the Fubini-Study metric in $\mathbb{CP}^1 \cong \mathbb{P}H_0^4$ and is strictly positive by the contradiction assumption). Moreover assume, without loss of generality, that the disk of $C_\infty ^{(1)}$ on which the minimum is achieved is $D_0$, the disk represented by $[1,0]\in \CP^1$. By abuse of notation we will write $D_0 \in C_\infty ^{(1)}$ to express the fact that $D_0$ is one of the disks that build up the cone $C_\infty ^{(1)}$. Choose $\rho_{i_0}$ such that
\begin{description}
	\item[(i)] for $j \geq i_0$, $\partial (C \res B_{\rho_j})$ is contained in $E_2^\delta$, the $\delta$-conic-neighbourhood of $C_\infty ^{(2)}$ (possible by lemma \ref{lem:cono0});
	\item[(ii)] $k(C_{\rho_j}, \Sigma_0^{[0,1]})= Q$ for any $j \geq i_0$. Remark that we can assume, without loss of generality, that $\Sigma_0^{[0,1]}$ is transversal to $C_\infty ^{(2)}$. By homotopy, it also holds that $k(C_{\rho_j}, \Sigma_0^X)= Q$ for any $j \geq i_0$ and any $\Sigma_0^X$ with $X\in \CP^1$ in a $\delta$-neighbourhood of $[0,1]$. Indeed, the homotopy keeps the condition of non-crossing boundaries expressed in (\ref{eq:noncrossingboundaries}).
\end{description}
Choose now $r_{i_1}<\rho_{i_0}$ such that
\begin{description}
	\item[(iii)] denoting by $E_0^\delta$ the $\delta$-conic-neighbourhood of $D_0$ and by 
	\[W_1=(B_{r_{i_1}} \setminus B_{\frac{r_{i_1}}{2}}) \cap E_0^\delta,\] we have
	\[C \res W_1 \neq 0,\]
	which will be true for $i$ large enough since $C_{r_i} \rightharpoonup C_\infty ^{(1)} \ni D_0$.
\end{description}
Take now $\rho_{i_2} <<\frac{r_{i_1}}{2}$. Denote by $W_2 =(B_{\rho_{i_0}} \setminus B_{\rho_{i_2}}) \cap E_0^\delta \supset W_1$. A.e. $\Sigma_0$ of the polar foliation of $W_2$ born at $0$ is transversal to $C$ and the intersection exists; (i) and (iii) ensure that the boundary and  interior conditions that allow us to use lemma \ref{lem:technical} are satisfied. Then
\[k(C \res B_{\rho_{i_0}}, \Sigma) = k(C \res B_{\rho_{i_1}}, \Sigma) + k(C \res (B_{\rho_{i_0}} \setminus B_{\rho_{i_1}}), \Sigma)= \]\[=k(C_{\rho_{i_1}}, \Sigma) + k(C \res W_2, \Sigma)  + k(C \res ((B_{\rho_{i_0}} \setminus B_{\rho_{i_1}}) \setminus W_2), \Sigma)= \]\[=Q + k(C \res W_2, \Sigma) + k(C \res ((B_{\rho_{i_0}} \setminus B_{\rho_{i_1}}) \setminus W_2), \Sigma) \geq Q+1,\]
from the positivity ($\geq 0$) of intersection in $(B_{\rho_{i_0}} \setminus B_{\rho_{i_1}}) \setminus W_2$ and the strict positiveness ($\geq 1$) guaranteed in $W_2$. This contradicts (ii).
\end{proof}

Now that this "almost uniqueness" of the tangent cone is established, we can improve lemma \ref{lem:cono0} as follows:

\begin{lem}
\label{lem:cono}
Let $\{D_k\}_{k=1}^n$ be the uniquely determined disks on which any tangent cone to $C$ at $0$ must be supported. Let us therefore write $T=\cup_k D_k$ for this well-determined support. Denote by $E_{\eps}$ the cone $\{x\in B_1, dist(x, T) \leq \eps |x| \}$. Then for any ${\eps} >0$ there is $\rho_{\eps}$ small enough such that for any $\rho \leq \rho_{\eps}$ 
\[spt C_{\rho, 0} \cap E_{\eps} ^c = \emptyset.\]
\end{lem}

\begin{proof}[\bf{proof of lemma \ref{lem:cono}}]
The proof is similar to the one of lemma \ref{lem:cono0}. Assume the existence of $\eps_0 >0$ and ${\rho_n}\to 0$ contradicting the claim and argue as in the proof of lemma \ref{lem:cono0}. The only modification in the proof consists in using the "almost uniqueness" of the tangent cone at $0$ (lemma \ref{lem:uniqtangconeI}) instead of the condition $\displaystyle R \leq \frac{\rho_n |x_n|}{\rho_n} \leq 1$. If $C_{\rho_n, 0}$ converges to the cone $C_\infty= \oplus_{k=1}^n N_k D_k$, then $C_{\rho_n |x_n|, 0}$ must tend to a limiting cone $\tilde{C}_\infty = \oplus_{k=1}^n \tilde{N}_k D_k$. So the computation in (\ref{eq:computationcono0}) can be performed with $\tilde{C}_\infty$ instead of $C_\infty$, still leading to a contradiction since the supports of $\tilde{C}_\infty$ and $C_\infty$ are the same.
\end{proof}

Now we can complete the proof of the uniqueness of the tangent cone:

\begin{thm}
\label{thm:uniqtangcone}
The tangent cone at any point $x_0$ of a Special Legendrian cycle $C$ is unique.
\end{thm}

\begin{proof}[\bf{proof of theorem \ref{thm:uniqtangcone}}]
With the result and the notations of lemma \ref{lem:uniqtangconeI} in mind, we only have to exclude that the multiplicities $N_k$ may depend on the chosen sequence that we blow-up. 

Choose $\eps$ small enough to ensure that different $\eps$-neighbourhoods
\[E_{\eps}^i= \{x\in B_1, dist(x, D_i) \leq \eps |x| \}, \;\; E_{\eps}^j= \{x\in B_1, dist(x, D_j) \leq \eps |x| \}\]
of different disks $D_i$ and $D_j$ do not overlap, i.e. $E_{\eps}^i \cap E_{\eps}^j = \{0\}$.

Rotate $B_1$ in order to have that the family $\Sigma_p:=\Sigma_p^{[0,1]}$ is transversal to all the disks $D_k$. Then, for $p$ in a neighbourhood $B_\delta$ of $0$ and for all small enough $r$, the index $k(C_r, \Sigma_p)$ is well-defined since lemma \ref{lem:cono} ensures the condition (\ref{eq:noncrossingboundaries}) of non-crossing-boundaries. 

The key observation is that the rescaled $C_r$ form a continuous (with respect to $r$) family of currents (with respect to the flat-topology) and they are always constrained in the $E_{\eps}$-neighbourhood given by lemma \ref{lem:cono}.
Fix $i$: the fact that the $E_{\eps}^k$ are well separated implies that, for any $p \in B_\delta$,
$$\p(C_r \res E_{\eps}^i) \cap \Sigma_p =\emptyset, \;\; (C_r \res E_{\eps}^i) \cap \p \Sigma_p =\emptyset.$$ 
Moreover, due to the mentioned continuity, as $r \to 0$ the currents $C_r \res E_{\eps}^i$ are all homotopic to each other, and these homotopies keep the condition (\ref{eq:noncrossingboundaries}) between $C_r \res E_{\eps}^i$ and $\Sigma_p\res E_{\eps}^i$. 

Therefore $k(C_r \res E_{\eps}^i, \Sigma_p)$ must stay constant as $r\to 0$, so there is a well-determined $N_i \in \N$ such that $k(C_r \res E_{\eps}^i, \Sigma_p)=N_i$. Then any limiting cone $C_\infty$ must satisfy $k(C_\infty \res E_{\eps}^i, \Sigma_p)=N_i$, with the same proof as in lemma \ref{lem:limcon}. This means that  $C_\infty \res E_{\eps}^i= N_i D_i$, so all the multiplicities $N_k$ are uniquely determined.

\end{proof}

\medskip

\textbf{Easy case of non-accumulation}.
The following result solves the "easy case" of non-accumulation of singularities of multiplicity $Q$ to a singularity $p$ of the same multiplicity: this "easy case" arises when the tangent cone at $p$ is not made up of $Q$ times the same disk. We will see how to handle the "difficult case" (tangent cone made of $Q$ times the same plane) in sections \ref{PDEaverage} and \ref{uniquecont}.

Define the set $Sing^Q$ of singularities of multiplicity (or order) $Q$ of the Special Legendrian cycle $C$:
\[Sing^Q:=\{p \in C: p \text{ is a singular point}, \;\theta(p)=Q\}=\mathcal{C}^Q \cap Sing C.\]

In the same fashion we will use the notation
\[Sing^{\leq Q}:=\{p \in C: p \text{ is a singular point }, \theta(p)\leq Q\}.\]

\begin{thm}
\label{thm:nonacchigh}
For a Special Legendrian cycle $C$, assume $0 \in Sing^Q$, $T_0 C \neq Q \llbracket D \rrbracket$, i.e. $T_0 C = \oplus_{k=1}^n N_k D_k$, where $D_k$ are distinct Special Legendrian disks and $n \geq 2$. Then $\exists r> 0$ such that
\[Sing^Q \cap B_r(0) = \{0\}.\]
\end{thm}

\begin{proof}[\bf{proof of theorem \ref{thm:nonacchigh}}]
The proof uses techniques similar to those from theorem \ref{thm:uniqtangcone}. 
By contradiction, assume $\exists \; x_n \to 0$, with $x_n \in Sing^Q$. Rename the $D_i$'s so that $D_1$ and $D_2$ realize the minimum $\gamma$ of the angular distances $\widehat{D_i, D_j}$. $\gamma >0$ since $T_0 C \neq Q \llbracket D \rrbracket$ and $\gamma \leq \frac{\pi}{2}$ since this is the maximum for the Fubini-Study metric. 
\newline Define $\rho_n = 2 |x_n|$ and blow up about $0$ using $\rho_n$ as rescaling factors. Up to a possible exchange of the roles of $D_1$ and $D_2$ and up to a subsequence, we can assume $\frac{x_n}{2|x_n|} \to p \in D_2 \cap \partial B_{1/2}$. Rotate $B_1$ to ensure that $D_1$ and $D_2$ are contained in the $\frac{3 \pi}{4}$-cone around $D_0 \cong [1,0]$ and that $\Sigma_p^{[0,1]}$ is transversal to the disks $\{D_j\}_{j=1}^n$. Take $\alpha << \gamma$; for all $n$ large enough, thanks to lemma \ref{lem:cono} we can ensure that 
\[C_{{\rho_n},0} \subset \cup_{i=1}^n E_i^\alpha,\]
where $E_i^\alpha$ denotes the cone $E^\alpha$ around $D_i$. Thanks to the position of $D_1$ and $D_2$, we can find a small enough ball $U \subset \mathbb{CP}^1$ centered at $[0,1]$ such that for any $X \in U$ we have that $\Sigma_p^X$ is transversal to the disks $\{D_j\}_{j=1}^n$  and that $D_j \cap \Sigma_p^X \cap \partial B_1 = \emptyset$ for $j=1,...,n$.
\newline In this situation, thanks to lemma \ref{lem:limcon}, we know that for $X \in U$ and for all large enough $n$
\[k(C_{{\rho_n},0}, \Sigma_p^X)= k(T_0 C, \Sigma_p^X)= \sum_{j=1}^n N_j k(D_j,\Sigma_p^X) \leq Q.\]
Let $V$ be a ball strictly smaller than $U$ with the same center. Take now $n_0$ large enough in such a way that:
\begin{description}
	\item[(i)] $y_n =  \frac{x_n}{2|x_n|}$ is close enough to $p$ to ensure, for $n \geq n_0$, 
	\[\forall X \in V\;\;\; \Sigma_{y_n}^X \text{ satisfies } D_j \cap \Sigma_{y_n}^X \cap \partial B_1 = \emptyset;\]
	\item[(ii)] defining $\displaystyle W_n := E_1^\alpha \cap (\cup_{X \in V} \Sigma_{y_n}^X) $,
	\[\cap_{n\geq n_0} W_n \text{ contains an open set $W$}\]
	($W$ is a neighbourhood of a "piece" of $D_1$).
\end{description}
Thanks to the convergence $C_{{\rho_n},0} \rightharpoonup T_0 C \ni D_1$, we can ensure that for all $n$ large enough
\[C_{{\rho_n},0} \res W \neq 0,\]
which trivially implies $C_{{\rho_n},0} \res W_n \neq 0$; $W_n$ is foliated by $\cup_{X \in V} \Sigma_{y_n}^X$ and $\displaystyle \partial (C_{{\rho_n},0} \res \overline{W_n}) \subset \cup_{X \in \partial V} \Sigma_{y_n}^X$ (by slicing). Then by lemma \ref{lem:technical} we have that, for all $Y \in V$, 
\[k(C_{{\rho_n},0} \res W_n, \Sigma_{y_n}^Y) \geq 1.\]
On the other hand, for $\eps$ small enough, $k(C_{{\rho_n},0} \res B_{\eps}(y_n), \Sigma_{y_n}^Y)=Q$ for the $Y$-s such that $Y\neq D_j$, by remark \ref{oss:remmult}. Since $W_n \cap B_{\eps}(y_0)=\emptyset$, we get 
\[k(C_{{\rho_n},0},\Sigma_{y_n}^Y) \geq Q+1  \]
But then 
\[k(C_{{\rho_n},0}, \Sigma_p^Y) \geq Q+1\]
by homotopy, contradicting $k(C_{{\rho_n},0}, \Sigma_p^Y) \leq Q$.
\end{proof}

\textbf{Lipschitz-type estimate}.
The following theorem still uses the same ideas and is of central importance for treating the more delicate case of a singular point $p$ having a tangent cone which is $Q$ times the same plane. We can without loss of generality assume that the plane involved is $D_0 \cong [1,0]$. The result shows the "continuous behaviour" of tangent cones at points of multiplicity $Q$ as they approach $p$.

\begin{thm}
\label{thm:lipesttan}
Let $0$ be a singular point of order $Q$ of a Special Legendrian cycle, $0 \in Sing^Q$, with $T_0 C= Q \llbracket D_0 \rrbracket$. Then $\forall \{y_n \} \to 0$ sequence of points having multiplicity $Q$, the following holds: 
\[T_{y_n} C \to Q \llbracket D_0 \rrbracket.\]
\end{thm}

\begin{oss}
The convergence in the statement can of course be understood in the Flat-sense for currents in the tangent bundle and what we are proving is:
\[\forall \eps \; \exists \delta \; \text{s.t.} \; |x-0| < \delta \text{ and } \theta(x)=Q \Rightarrow \mathcal{F}(T_x C - Q \llbracket D_0 \rrbracket) < \eps.\]

We give however a more concrete definition in terms of "angles" between the disks.

We are going to speak of "the angle between $D_0$ and $D_p$" although these disks may lie in the horizontal hyperplanes at different points. More precisely: let $D_0 \subset H^4_0$ and $D_p \subset H^4_p$ be holomorphic disks for the respective J-structures. Then we can define $\widehat{D_p, D_0}$ after identifying the two hyperplanes according to the coordinates induced by the first parallel foliation, see (\ref{eq:coordindparfol}) in section \ref{preliminaries}, and taking the distance in the Fubini-Study metric. The convergence above amounts of course to the fact that the angles between $D_0$ and the disks of $T_{x_n}$ go to $0$.

In the same fashion we will speak of $\widehat{\Sigma_p^X, D_0}$ for some $3$-surface born at $p$, meaning the angle between $X$ and $D_0$ as just explained.
\end{oss}

\begin{proof}[\bf{proof of theorem \ref{thm:lipesttan}}]
Assume, by contradiction, that there exists $\{y_n\} \to 0$ such that $T_{y_n} C \not \to Q \llbracket D_0 \rrbracket$. Take as rescaling factors $\rho_n = 2 |y_n|$ and blow up about $0$. Denote $x_n = \frac{y_n}{2 |y_n|}$ and keep denoting $\oplus_{i=1}^Q D_n^i$ the tangent disks at $x_n$. Now, up to a subsequence, for some $\alpha >0$, $\widehat{D_n^i, D_0} \geq \alpha >0$ and $x_n \in \partial B_{1/2}$ hold for all $n$. Choose $\eps << \alpha$ such that $E_0^{\eps} \cap \partial B_1$ is disjoint from any $\Sigma_p$ of the set $\{\Sigma_p \; |\; p\in E_0^{\eps} \cap \partial B_{1/2}, \widehat{\Sigma_p, D_0}\geq \frac{\alpha}{2}\}$. For a large enough $n$ 
\begin{description}
	\item[(i)] $\displaystyle C_{{\rho_n},0} \subset E_0^{\eps}$,
	\item[(ii)] $k(C_{{\rho_n},0}, \Sigma_q)=Q \;\;\; \forall \Sigma_q$ with $\widehat{\Sigma_q, D_0}\geq \frac{\alpha}{2}$ and $q \in D_0 \cap B_{3/4}(0)$.
\end{description}
For notational convenience, call $x$ the point $x_n \in C_{{\rho_n},0}$. Let $T_x C_{{\rho_n},0} = \oplus_{i=1}^n N_i D_x^i$ (the total multiplicity is $Q$) and denote by $\Sigma_x^i$ the 3-surface born at $x$ and containing $D_x^i$. From the contradiction assumption, at least for one index $l$, $\widehat{D_x^l,D_0}$ is greater than a positive number very close to $\alpha$. 
\newline
Observe now the following:
Take $\beta << \min_{i\neq j}\{\alpha, \widehat{D_x^i, D_x^j}\}$. Consider the cone $E_{x,l}^\beta$ around $\Sigma_x^l$. It is not possible that for all $\Sigma_x^X$ foliating $E_{x,l}^\beta$ except $\Sigma_x^l$ the intersection with $C_{{\rho_n},0} \res E_{x,l}^\beta$ is empty, otherwise $C_{{\rho_n},0}\res E_{x,l}^\beta  \subset \Sigma_x^l$. This would imply that the boundaryless current $C_{{\rho_n},0}\res E_{x,l}^\beta$ must escape the barrier $E_0^{\eps}$ (by remark \ref{oss:coincidesphere}, $C_{{\rho_n},0}$ would have to coincide with the Special Legendrian 2-sphere tangent to $D_x^l$), which contradicts lemma \ref{lem:cono}.
\newline
So take\footnote{$p$ and $P$ depend on $n$, we are not explicitly writing this dependence in order not to make the notation too heavy.} $p \in spt C_{{\rho_n},0} \cap E_{x,l}^\beta$, $p \notin \Sigma_x^l$. Let $\Sigma_x^P = \Sigma_{x,p}$ be the 3-surface born at $x$ going through $p$; surely $\widehat{P,D_0} \geq \frac{3 \alpha}{4}$. We are going to show now that, up to tilting $\Sigma_x^P$ a bit, we can assume that it is transversal to $C$ and the intersection is still non-zero. 
\newline Take $\displaystyle \delta_n << \widehat{D_x^l, \Sigma_{x,p}}$ and $r_n << dist(x,p)$ in such a way that 
\begin{description}
	\item[(iii)]$k(C_{{\rho_n},0} \res B_{r_n} (x), \Sigma_{x,p}) = Q \;\;\;\text{ ($\Sigma_{x,p}$ is transversal to $T_x C_{{\rho_n},0}$}),$
	 \item[(iv)]$C_{{\rho_n},0} \res B_{r_n} (x) \subset E_x^{\delta_n}$ (the $\delta_n$-conic neighbourhood of $T_x C$).
\end{description}
By homotopy (see the remark following lemma \ref{lem:limcon}) 
\[k(C_{{\rho_n},0} \res B_{r_n}(x), \Sigma_x^Y) = Q\] 
for all $Y$ in a small ball around $P$ in $\mathbb{CP}^1$. The ball should be chosen small enough so that $\Sigma_x^Y$ stays away from $E_x^{\delta_n}$ and $\widehat{Y,D_0} \geq \frac{ \alpha}{2}$. We are going to apply lemma  \ref{lem:technical}: 
\[W = (\cup_Y \Sigma_x^Y) \cap (B_1 \setminus B_{r_n} (x))\]
is a foliated neighbourhood of $p$ and there is no boundary on $\overline{W} \cap \partial B_1$ or $\overline{W} \cap \partial B_{r_n} (x)$ thanks to (i) and (iv). So, for the $Y$-s in the chosen ball, if $r_n$ is small enough (recall remark \ref{oss:remmult}), then
\[k(C_{{\rho_n},0},\Sigma_x^Y )= k(C_{{\rho_n},0} \res B_{r_n}(x), \Sigma_x^Y) + k(C_{{\rho_n},0} \res (B_1 \setminus B_{r_n} (x)), \Sigma_x^Y) \geq Q + 1.\]
But, by homotopy, $k(C_{{\rho_n},0},\Sigma_x^Y )= k(C_{{\rho_n},0},\Sigma_w^Y)$ for some $w \in D_0 \cap B_{3/4}(0)$. So we have contradicted (ii) (we have to choose $Y$ so that $\Sigma_x^Y$ is transversal to $T_xC$, but this is not restrictive in the argument since it only excludes finitely many $Y$-s).
\end{proof}

The result just proved will be restated as a Lipschitz-type estimate (for the multi-valued graph describing the current) in corollary \ref{cor:lipest}.

\section{First part of the proof of theorem \ref{thm:main}: \newline coordinates, PDEs and average}
\label{PDEaverage}

Having established the previous results, in this section we start the proof of the regularity theorem \ref{thm:main}, which will go on in the next sections.

The proof proceeds by induction. By the monotonicity formula, the multiplicity function is upper semi-continuous on the Special Legendrian $C$, therefore the set of points with multiplicity $\geq N$, for $N \in \mathbb{N}$ is closed in $C$. Then, to achieve our result, a singular point $q$ with multiplicity $Q$ being given, we only need to show that singular points of multiplicity $\leq Q$ cannot accumulate to $q$. The idea is hence to prove this result by induction on the multiplicity $Q$: at each inductive step, we will assume that we are working in a neighbourhood where $Q$ is the maximal multiplicity.

\textbf{Basis of induction : Q=1} We are in an open set where all points of the Special Legendrian $C$ have multiplicity $1$. Since $C$ is minimal ($H=0$) and boundaryless, we can deduce the smoothness in this set straight from Allard's theorem, see \cite{Sim}. We can however provide a self-contained argument here: from theorem \ref{thm:lipesttan} we know that the tangent planes are continuous, therefore $C$ is a $C^1$ current. A classical bootstrapping argument then leads to $C^\infty$ regularity.

\textbf{Assumptions for the inductive step : Q-1 $\Rightarrow$ Q}. We are in an open ball $B$, where $Sing^Q$ is a closed set (that could a priori have positive $\mathcal{H}^2$ - measure) and $C \setminus Sing^Q$ is smooth except at the points $Sing^{\leq Q-1}$, which are isolated in the open set $C\setminus Sing^Q$. 
\newline We are going to divide the proof of the inductive step into two parts: 
\begin{description}
	\item[${\partone}$]: $Sing^Q$ is made of isolated points in B, i.e. there is no possibility of accumulation of singularities of multiplicity $Q$ to another singularity $p$ of the same multiplicity;
	\item[${\parttwo}$]: singularities of multiplicity $\leq Q-1$ cannot accumulate on a singularity of multiplicity $Q$.
\end{description}

\medskip

The proof of ${\partone}$ begins here and goes on throughout this section and the next one. For $p \in Sing^Q$ having a tangent cone which is not $Q$ times the same disk, the result is just theorem \ref{thm:nonacchigh}. 
\newline Therefore we only need to prove ${\partone}$ if the tangent cone at $p$ is $Q \llbracket D \rrbracket$. 

\medskip

\textbf{Coordinates}. We are now going to choose appropriate coordinates to guarantee later a $W^{1,2}$-type estimate. In order to do that, we will need the result contained in the next lemma. First observe the following:

\begin{oss}
\label{oss:remcmplxlines}
Due to the construction of $\Sigma$, given any 3-surface $\Sigma_q^X$ and for any point $p \in \Sigma_q^X$, then $T_p(\Sigma_q^X) \cap H_p^4$ is a complex line in $H_p^4$. This can be seen as follows: $T_p(\Sigma_q^X) \cap H_p^4$ is a two-dimensional subspace since $\Sigma_q^X$ is transversal to $H_p^4$; moreover one of the Special Legendrian spheres foliating (and building up) $\Sigma_q^X$ must go through $p$ and it is tangent to $H_p^4$.
\end{oss}

\begin{oss}
In the construction of the 3-surfaces $\Sigma_q^X$ performed in section \ref{preliminaries}, $q$ was taken in a neighbourhood of the Special Legendrian 2-sphere $L_0$. We can parametrize this neighbourhood of $L_0$ with a complex coordinate $w$ such that the point $(1,0,0)\in L_0$ has coordinate $0$. By abuse of notation we will also write $\Sigma_w^X$ instead of $\Sigma_q^X$ when the point $q \in L_0$ has coordinate $w$.
\end{oss}

\begin{lem}
\label{lem:setcoords}
There exist open neighbourhoods $V,U$ of $[0,1]$ in $\mathbb{CP}^1$ so that we can define\footnote{By $B_1^5$ we mean the 5-dimensional ball of radius $1$. Analogously for $B_2^2$, which we implicitly identify with the disk in $\C$ of radius $2$.} the function:
\newline $d: B_1^5 \times V \to B_2^2 \times U$,  given by $d(p,Y)=(w,X)$ s.t. $\Sigma_w^X$ contains $p$ and $Y \subset T_p \Sigma_w^X$. Moreover, $d$ is of class $C^1$.
\newline In other words, for any point $p \in B_1^5$ and any almost vertical direction $Y$ there exist a unique point $w \in L_0$ and direction $X$ such that $\Sigma_w^X$ goes through $p$ with direction $Y$. Moreover this correspondence is $C^1$.
\end{lem}

\begin{proof}[\bf{proof of lemma \ref{lem:setcoords}}]
Take the following neighbourhood $U$ of $[0,1]$ in $\mathbb{CP}^1$, $U=\{[Z;W] \in \mathbb{CP}^1 :\; |W|> 2 |Z|\}$. Define the function
\[\tilde{w}:B_1^5 \times U \to B_2^2\]
where $\tilde{w}=\tilde{w}(p,X)$ is the point in $B_2^2 \cong L_0 \cap B_2^5$ such that $p \in \Sigma_{\tilde{w}}^X$ ($\tilde{w}$ is uniquely defined since $\{\Sigma_w^X\}$ foliates $B_1^5$ as the base point runs over $L_0$). $\tilde{w}$ is a smooth function. 
\newline Recall remark \ref{oss:remcmplxlines}. Denote by $\tilde{X}=\tilde{X}(p,X) \in \mathbb{CP}^1$ the complex line\footnote{Recall that $\CP^1 \cong \mathbb{P}H_p^4$.} in $H_p^4$ such that $\tilde{X}$, as a 2-dimensional plane, is contained in the tangent to $\Sigma_{\tilde{w}}^X$ at $p$. $\tilde{X}(p,X)$ is a smooth perturbation of $X$, since the contact structure in $B_1^5$ is a smooth perturbation of the integrable structure $\mathbb{C}^2 \times \mathbb{R}$.
Consider
\[D: B_1^5 \times U \times B_2^2 \times U \to \mathbb{C} \times \mathbb{CP}^1\]
\[D:(p,Y,w,X) \to (w-\tilde{w}(p,X), \tilde{X}(p,X) - Y).\]
The function $D$ is $C^1$ and we can compute its $(w,X)$-differential
\[\frac{\partial D}{\partial(w,X)} = 
\left( \begin{array}{cc}
1 \;\;\; & \frac{\partial \tilde{w}}{\partial X} \\
0 \;\;\;& \frac{\partial \tilde{X}}{\partial X}\approx 1 \end{array} \right)\]
and its determinant is non-zero, therefore, by the implicit function theorem, the set $\{D=0\}$ can be described as a graph over $B_1^5 \times V$
\[(p,Y, d(p,Y))\]
for some $d \in C^1$ and $V \subset U$. The condition $D(p,Y,w,X)=0$ expresses the fact that $\Sigma_w^X$ goes through $p$ with direction $Y$, thus $d$ satisfies the statement of lemma \ref{lem:setcoords}.
\end{proof}

Before starting the proof of non-accumulation of singularities of order $Q$ to a singular point $x_0$ having tangent cone of the form $Q \llbracket D \rrbracket$, we are going to set coordinates so that the current and the leaves of the chosen foliation $\Sigma^X$ have only isolated and at most countably many points of non-transversality. 

Recall that a parallel foliation $\{\Sigma_p^X\}$ for $X$ fixed, of the type constructed in section \ref{preliminaries}, locally induces a system of 5 real coordinates around $x_0=(1,0,0)$, the first two, $(s,t)$, lying in the space of parameters $L_0$ (the chosen Special Legendrian 2-sphere) and the remaining three in $\Sigma$, see lemma \ref{lem:firstfoliation} and the discussion about families of parallel foliations. We can also think of having a complex coordinate on $L_0 \cap B^5_2 \cong B_2^2$ rather than two real ones. This means, for instance, that in this coordinates, if $q \in L_0$ has coordinate $z_0 \in \C$, the leaf $\Sigma_q^X$ is described by $\{(z_0, b,c,a)\}$, as $(b,c,a)$ describes to $B_2^3\subset \R^3$. In the same vein, $L_0$ is described by $\{(z,0,0,0)\}$ or by $\{(s,t,0,0,0)\}$, where we used respectively a complex and two real coordinates for $L_0 \cap B^5_2 \cong B_2^2$.
\newline In the coordinates so induced by $\{\Sigma_p^X\}$, introduce the projection map $\pi:B_2^2 \times B_2^3 \to B_2^2$ sending $(z, b,c,a)$ to $z$.

Now we want to choose a privileged direction $X$ to ensure the transversality announced above. Recall that we are working in a neighbourhood of $x_0$ where the multiplicity is everywhere $\leq Q$. Start with coordinates set in such a way that $x_0 = 0$, $D= D_0\cong[1,0]$ and the foliation we are using is given by $\{\Sigma_p^{[0,1]}\}$, and assume that we have blown up enough in order to ensure that $C_{\rho, 0} \subset E^\delta$ for some small $\delta $ (lemma \ref{lem:cono}) and that $T_y C_{\rho, 0}$ makes an angle smaller than $\delta$ for any $y \in Sing^Q$ (theorem \ref{thm:lipesttan}). 
\newline Recall lemma \ref{lem:setcoords} and let $S$ be the smooth part of the current $C_{\rho, 0}$ where the tangent planes are in $V$. Define the following function $\psi: S \to \mathbb{CP}^1$
\[\psi(p):= d(p, T_p S).\]
The tangent on $S$ is a smooth function, thus, by composition, $\psi$ is also smooth. Therefore we can find a regular value $X$ for $\psi$ as close as we want to $[0,1]$. \underline{We choose then the coordinates induced by this $\Sigma ^X$}, which we will denote by $\{(z,b,c,a)\}$ or by $\{(s,t,b,c,a)\}$, where $z=s+it$. They have the property that the leaves $\Sigma_z ^X$ are tangent to the smooth part of the current only at isolated points $\{t_i\}_{i=1}^\infty$ (they can possibly accumulate on the singular set). As for the singular set, the points of multiplicity up to $Q-1$ are also isolated singularities by inductive assumption, so we can assume that there is transversality there up to picking a new $X$, again among the regular values (only a countable set of $X$ must be avoided). On the set $Sing^Q$ the tangent cone makes a small angle with the horizontal, thanks to the Lipschitz estimate from theorem \ref{thm:lipesttan}.

\medskip

\textbf{Multi-valued graph}. Denote by $\pi$ the projection onto $D_0 \cong \{(z,0,0)\}$. We can now say that, by intersection theory, except on the countable set $\{\pi(t_i)\}$, the leaves intersect $C$ transversally and positively; as explained in remark \ref{oss:remmult}, for some $R<1$, $\Sigma_z ^X$ intersect the current at exactly $Q$ points (counted with multiplicities) for a.e. $|z| \leq R $. We have thus defined a $Q$-valued  function
\[\{b_i, c_i, \alpha_i\}_{i=1}^Q (z): D_R \to \mathbb{R}^3, \text{ or}\]
\[\{\varphi_i, \alpha_i\}_{i=1}^Q (z): D_R \to \mathbb{C} \times \mathbb{R},\]
with $D_R = \{(z,0,0), |z| \leq R\}$, $\varphi_j=b_j + i c_j$. Equivalently, we have a function from $D_R$ into the Q-th symmetric product 
\[\mathcal{S}^Q(\mathbb{C} \times \mathbb{R})= \frac{(\mathbb{C} \times \mathbb{R})^Q}{ \sim} ,\]
where two $Q$-tuples are equivalent if one is a permutation of the other.
When using the notation $\{\varphi_i, \alpha_i\}_{i=1}^Q$ it should be kept in mind that the indexation is arbitrary and not global.
\newline The $Q$-valued  function just constructed is $L^\infty$ since the current was contained in a cone $E^{2 \delta}$ around $D_R$. 

\begin{oss}
Introduce the following notation: 
\[\mathcal{A}= D_R \setminus \pi(Sing^Q), \; \mathcal{B}= \mathcal{A} \setminus \pi(Sing^{\leq Q-1}),\; \displaystyle \mathcal{G}=\mathcal{B} \setminus \cup_{i=1}^\infty \{\pi(t_i)\}.\]
$\pi(Sing^Q)$ is a closed set since we are working in a neighbourhood where $Q$ is the highest multiplicity, therefore $\mathcal{A}$ is open. $\mathcal{B}=D_R \setminus \pi(Sing^{\leq Q})$ is also open since $Sing^Q$ is a closed set. $\mathcal{G}$ is open since we are taking away from the open set $\mathcal{B}$ a countable set of isolated points that can only accumulate on the complement of $\mathcal{B}$. 
\newline Observe that, locally on $\mathcal{B}$, it is possible to give a coherent global indexation of $\{\varphi_i, \alpha_i\}_{i=1}^Q$; i.e., for any point in $\mathcal{B}$ there is a small ball centered at this point on which the multifunction is made of $Q$ distinct smooth functions.
\end{oss}

\textbf{Average} Define the average of the branches $\{\varphi_i, \alpha_i\}_{i=1}^Q$ by
\[\tilde{\Psi}=(\tilde{\varphi}, \tilde{\alpha}) := \left( \frac{\sum_{i=1}^Q \varphi_i}{Q}, \frac{\sum_{i=1}^Q \alpha_i}{Q}\right ),\]
which is a single-valued $L^\infty$ function of $D_R$. The next steps aim to prove that this average is actually a $W^{1,2}$ function. This will be achieved with theorem \ref{thm:w12}. The strategy is as follows: 
\begin{itemize}
	\item after writing the PDEs satisfied by the branches of the $Q$-valued function at smooth points, we will estimate that the $W^{1,2}$-norm on $\mathcal{G}$ is finite and bounded by the mass of the current ;
	\item we will successively extend the estimate to $\mathcal{B}$ and $\mathcal{A}$ by using the fact that, in dimension two, the $W^{1,2}$-capacity of an isolated point is zero;
	\item eventually, thanks to theorem \ref{thm:lipesttan}, we will conclude that $(\tilde{\varphi}, \tilde{\alpha})$ is $W^{1,2}$ on the whole of $D_R$.
\end{itemize}

\textbf{PDEs} As noted above, on the open set $\mathcal{G}$ the branches $\{\varphi_i, \alpha_i\}_{i=1}^Q$ are locally smooth functions. We restrict ourselves to a small ball $\Delta \subset \mathcal{G}$ on which they can be globally indexed and we are going to write the PDEs satisfied by these $Q$ functions coming from the fact that these (smooth) pieces are calibrated by $\omega$. Notice that also the derivatives of the $Q$ branches are well-defined functions. 
We are using coordinates $(z,\zeta,a) = (s,t,b,c,a)$, where $z=s+it,\zeta=b+ic$ are complex and the others real. Recall that $\Sigma$ were built so that the coordinate vectors $\frac{\partial}{\partial b}$ and $\frac{\partial}{\partial c}$ are always tangent to the 4-planes $H^4$ of the horizontal distribution. Denote by $J$ the J-structure defined on these hyperplanes, 
\[J_p : H_p^4 \to H_p^4.\]
We can assume that each leaf $\Sigma_z$ is parametrized in such a way that 
\begin{equation}
\label{eq:bcj}
J \left ( \frac{\partial}{\partial b}\right )= \frac{\partial}{\partial c},\;\;\;\;\; J\left ( \frac{\partial}{\partial c}\right ) = -\frac{\partial}{\partial b}.
\end{equation}
Recall that we are assuming, without loss of generality, that $D_R$ is centered at $0=(1,0,0) \in \mathbb{C}^3$ (this can be done by rotating $S^5$ via a rotation in $SU(3)$). We also assume that $(s,t)$ are such that $\frac{\p}{\p s}$ and $\frac{\p}{\p s}$ coincide respectively with $\frac{\p}{\p x^2}$ and $\frac{\p}{\p x^3}$ in $\C^3$ at the point $0$, so $\omega(0)= dx^2 \wedge d x^3 - dy^2 \wedge d y^3$ as a form in $\mathbb{C}^3$ is $ds\wedge dt + db \wedge dc$ in the new coordinates.
Moreover,
\[\frac{\partial}{\partial b},\frac{\partial}{\partial c},\frac{\partial}{\partial a} \text{ are always orthogonal to each other,}\]
\[\frac{\partial}{\partial b},\frac{\partial}{\partial c} \text{ are also orthogonal to the unit fiber vector $v$ ($v=i \frac{\p}{\p r}$ in $\C^3$).}\]
All the other scalar products of the\footnote{Throughout the section, $K$ will always represent a constant independent of the chosen $\Delta \subset D_R$.} coordinate vectors $\frac{\partial}{\partial s}, \frac{\partial}{\partial t},\frac{\partial}{\partial b},\frac{\partial}{\partial c},\frac{\partial}{\partial a}$ at a point $p$ are bounded by $K \eps$, for an arbitrarily small $\eps$, as long as we blow-up of a factor $r$ small enough, since they are orthogonal at the point $0$ and the structure is smooth.

Analogously, since the fiber vector at $0$ is also equal to $\frac{\partial}{\partial a}$ and orthogonal to $\frac{\partial}{\partial s}, \frac{\partial}{\partial t}$, we have  
\begin{equation}
\label{eq:scprod}
-K \eps \leq \langle \frac{\partial}{\partial s}, v \rangle , \langle \frac{\partial}{\partial t}, v \rangle, \langle \frac{\partial}{\partial s}, \frac{\partial}{\partial a} \rangle, \langle \frac{\partial}{\partial t}, \frac{\partial}{\partial a} \rangle \leq K \eps.
\end{equation}

Further, with $\omega_r:=\frac{1}{r^2}((rx)^*(\omega))$, for any $l$ we have that $\|\omega_r - \omega(0)\|_{C^l(B_1)} \to 0$ as $r \to 0$. Remark that $\omega_r$ calibrates the blown-up current $C_r$.
\newline Each branch $\psv_j =(\varphi_j, \alpha_j)$ is a graph on $\Delta$ (we are going to drop the subscript $j$ since we can focus on one precise branch - with this notation $b$ and $c$ denote both the coordinates and the functions describing this graph); the parametrization of this smooth piece is 
\[\Lambda(s,t):=(s,t, b(s,t), c(s,t), \alpha(s,t))\]
with tangent vectors
\begin{equation}
\label{eq:tangvect}
\frac{\partial \Lambda}{\partial s}= \left( \begin{array}{ccccc}
1, 
0 ,
\frac{\partial b}{\partial s},
\frac{\partial c}{\partial s} ,
\frac{\partial \alpha}{\partial s} \end{array} \right) , \;\;\; 
\frac{\partial \Lambda}{\partial t}= \left( \begin{array}{ccccc}
0, 
1 ,
\frac{\partial b}{\partial t},
\frac{\partial c}{\partial t} ,
\frac{\partial \alpha}{\partial t} \end{array} \right).
\end{equation}
On each tangent space $T_p S^5$ extend $J$ to a linear map defined on the whole of $T_p S^5$
\[J: T_p S^5 \to T_p S^5,\]
by setting $\displaystyle J\left (\frac{\partial}{\partial a} \right )= \frac{\partial}{\partial a}$ (this is quite arbitrary).
Introduce the following notation for the coefficient of this map in the given basis:
\[J\left (\frac{\partial}{\partial s} \right )= \varsigma \frac{\partial}{\partial s}+ \lambda \frac{\partial}{\partial t}+\eta \frac{\partial}{\partial a} + \beta \frac{\partial}{\partial b}+\gamma \frac{\partial}{\partial c},\]
where $\varsigma, \eta, \beta, \gamma$ are small in modulus, say less than some $K \cdot r$ since they are equal to $0$ at the point $0$, while $|\lambda|$ is close to $1$. These five functions depend on the variables $(s,t,b,c,a)$, but we will not explicitly write this dependence. For the other coefficients of $J$, recall (\ref{eq:bcj}) and the extension of $J$ done above.
The condition of being a Special Legendrian expressed by proposition \ref{Prop:slc} is then given by the two relations valid at any point:
\begin{equation}
\label{eq:spleg1}
\frac{\partial \Lambda}{\partial s} \wedge \frac{\partial \Lambda}{\partial t} \subset H^4,
\end{equation}
\begin{equation}
\label{eq:spleg2}
J\left (\frac{\partial \Lambda}{\partial s} \right )= \lambda \frac{\partial \Lambda}{\partial t} + \varsigma \frac{\partial \Lambda}{\partial s}
\end{equation}
(the fact that the last two coefficients must be exactly $\lambda$ and $\varsigma$ will be clear in a moment). We explicit now (\ref{eq:spleg2}) using (\ref{eq:tangvect}):

\[J\left (\frac{\partial \Lambda}{\partial s} \right )= \varsigma \frac{\partial}{\partial s}+ \lambda \frac{\partial}{\partial t}+\eta \frac{\partial}{\partial a} + \beta \frac{\partial}{\partial b}+\gamma \frac{\partial}{\partial c} + \frac{\partial b}{\partial s}\frac{\partial}{\partial c} - \frac{\partial c}{\partial s}\frac{\partial}{\partial b} + \frac{\partial \alpha}{\partial s} \frac{\partial}{\partial a}=\] 
\begin{equation}
\label{eq:coeffeq}
=\lambda \frac{\partial \Lambda}{\partial t} + \varsigma \frac{\partial \Lambda}{\partial s}=
\end{equation} 
\[=\lambda \left (\frac{\partial}{\partial t} + \frac{\partial b}{\partial t} \frac{\partial}{\partial b}+ \frac{\partial c}{\partial t}\frac{\partial}{\partial c}+ \frac{\partial \alpha}{\partial t} \frac{\partial}{\partial a}\right ) + \varsigma \left (\frac{\partial}{\partial s} + \frac{\partial b}{\partial s} \frac{\partial}{\partial b}+ \frac{\partial c}{\partial s}\frac{\partial}{\partial c}+ \frac{\partial \alpha}{\partial s} \frac{\partial}{\partial a}\right )\]
(from comparing the coefficients of $\frac{\partial}{\partial s}$ and  $\frac{\partial}{\partial t}$ we can see why we needed  $\lambda$ and $\varsigma$ in (\ref{eq:spleg2})). Identifying the coefficients of the coordinate vectors $\frac{\partial}{\partial b}$ and  $\frac{\partial}{\partial c}$ in the first and third line of (\ref{eq:coeffeq}) leads to

\begin{equation}
\label{eq:firstattempt}
\left \{ \begin{array}{c} 
-\frac{\partial c}{\partial s}+\beta = \lambda \frac{\partial b}{\partial t} + \varsigma \frac{\partial b}{\partial s}, \\
\frac{\partial b}{\partial s}+ \gamma = \lambda \frac{\partial c}{\partial t} + \varsigma \frac{\partial c}{\partial s}. \end{array} \right .
\end{equation}

Substituting the expression for $\frac{\partial c}{\partial s}$ given by the first line of (\ref{eq:firstattempt}) into the second we get 
\[\frac{\partial b}{\partial s} = \lambda \frac{\partial c}{\partial t} + \varsigma \left ( \beta- \lambda \frac{\partial b}{\partial t}- \varsigma \frac{\partial b}{\partial s} \right ) - \gamma, \]
which implies
\begin{equation}
\label{eq:req1}
\frac{\partial b}{\partial s} = \frac{\lambda}{1 + \varsigma ^2}\left ( \frac{\partial c}{\partial t} - \varsigma \frac{\partial b}{\partial t} + \frac{\varsigma \beta - \gamma}{\lambda} \right ).
\end{equation}
Plugging this back into the first identity of (\ref{eq:firstattempt}) we get 
\begin{equation}
\label{eq:req2}
\frac{\partial c}{\partial s} = -\frac{\lambda}{1 + \varsigma ^2}\left ( \frac{\partial b}{\partial t} + \varsigma \frac{\partial c}{\partial t} - \frac{\beta + \varsigma \gamma}{\lambda} \right ).
\end{equation}

\medskip

Let us now draw some conclusions from (\ref{eq:spleg1}). We have to impose that $\frac{\partial \Lambda}{\partial s}$ and $\frac{\partial \Lambda}{\partial t}$ are always orthogonal to the vertical fiber vector $v$.
Since the first two components of $\frac{\partial \Lambda}{\partial s}$ are fixed and equal $(1,0)$ and $\frac{\partial}{\partial b},\frac{\partial}{\partial c}$ are orthogonal to $v$, (\ref{eq:spleg1}) means 
\begin{equation}
\label{eq:££}
\langle \frac{\partial}{\partial s}, v \rangle = -\frac{\partial \alpha}{\partial s}\;\; \langle \frac{\partial}{\partial a}, v \rangle.
\end{equation}

Doing the same with $\frac{\partial \Lambda}{\partial t}$ we obtain
\begin{equation}
\label{eq:£}
\langle \frac{\partial}{\partial t}, v \rangle = -\frac{\partial \alpha}{\partial t}\;\; \langle \frac{\partial}{\partial a}, v \rangle. 
\end{equation}

Since $\displaystyle \langle \frac{\partial}{\partial a}, v \rangle$ is close to $1$ (see (\ref{eq:scprod})), we get 
\begin{equation}
\label{eq:smallcomp}
\left |\frac{\partial \alpha}{\partial s}\right |, \left |\frac{\partial \alpha}{\partial t}\right | \leq K \eps.
\end{equation}

We can rewrite\footnote{Recall again that here we are dropping the subscript $j$ in $\psv_j =(b_j, c_j, \alpha_j)$, which describes a smooth piece of the multi-valued graph above $\Delta$.} equations (\ref{eq:req1}), (\ref{eq:req2}), (\ref{eq:££}) and (\ref{eq:£}) as
\begin{equation}
\label{eq:realcrsystem}
\left \{ \begin{array}{c} 
\frac{\partial b}{\partial s} = A \frac{\partial c}{\partial t} + B \frac{\partial b}{\partial t} + C\\
\frac{\partial c}{\partial s} = -A \frac{\partial b}{\partial t} + B \frac{\partial c}{\partial t} + F \\ 
\nabla \alpha = h(s,t,\Psi). \end{array} \right .
\end{equation}
Here $A,B,C,F$ are smooth real functions of $(s,t,b(s,t),c(s,t),\alpha(s,t))$ with $A(0,0,0,0,0)=1$, $B(0,0,0,0,0)=C(0,0,0,0,0)=F(0,0,0,0,0)=0$, so $A$ is close to $1$ and $B,C,F$ are less than $\eps$ in modulus\footnote{$\eps$ is a positive number which can be assumed as small as we wish: it is of order $R$, the rescaling factor that we use for the blow-up.}. The $\R^2$-valued function $h$ is Lipschitz thanks to (\ref{eq:smallcomp}).

\textbf{Complex PDE}.
We are going to rewrite equations (\ref{eq:req1}) and (\ref{eq:req2}) in complex form, so we use the complex coordinate $z=s+it$, and observe the function $\varphi(z) = b(s,t) + ic(s,t)$. The complex derivatives $\frac{\partial}{\partial z}= \frac{1}{\sqrt{2}} \left( \frac{\partial}{\partial s} -i \frac{\partial}{\partial t}\right)$ and $\frac{\partial}{\partial \overline{z}}=\frac{1}{\sqrt{2}} \left(\frac{\partial}{\partial s} +i \frac{\partial}{\partial t}\right)$ will be denoted respectively by $\partial$ and $\overline{\partial}$. Compute the first equation in (\ref{eq:realcrsystem}) plus $i$ times the second:
\[\frac{\partial \varphi}{\partial s}= (-iA + B) \frac{\partial \varphi}{\partial t} +C+iF.\]
Then 
\[\overline{\partial} \varphi = ((1-A)i + B)\frac{\partial \varphi}{\partial t} + C + iF,\]
\[\partial \varphi = (-(1+A)i + B)\frac{\partial \varphi}{\partial t} + C + iF.\]
We seek a function $\nu=\nu_1 + i \nu_2$ so that
\[(1-A)i + B = -(\nu_1 + i \nu_2)(-(1+A)i + B), \] 
which rewrites, separating imaginary and real parts:
\[\left ( \begin{array}{cc}
1+A & -B \\ 
B & 1+A \end{array} \right)
\left ( \begin{array}{c}
\nu_1 \\ 
\nu_2 \end{array} \right)
= \left ( \begin{array}{c}
1-A \\ 
-B \end{array} \right).\]
The matrix on the l.h.s. is a perturbation of $2\; Id$, and the vector on the r.h.s. has norm bounded by $\eps$, therefore we can invert the system and find that there is a unique solution for $\nu=\nu_1 + i \nu_2$ whose norm is bounded by $\eps$. Then, setting $\mu=(1+\nu)(C+iF)$ we can rewrite (\ref{eq:req1}) and (\ref{eq:req2}) as
\begin{equation}
\label{eq:ceq}
\overline{\partial} \varphi + \nu(z,\varphi, \alpha) \partial \varphi + \mu(z,\varphi, \alpha) =0,
\end{equation}
with $\nu,\mu:\mathbb{C}_z \times \mathbb{C}_\zeta \times \mathbb{R}_a \rightarrow \mathbb{C}$ smooth functions, $\nu(0)=\mu(0)=0$, $|\nu|,|\mu| \leq \eps$.

\medskip

The first two equations in (\ref{eq:realcrsystem}), or equivalently equation (\ref{eq:ceq}), are perturbations of the classical Cauchy-Riemann equations. Notice however that the coefficients depend on $s,t,b(s,t)$, $c(s,t)$ and $\alpha(s,t)$, and we need the third equation in (\ref{eq:realcrsystem}), to clarify the "$\alpha$-dependence". 

\medskip

At this stage, we can estimate the $L^2$-norm of the jacobian of $\psv$ using (\ref{eq:req1}), (\ref{eq:req2}) and (\ref{eq:smallcomp}). Recall that the functions $\varsigma, \eta, \beta, \gamma$ are in modulus smaller than $K \eps$ and $\lambda$ is close to $1$. The metrics in the base space $\Delta_{s,t}$ and in the target $\mathbb{R}^3_{b,c,a}$ are perturbation of the standard euclidean metrics (at $0$ they coincide with them), so
\[|D \psv |^2 \leq K \left ( \left |\frac{\partial b}{\partial s}\right |^2 + \left |\frac{\partial b}{\partial t}\right |^2 + \left |\frac{\partial c}{\partial s}\right |^2 + \left |\frac{\partial c}{\partial t}\right |^2 + \left |\frac{\partial \alpha}{\partial s}\right |^2 + \left |\frac{\partial \alpha}{\partial t}\right |^2 \right ) \leq \]
\begin{equation}
\label{eq:estprew12}
\leq K \left ( \left |\frac{\partial b}{\partial t}\right |^2  + \left |\frac{\partial c}{\partial t}\right |^2 + C {\eps}^2 \right ) + K {\eps}^2 \leq K \left ( 1 + \left |\frac{\partial b}{\partial t}\right |^2  + \left |\frac{\partial c}{\partial t}\right |^2  \right ).
\end{equation}

The constant $K$ obtained at the end only depends on the factor $r_0$ that we used for the blow-up and is valid for any $r \leq r_0$, moreover it is independent of the chosen $\Delta$. We can assume that $K=2$, since this constant gets closer to $1$ as $r \to 0$.

\medskip

\textbf{$W^{1,2}$ estimate.}
For $\Psi=(b,c,\alpha)$ (we are still dropping the subscript $j$ since we are locally on $\Delta$ focusing on a single smooth branch), consider $\displaystyle \left( {\psv}\right) ^* \omega_0 = \left (  1 +\frac{\partial b}{\partial s} \frac{\partial c}{\partial t} - \frac{\partial b}{\partial t}  \frac{\partial c}{\partial s} \right )ds \wedge dt$ and plug in (\ref{eq:req1}) and (\ref{eq:req2}):
\[\left( {\psv}\right) ^* \omega_0 \geq 1 + \frac{\lambda}{1 + \varsigma ^2}\left ( \frac{\partial c}{\partial t} \right )^2+ \frac{\lambda}{1 + \varsigma ^2}\left ( \frac{\partial b}{\partial t} \right )^2 - \eps \left | \frac{\partial b}{\partial t} \frac{\partial c}{\partial t} \right | - \eps \left | \frac{\partial b}{\partial t} \right | - \eps \left | \frac{\partial c}{\partial t}\right | \geq\]
\begin{equation}
\label{eq:kaeler}
\geq \frac{1}{2} \left ( 1 + \left |\frac{\partial b}{\partial t}\right |^2  + \left |\frac{\partial c}{\partial t}\right |^2\right) \geq   \frac{1}{4} \left ( 1 + |\nabla b|^2 + |\nabla c|^2 \right),
\end{equation}
where we used $\eps \left| \frac{\partial b}{\partial t} \frac{\partial c}{\partial t} \right| \leq \frac{1}{2}\left( \eps \left( \frac{\partial b}{\partial t}\right)^2 + \eps \left( \frac{\partial c}{\partial t}\right)^2\right)$ and $\eps \left|\frac{\partial c}{\partial t}\right| \leq \frac{1}{2}\left( \eps  + \eps \left( \frac{\partial c}{\partial t}\right)^2\right)$, the hypothesis on $\varsigma, \eta, \beta, \gamma, \lambda$ and (\ref{eq:estprew12}) with $K=2$ as said above.
\newline Consider now $\omega_r - \omega_0$. Write this 2-form in the canonical basis in the coordinates $s,t,b,c,a$. All the coefficients are smaller than $\eps$ in modulus, if $r$ was chosen small enough. Therefore
\[\left( {\psv}\right) ^* (\omega_r - \omega_0)\]
is a 2-form in $ds \wedge dt$ whose coefficient comes from summing products of derivatives of $\psv$. As above, we can bound this coefficient by $\eps \left ( 1 + \left |\frac{\partial b}{\partial t}\right |^2  + \left |\frac{\partial c}{\partial t}\right |^2 \right)$.
Using this fact, together with (\ref{eq:kaeler}) and the triangle inequality we have
\[\int_{\Delta} \left( {\psv}\right) ^* \omega_r \geq \left(\frac{1}{4} - \eps \right)\int_{\Delta}1 + |\nabla b|^2 + |\nabla c|^2.\]
Recalling (\ref{eq:estprew12}) we can finally write the desired estimate:
\begin{equation}
\label{eq:w12smooth}
\int_{\Delta} |D \psv |^2 \leq K \int_{\Delta} \left( {\psv}\right) ^* \omega_r \leq K \int_{\psv(\Delta)} \omega_r = K \cdot \mathcal{H}^2 (C_r \res \psv(\Delta)),
\end{equation}
with a constant $K$ independent of the chosen $\Delta$. We can therefore conclude, recalling the notations taken during the inductive assumptions,

\begin{lem}
On the set $\displaystyle \mathcal{G}= D_R \setminus Sing^{\leq Q} \setminus \cup_{i=1}^\infty \{\pi(t_i)\}$ there holds
\[\sum_{i=1}^Q \int_{\mathcal{G}} |D \varphi_i |^2 + |D \alpha_i |^2\leq K \cdot \mathcal{H}^2 (C_r) < \infty\]
and therefore the average function $\tilde{\Psi}=(\tilde{\varphi}, \tilde{\alpha})$ is $W^{1,2}(\mathcal{G})$ with norm bounded by the mass of $C_r$ (we already knew that it was $L^\infty$).
\end{lem}

The next considerations will allow us to extend this estimate for $(\tilde{\varphi}, \tilde{\alpha})$ to the set $\displaystyle \mathcal{B} = \mathcal{G} \cup_{i=1}^\infty \{\pi(t_i)\}$. One can do this in a straightforward way recalling that the capacity of a point in $\mathbb{R}^2$ is zero. Anyway we also give a direct proof.
Rename for notational convenience $q_i = \pi(t_i)$ and take $B^2_{\rho_i}(q_i) \subset \mathcal{B}$ balls centered at the $q_i$ so that $\sum_i \rho_i \leq \delta$ for $\delta$ chosen arbitrarily small. Let $\xi$ be any test-function in $C_c^\infty (\mathcal{B})$. Then
\[\left|\int_\mathcal{B} \tilde{\varphi} \frac{\partial \xi}{\partial s} \right| \leq \left|\int_{\cup_i B_{\rho_i}(q_i)} \tilde{\varphi} \frac{\partial \xi}{\partial s} \right| +  \left|\int_{\mathcal{B} \setminus \cup_i B_{\rho_i}(q_i)} \frac{\partial \tilde{\varphi}}{\partial s} \xi  \right| + \sum_i \left|\int_{\partial B_{\rho_i}(q_i)} \tilde{\varphi} \xi \langle \frac{\partial}{\partial s}, \nu \rangle \right|\]
\[\leq C \delta ^2 \| \tilde{\varphi}\|_{\infty}  \|\nabla \xi \|_{\infty} + \| \tilde{\varphi}\|_{W^{1,2}(\mathcal{G})} \| \xi \|_{L^2 (\mathcal{B})} + C \delta \|\tilde{\varphi}\|_{\infty} \| \xi\|_{\infty}.\]
Since $\delta$ was arbitrarily small,
\[\left|\int_\mathcal{B} \tilde{\varphi} \frac{\partial \xi}{\partial s} \right| \leq \|\tilde{\varphi}\|_{W^{1,2}(\mathcal{G})} \| \xi \|_{L^2 (\mathcal{B})}.\]
We can do the same for the $t$-derivative. For $ \tilde{\alpha}$ things are even easier, indeed $\tilde{\alpha}$ is Lipschitz. Therefore the average function is $W^{1,2}$ on $\mathcal{B}$ with the same norm as on $\mathcal{G}$.
We can do the same passing from $\mathcal{B}$ to $\mathcal{A}= D_R \setminus Sing^Q$: again we have to add a (countable) set of points which are isolated in $\mathcal{A}$, so the same as above applies. Eventually we have proved

\begin{lem}
On the set $\displaystyle \mathcal{A}= D_R \setminus Sing^{Q}$  the average function $\tilde{\Psi}=(\tilde{\varphi}, \tilde{\alpha})$ defines a $W^{1,2}$ map from $\mathcal{A}$ into $\C \times \R$ with norm bounded by the mass of $C_r$.
\end{lem}
 
The next step will establish the definitive result on the whole of $D_R$.

The following corollary is basically a restatement of theorem \ref{thm:lipesttan} in terms of the coordinates and of the multi-valued graph, which were introduced in this section:

\begin{cor}
\label{cor:lipest}
Let $x_0 \in Sing^Q$ and $T_{x_0} C= Q \llbracket D_0 \rrbracket$, as before. Then $\forall \eps>0 \;\; \exists r=r(\eps, x_0)$ such that
\[\forall x=(z,\zeta,a) \in \mathcal{C^Q} \text{ and } x'=(z', \zeta', a') \in spt C \cap B_r(x_0) \text{ we have the estimate}\]
\[|(\zeta,a)-(\zeta',a')|_{\mathbb{R}^3} \leq \eps |z-z'|_{\mathbb{R}^2}.\] 
\end{cor}

\begin{proof}[\bf{proof of corollary \ref{cor:lipest}}]
The estimate for the third coordinate $a$ is obvious. We need to show that
\[|(\zeta-\zeta')|_{\mathbb{R}^2} \leq \eps |z-z'|_{\mathbb{R}^2}.\]
Without loss of generality we may assume that $x_0 = 0$ and use theorem \ref{thm:lipesttan} which guarantees the continuity at $0$ of tangent cones at points in $\mathcal{C^Q}$: choose $r$ s.t. $\forall x \in B_{2r} (0)$ having multiplicity $Q$ the angular distance $(D_0, T_x C)$ is less than $\frac{\eps}{2}$; we can also guarantee that 
\begin{equation}
\label{eq:19a}
k(C \res B_{2r}(0), \Sigma_w^X)=Q
\end{equation}
for any $w \in D_0 \cap B_{3r/4} (0)$ and $Y \in \mathbb{CP}^1$ realizing $\widehat{D_0, Y} \geq \eps$. Assume by contradiction that we can find $x\in \mathcal{C^Q}$ and $y \in spt C$, with $x,y \in B_r (0)$ for which 
\[|(\zeta-\zeta')|_{\mathbb{R}^2} > \eps |z-z'|_{\mathbb{R}^2}\]
holds. Then take $\Sigma_{x,y}$: this 3-surface is transversal to the current at $x$ since $\widehat{T_x C, \Sigma_{x,y}} > \frac{\eps}{2}$ and we can tilt it a bit finding a $\Sigma_x^Y$ transversal to $C$, with $\widehat{T_x C, Y} > \frac{\eps}{2}$ and with a non-zero intersection, as already done in the proof of theorem \ref{thm:lipesttan}. Then
\[k(C \res B_{2r} (0), \Sigma_x^Y) = k(C \res B_\rho (x), \Sigma_x^Y) + k(C \res (B_{2r} (0) - B_\rho (x)), \Sigma_x^Y) \geq Q+1,\]
for some small enough $\rho << dist(x,y)$. Since $\widehat{D_0, Y} > \eps$, we can homotope $\Sigma_x^Y$ into a $\Sigma_w^Y$ for some $w \in D_0 \cap B_{3r/4} (0)$ keeping it away from $C$ on $\partial B_{2r}$, so we are contradicting the identity in (\ref{eq:19a}).
\end{proof}

\begin{thm}
\label{thm:w12}
The average function $\tilde{\Psi}: D_R \to \mathbb{R}^3$ is in $W^{1,2}(D_R)$.
\end{thm}

\begin{proof}[\bf{proof of theorem \ref{thm:w12}}]
$Sing^Q$ is a closed set (possibly with positive $\mathcal{H}^2$ measure) and on $\mathcal{F}=\pi(Sing^Q)$ (still a closed set) $\tilde{\Psi}$ coincides with the $Q$ branches $\Psi_i$. We know that the Lipschitz estimate of corollary \ref{cor:lipest} holds for any couple of points $x,y$ such that $\tilde{\Psi}(x) \in Sing^Q$. In particular, $\tilde{\Psi}|_{\mathcal{F}}$ is Lipschitz, it is therefore possible to extend it to a function $u$ defined on the whole of $D_R$ which is Lipschitz with constant $K$ equal 3 times the Lipschitz constant of $\tilde{\Psi}|_{\mathcal{F}}$ (see \cite{F} sec. 2.10.44).
Let $\delta$ be positive and arbitrarily small. Take now a smooth compactly supported function $\sigma_{\delta}$ such that
\[ \sigma_{\delta}(x)= \left \{ \begin{array}{cc}
1 & \;\; \text{ if } dist(x,\mathcal{F}) \leq \delta \\ 
0 & \;\; \text{ if } dist(x,\mathcal{F}) \geq 2 \delta  \end{array} \right .\]
and $|D \sigma_{\delta}| \leq \frac{k}{\delta}$ for some $k>0$. 
Explicitly $\sigma_{\delta}$ can be defined as follows: take a smooth bump-function $\chi$ on $[0,\infty)$, which is $1$ on $x\leq 1$ and $0$ for $x\geq 2$. Set $\chi_{r,y}(x)=\chi\left( \frac{|x-y|}{r} \right)$ for $x,y \in \C$. Define
\[\sigma_{\delta}(z)=\frac{G}{\delta^4} \int_{\{x: dist(x,\mathcal{F}) \leq \frac{3 \delta}{2}\}} \chi_{\frac{\delta}{4},z}(w)dw d\overline{w},\]
the right normalization constant $G$ depending on $\int_0^\infty \chi(t) t^3 dt$.
Introduce
\[\tilde{\Psi}_\delta := \sigma_\delta u + (1- \sigma_\delta)\tilde{\Psi} \]
and notice that, for any $\delta >0$ this function is $W^{1,2}$. Moreover, for $x \in \{dist(x,\mathcal{F})\leq 2 \delta \}$, denoting by $p \in \mathcal{F}$ the point realizing this distance, from corollary \ref{cor:lipest} and by the definition of $u$
\[|(u-\tilde{\Psi})(x)| = |u(x)-u(p)+\tilde{\Psi}(p)- \tilde{\Psi}(x)| \leq 2K |p-x| \leq K' \delta.\]
In the following, $D$ is the partial derivative with respect to either of the coordinates $s,t$; notice that, in order to control $D \tilde{\Psi}_\delta$, we need to take $D \tilde{\Psi}$ only on the set $\{dist(x,\mathcal{F}) \geq  \delta\} \subsetneq \mathcal{A}$, since elsewhere $1- \sigma_\delta = 0$, so we can freely take derivatives.
\[D \tilde{\Psi}_\delta = (D \sigma_\delta )u + \sigma_\delta Du - (D \sigma_\delta )\tilde{\Psi} + (1- \sigma_\delta)D \tilde{\Psi}=\]
\[= (D \sigma_\delta )(u - \tilde{\Psi}) + \sigma_\delta Du + (1- \sigma_\delta)D \tilde{\Psi}.\]
We can now compute
\[\|D \tilde{\Psi}_\delta \|_{L^2(D_R)}^2 \leq \int_{\{\delta \leq dist(x,\mathcal{F}) \leq 2 \delta\}} |D \sigma_\delta|^2 |u - \tilde{\Psi}|^2 + \int_{D_R} |\sigma_\delta|^2 |Du|^2 + \]
\[+\int_{\{dist(x,,\mathcal{F}) \geq  \delta\}} |1- \sigma_\delta|^2 |D \tilde{\Psi}|^2 \leq c(K,k) + \|D \tilde{\Psi} \|_{L^2(A)}^2 \leq c(K,k).\]
So the $W^{1,2}$-norm of the $\tilde{\Psi}_\delta$ are uniformly bounded as $\delta \to 0$, therefore, by compactness, we can find a sequence $\tilde{\Psi}_{\delta_n}$, $\delta_n \to 0$, which converges in $L^2$ and weakly* in $W^{1,2}$ to some $\psi \in W^{1,2}(D_R)$.
On the other hand, from the computation above,
\[|\tilde{\Psi}_\delta - \tilde{\Psi}|= |\sigma_\delta(u-\tilde{\Psi})|= \left \{ \begin{array}{ll}
0 & \;\; \text{ on } \mathcal{F} \\ 
\leq K' \delta & \;\; \text{ on } \{dist(x,\mathcal{F}) \leq 2 \delta \}-\mathcal{F} \\
0 & \;\; \text{ on } \{dist(x,\mathcal{F}) \geq 2 \delta \} \end{array} \right . ,\]
so $\tilde{\Psi}_{\delta_n}$ converge uniformly to $\tilde{\Psi}$ on $D_R$. Therefore $\mathcal{H}^2$-a.e. it holds $\psi=\tilde{\Psi}$ and theorem \ref{thm:w12} is proven. 

\end{proof}

\section{End of the proof of $\partone$: unique continuation}
\label{uniquecont}
In this section we will complete the proof of  $\partone$, the first part of the inductive step, i.e. the fact that there is no possibility of accumulation among singularities of equal multiplicity.

\medskip

\textbf{H\"older estimate}. We are going to establish the following
\begin{thm}(\textbf{H\"older estimate})
\label{thm:holderestimate}
For any small enough disk $D_R$, there exist constants $C, \delta >0$ such that, for any $r \leq R$
\begin{equation}
\label{eq:holder}
\sum_{j=1}^Q \int_{D_r}|D\varphi_j|^2 \leq C r^\delta.
\end{equation}
This easily yields
\begin{equation}
\label{eq:holderav}
\int_{D_r}|D\tilde{\Psi}|^2 \leq C r^\delta.
\end{equation}

\end{thm}

\begin{oss}
This decay implies that $\tilde{\Psi}$ is $\frac{\delta}{2}$-H\"older thanks to Morrey's embedding theorem, see \cite{Morrey} for instance.
\end{oss}

\begin{oss}
The integral in (\ref{eq:holder}) should always be understood as \[\sum_{j=1}^Q\left(\int_{D_r-\mathcal{F}}|d \varphi_j|^2 ds \; dt + \int_{\mathcal{F}} |d \fiav|^2 d \mathcal{H}^2 \right),\] where $\mathcal{F}=\pi(Sing^Q)$; recall that all branches agree with the average on $\mathcal{F}$.
\end{oss}

\begin{proof}[\bf{proof of theorem \ref{thm:holderestimate}}]
Remark that the $\alpha_i$-s are Lipschitz thanks to (\ref{eq:smallcomp}); therefore, once (\ref{eq:holder}) will be established, (\ref{eq:holderav}) will follow immediately.

We are going to analyse the behaviour of the function $y(r)=\sum_{j=1}^Q \int_{D_r}|D\varphi_j|^2$. We already showed in the previous section that, for any $r$ small enough, $y(r)$ is finite, being bounded by the mass of the current in the cylinder $Z_r= D_r \times \mathbb{R}^3 = \{|z|\leq r\}$.
Recalling that $C$ is boundaryless, 
\[(C \res Z_r)(d \zeta \wedge d \overline{\zeta})=(\partial (C \res Z_r))(\zeta \wedge d \overline{\zeta})= \langle C,|z|,r \rangle (\zeta \wedge d \overline{\zeta}).\]
Denote by $T$ the simple 2-vector describing the oriented approximate tangent plane to the rectifiable set $\mathcal{C}$; by definition
\[(C \res Z_r)(d \zeta \wedge d \overline{\zeta})= \sum_{j=1}^Q \int_{\{\varphi_j, \alpha_j\}(D_r)} \langle d \zeta \wedge d \overline{\zeta}, T \rangle d \mathcal{H}^2=\]\[=\sum_{j=1}^Q \int_{D_r-\mathcal{F}}(|\overline{\partial} \varphi_j|^2 -|\partial \varphi_j|^2) ds\; dt+ \sum_{j=1}^Q \int_{Sing^Q} \langle d \zeta \wedge d \overline{\zeta}, T \rangle  d \mathcal{H}^2=\]\[=\sum_{j=1}^Q \int_{D_r-\mathcal{F}}2(|\overline{\partial} \varphi_j|^2 -|d \varphi_j|^2) ds \;dt+ \sum_{j=1}^Q \int_{Sing^Q} \langle d \zeta \wedge d \overline{\zeta}, T \rangle d \mathcal{H}^2\]
Recalling that the average $\fiav$ is $W^{1,2}$ and that the tangent plane at points in $Sing^Q$ is $Q$ times the tangent to the average, we can rewrite this last term as
\[\sum_{j=1}^Q \int_{D_r-\mathcal{F}}2\left(|\overline{\partial} \varphi_j|^2 -|d \varphi_j|^2\right)ds \;dt + Q \int_{\mathcal{F}} 2\left(|\overline{\partial} \fiav|^2 -|d \fiav|^2\right) d \mathcal{H}^2.\]
So we have 
\[\sum_{j=1}^Q \int_{D_r-\mathcal{F}}|d \varphi_j|^2 ds \;dt+ Q\int_{\mathcal{F}} |d \fiav|^2 d \mathcal{H}^2=\]
\begin{equation}
\label{eq:holderi}
= \sum_{j=1}^Q \int_{D_r-\mathcal{F}}2|\overline{\partial} \varphi_j|^2 ds \;dt + Q\int_{\mathcal{F}} 2|\overline{\partial} \fiav|^2 d \mathcal{H}^2+ \langle C,|z|,r \rangle (\zeta \wedge d \overline{\zeta}).
\end{equation}
Now consider (\ref{eq:ceq}), which is satisfied by the smooth parts of $\{\varphi_j\}_{j=1}^Q$, i.e. on $D_R \setminus \mathcal{F}$ minus countably many points. This gives
\begin{equation}
\label{eq:holderii}
\sum_{j=1}^Q \int_{D_r-\mathcal{F}}|\overline{\partial} \varphi_j|^2 ds \;dt\leq C_1 {\eps}^2\sum_{j=1}^Q \int_{D_r-\mathcal{F}}|d \varphi_j|^2 ds \;dt+ C_2 r^2.
\end{equation}
Putting (\ref{eq:holderi}) and (\ref{eq:holderii}) together,
\[y(r)=\sum_{j=1}^Q\left(\int_{D_r-\mathcal{F}}|d \varphi_j|^2 ds \;dt+ \int_{\mathcal{F}} |d \fiav|^2 d \mathcal{H}^2 \right) \leq\]\[\leq 3Q \int_{\mathcal{F}} |d \fiav|^2 d \mathcal{H}^2 + K_1\langle C,|z|,r \rangle (\zeta \wedge d \overline{\zeta})+ C_3 r^2.\]
By corollary \ref{cor:lipest}, $|d \fiav|$ is bounded by a small constant on $\mathcal{F}$, so
\begin{equation}
\label{eq:holderiii}
y(r) \leq  K_1 \langle C,|z|,r \rangle (\zeta \wedge d \overline{\zeta})+ K_2 r^2.
\end{equation}
The slice of the current with $|z|=r$ exists as a rectifiable 1-current for a.e. $r$, as explained in lemma 1 of \cite{Giaquinta}, page.152. On the set $D_R \setminus \mathcal{F}$, the multigraph is smooth except at a countable set of isolated points. For all but countably many choices of $r$, $\partial D_r$ will avoid this set. Also, for a.e. $r$, the current $\langle C,|z|,r \rangle$ is described by the same multigraph $\{\varphi_j\}$. This multigraph, being one-dimensional, can be actually described as a superposition of honest $W^{1,2}$ functions as follows:
\begin{description}
	\item[(i)] $\partial D_r \cap \mathcal{F} = \emptyset$: for such a $r$, $\{\varphi_j\}$ is smooth on $\partial D_r$, then, starting from any point in the multigraph, we can follow the loop and we will eventually come back to the same point after a certain number $n$ of laps, $n_1\leq Q$. Then we can define the function $g_1$ to be equal $\varphi_j$ on an interval $I_1$ of length $2 \pi n_1 r$, and $g_1$ has the same value at the endpoints of $I_1$. Then do the same, starting from a point that was not covered yet by $g_1$. This procedure leads to the construction of $K$ smooth functions $g_k$, $K \leq Q$. By \cite{EG}, page 164, $g_k$ are $W^{1,2}$ for a.e. $r$, since it is the restriction of a $W^{1,2}$ function to a line. 
	\item[(ii)] $\partial D_r \cap \mathcal{F} \neq \emptyset$: in this case the set $\partial D_r - \mathcal{F}$, being open in $\partial D_r$, must be an at most countable union of open intervals $\displaystyle \cup_i (a_i, b_i)$. Then $\partial D_r \cap \mathcal{F} = \displaystyle \cup_i [b_i, a_{i+1}]$. On each $(a_i, b_i)$ we can give a coherent labelling to the $\{\varphi_j\}$, while on the $[b_i, a_{i+1}]$ all the branches agree. Then we can write the multigraph as a superposition of $Q$ functions $g_i$. Each $g_i$ is $W^{1,2}$: in fact, on each $(a_i, b_i)$ we can use the result from \cite{EG} again, and therefore for a.e. $r$, $g_i|_{\partial D_r -\mathcal{F} }$ is $W^{1,2}$. Then we can get that $g_i \in W^{1,2}(\partial D_r)$ by the same argument that we used to prove theorem \ref{thm:w12} by means of the Lipschitz property from theorem \ref{thm:lipesttan} which holds on $\partial D_r \cap \mathcal{F}$.
\end{description}
Then, using H\"older's and Poincar\'e's inequalities,
\[\partial (C \res Z_r))(\zeta \wedge d \overline{\zeta})= \langle C,|z|,r \rangle (\zeta \wedge d \overline{\zeta})= \sum_{j=1}^Q \left(\int_{\partial D_r - \mathcal{F}} \varphi_j d \overline{\varphi_j} + \int_{\partial D_r \cap \mathcal{F}} \fiav d \overline{\fiav} \right)=\]
\[= \sum_k \int_{I_k} g_k d \overline{g_k}= \sum_k \int_{I_k} (g_k- \lambda_k) d \overline{g_k} \leq \sum_k \left(\int_{I_k} |g_k-\lambda_k|^2\right)^{\frac{1}{2}} \left(\int_{I_k} |d g_k|^2\right)^{\frac{1}{2}} \leq \]
\begin{equation}
\label{eq:holderiv}
\leq \sum_k K n_k r \left(\int_{I_k} |d g_k|^2\right) \leq K Q r \sum_{j=1}^Q \left(\int_{\partial D_r - \mathcal{F}}|d \varphi_j|^2 + \int_{\partial D_r \cap \mathcal{F}}|d \fiav|^2 \right).
\end{equation}
The function $y(r)$ is weakly increasing in $r$ and absolutely continuous, being an integral; therefore it is a.e. differentiable and, thanks to (\ref{eq:holderiii}) and (\ref{eq:holderiv}), satisfies at a.e. $r$ (we can assume $k>1$)
\[y(r)\leq k r y'(r) + c y^2.\]
By setting $\upsilon(r)=y(r)-\frac{c}{1-2k} r^2$, we turn the equation into
\[\upsilon(r)\leq k r \upsilon'(r).\]
This yields 
\[\upsilon(\rho) \leq C \rho^{\frac{1}{k}}\]
and then, adding $\frac{c}{1-2k} r^2$, we get the desired estimate for $y(r)$:
\[y(r) \leq C r^\delta\]
for some $\delta = \frac{1}{k} >0$.
\end{proof}

\textbf{Unique continuation argument:} this will conclude the proof of $\sharp_1$ and is inspired to the techniques used in \cite{Taubes}, and before by Aronszajn in \cite{Aro}.
For this section we are going to describe our current by a multigraph $D_R \rightarrow \mathbb{C}^2$, by setting the fourth (real) coordinate equal $0$. So we have a multigraph $\{\varphi_j(z), \alpha_j(z)\}_{j=1}^Q$, with $\alpha$ purely real. The average $\fiav(z)$, is a $W^{1,2}$, holder (and bounded) function, $\aav(z)$ is Lipschitz.

\begin{lem}
\label{lem:lemmataubes}
There exists a constant $K$ such that, if $R$ is small enough, there exists a $W^{1,2}$ and $C^{1,\delta}$ solution $w(z):D_R \rightarrow \mathbb{C}$ to the equation
\begin{equation}
\label{eq:wCRiem}
\overline{\partial}w+\nu(\fiav,\aav) \partial w =0
\end{equation}
which is a perturbation of the identity, precisely it satisfies
\[|w(z)-z| \leq KR|z|.\]
\end{lem}

\begin{proof}[\bf{proof of lemma \ref{lem:lemmataubes}}]
For a function $u$ defined on the whole of $\mathbb{C}$, we seek $w$ of the form $w=\chi_R (1+ u(z))z$, where $\chi_R$ is a radial, smooth cut-off function equal to $1$ on $D_R$ and $0$ on the complement of $D_{2R}$.
The requests on $w$ can be translated as follows
\[\bar{\partial} u+ \chi_R \nu(\fiav,\aav) \partial u + \chi_R \frac{s(\fiav,\aav)}{z}(1+u)=0,\]
\[|u| \leq KR.\]
It is very important at this stage to observe that $\frac{\nu(\fiav,\aav)}{z}$ is an $L^\infty$ function thanks to the Lipschitz estimate of corollary \ref{cor:lipest}, although it need not be continuous; so there is some constant $K$ (independent of $R$) such that $\frac{\nu(\fiav,\aav)}{z} \leq K$ (still from corollary \ref{cor:lipest} we actually know that this constant goes to $0$ as $R$ goes to $0$). The solution $u$ will be found by a fixed point method.
\newline Consider the space $H=\{f \in W^{1,2}(\mathbb{C})\text{ such that }Df\in L^{2,\lambda}\}$, for some $\lambda >0$ to be chosen later. By a result due to Morrey, these functions are $\frac{\lambda}{2}$-holder; they also decay at infinity, therefore they are bounded. $H$ is a Banach space with the norm whose square is
\[\|f\|_H^2= \|f\|_{L^\infty}^2+\|Df\|_{L^2}^2+ \|Df\|_{L^{2,\lambda}}^2=\]\[= \sup_{\mathbb{C}}|f|^2+ \int_{\mathbb{C}}|Df|^2 + \sup_{x_0 \in \mathbb{C},\rho>0} \frac{1}{\rho^\lambda}\int_{B_{\rho}(x_0)}|Df|^2.\]
Define the functional $\mathcal{P}$ on $H$ that sends $f$ to $\mathcal{P}(f)$ 
\[\mathcal{P}(f)(z)= \frac{1}{2\pi i}\int_{\mathbb{C}}\frac{\chi_R \nu(\fiav,\aav) \partial f + \chi_R \frac{\nu(\fiav,\aav)}{\xi}(1+f)}{\xi - z} d\xi d\bar{\xi}\]
(all the functions in the integral are functions of $\xi$). For any fixed $z$, the integral is finite: this can be seen as follows, by breaking it up as a series of integrals over annuli $A_n(z)$ centered at $z$ with outer and inner radii respectively $\frac{R}{2^n}$ and $\frac{R}{2^{n+2}}$ (all the constants we are calling $K$ are independent of $R$);
\[\sum_{n} \frac{2^{n+2}}{R}\int_{A_n(z)}|\chi_R \nu(\fiav,\aav) \partial f| + \left|\chi_R \frac{\nu(\fiav,\aav)}{\xi}\right|+ \left|\chi_R \frac{\nu(\fiav,\aav)}{\xi}f\right| \leq\]

\be
\label{eq:primopezzo}
\leq KR \sum_{n} \frac{2^n}{R}\left(\int_{A_n(z)}|\chi_R|\right)^{\frac{1}{2}} \left(\int_{A_n(z)}|\partial f|^2\right)^{\frac{1}{2}} +\sum_{n} \frac{KR}{2^n} + \sum_{n} \frac{KR\|f\|_{L^\infty}}{2^n},
\ee 
where we used $||\frac{\nu(\fiav,\aav)}{z}||_{L^\infty(D_R)} \leq K$; thanks to the finiteness of $\|Df\|_{L^{2,\lambda}}^2$  we can bound the first term in the following way:
\[\sum_n 2^n \left(\int_{A_n(z)}|\chi_R|\right)^{\frac{1}{2}} 
\left(\int_{A_n(z)}|\partial f|^2\right)^{\frac{1}{2}}\leq R \sum_n \left(\int_{A_n(z)}|\partial f|^2\right)^{\frac{1}{2}} =\]
\begin{equation}
\label{eq:29a}
= R\sum_n \frac{R^{\lambda /2}}{2^{\frac{n \lambda}{2}}} \left( \left(\frac{2^n}{R}\right)^\lambda \int_{A_n(z)}|\partial f|^2 \right)^{\frac{1}{2}} \leq KR \|Df\|_{L^{2,\lambda}}.
\end{equation}
Note that 
\[|\mathcal{P}(0)| \leq KR\]
and from the computations in (\ref{eq:primopezzo}) and (\ref{eq:29a}) we also see that
\be
\label{eq:pezzoinfinity}
\|\mathcal{P}(f)-\mathcal{P}(0)\|_{L^\infty}\leq KR \|f\|_{H}.
\ee
Also observe that, since we only need to integrate on $\xi \in B_{2R}(0)$, for $|z| \geq 4 R$ we have $|\xi-z| \geq \frac{|z|}{2}$, so $|\mathcal{P}(f)|$ is bounded by $\frac{K R^2}{|z|}(  \|Df\|_{L^2}+ \|f\|_{L^\infty})$.

$\mathcal{P}(f)$ is in $W^{1,2}$ (we will shortly show that $\mathcal{P}(f)\in H$) and solves
\begin{equation}
\label{eq:fixedpoint}
\bar{\partial}(\mathcal{P}(f))= -\chi_R \nu(\fiav,\aav) \partial f - \chi_R \frac{\nu(\fiav,\aav)}{z}(1+f),
\end{equation}
since $\frac{1}{z-\xi}$ is the fundamental solution for the operator $\bar{\partial}$; in fact, $\frac{1}{z-\xi}=\frac{\partial}{\partial z} (\ln|z-\xi|)$, and $\bar{\partial} \partial = i \Delta$, compare \cite{Gilbarg}, page.17.

Therefore, what we are looking for is a fixed point for $\mathcal{P}$ in $H$. Observe that $\mathcal{P}$ is an affine functional, therefore, to show that it is a contraction in $H$, it will be enough to show
\[\mathcal{P}(0)\in H,\]
\[\|\mathcal{P}(f)-\mathcal{P}(0)\|_H \leq k \|f\|_H,\]
for any $f$ and for some $0<k<1$. From (\ref{eq:fixedpoint}),
\[\|\bar{\partial}(\mathcal{P}(f)-\mathcal{P}(0))\|_{L^2}\leq KR (\|Df\|_{L^2} + \|f\|_{L^\infty})\]
and, since $\mathcal{P}(f)-\mathcal{P}(0)$ decays at infinity as $\frac{1}{|z|}$, we can integrate by parts to get
\be
\label{eq:pezzol2}
\|D(\mathcal{P}(f)-\mathcal{P}(0))\|_{L^2}=K\|\bar{\partial}(\mathcal{P}(f)-\mathcal{P}(0))\|_{L^2}\leq KR (\|Df\|_{L^2} + \|f\|_{L^\infty}).
\ee
The fact that
\[\|D(\mathcal{P}(f)-\mathcal{P}(0))\|_{L^{2,\lambda}} \leq KR (\|Df\|_{L^2}+ \|f\|_{L^{2,\lambda}})\]
follows from equation (\ref{eq:fixedpoint}) by theorem 5.4.1. in \cite{Morrey}, page. 146. 

The last estimate, together with (\ref{eq:pezzoinfinity}) and (\ref{eq:pezzol2}), implies
\[\|\mathcal{P}(f)-\mathcal{P}(0)\|_H\leq KR \|f\|_{H}.\]
Similarly we can show that $\mathcal{P}(0)\in H$, with $\|\mathcal{P}(0)\|_H \leq KR$.
If $R$ is small enough (recall that $\frac{\nu}{z} \leq K$ independently of $R$), we have a contraction  and by Caccioppoli's fixed point theorem we have the existence of a unique fixed point $u$ for $\mathcal{P}$ and 
\[\|u\|_{L^\infty}\leq 2KR.\]
So we have a H\"older function $w= z(1+u)$ solution to (\ref{eq:wCRiem}). Since $\nu(\fiav,\aav)$ is H\"older continuous of exponent $\delta$ thanks to theorem \ref{thm:holderestimate}, by means of a Shauder-type estimate $w$ is $C^{1,\delta}$.
\end{proof}

\begin{oss}
Observe that $|w(z)-z| \leq KR|z|$ implies that at $0$, $\partial w\approx 1, \bar{\partial}w \approx 0$, with perturbations of order $KR$. By taking $R$ smaller if necessary, we can assume, since $w$ is $C^{1,\delta}$, that $\partial w$ and $\bar{\partial}w$  stay as close as we like to $1$ and $0$ in $B_R$.
\end{oss}

We are now ready to complete the proof of non-accumulation, which will go on until the end of this section.
Take the function $G: \mathbb{C}^3 \rightarrow \mathbb{C}^3$ given by
\[G(z,\zeta, a) = (z,\zeta - \fiav(z), a - \aav (z)),\]
and consider the pushforward $\Gamma := G_* C$. Since $G$ is proper (if $K$ is compact, $G^{-1}(K)$ is closed by continuity and bounded since the average function is $L^\infty$) and $W^{1,2}$, the pushforward commutes with the boundary operator (see \cite{Giaquinta}, the point is that a $W^{1,2}$ function from a domain in $\mathbb{R}^2$ into $\mathbb{R}^3$ is approximable by $C^1$ functions), therefore $\partial\Gamma=0$. The current $\Gamma$ is described by the multigraph \[\{\sigma_j, \tau_j\}=\{\varphi_j - \fiav, \alpha_j-\aav\}.\] From (\ref{eq:ceq}), the smooth parts of $\{\sigma_j\}$ solve
\begin{equation}
\label{eq:eqav1}
\overline{\partial}\sigma_j+\nu(\fiav,\aav) \partial \sigma_j + \sum_{k=1}^Q S_j^k \sigma_k + \sum_{k=1}^Q T_j^k \tau_k =0,
\end{equation}
with $\displaystyle |T_j^k|, |S_j^k| \leq K (1 + \sum_{i=1}^Q |D\varphi_i| + \sum_{i=1}^Q |D\alpha_i|)$. Therefore, by the H\"older estimate in theorem \ref{thm:holderestimate}, $\displaystyle |T_j^k|, |S_j^k|$ are in $L^2(D_R)$. As for $\{\tau_j\}$, from (\ref{eq:££}) and (\ref{eq:£}) we have that 
\[\nabla \alpha_j(z) =h(z,\varphi_j(z), \alpha_j(z)),\]
for a smooth $\R^2$-valued $h$, so
\[\overline{\partial}\tau_j=\sum_{k=1}^Q A_j^k \sigma_k + \sum_{k=1}^Q B_j^k \tau_k\]
with $A_j^k, B_j^k$ bounded; for $\partial\tau_j$ we have a similar equation, since the $\tau_j$ are real (so the equation we wrote actually contains the whole information on the two real derivatives). Putting them together (we keep writing $A,B$ although these coefficient are different)
\begin{equation}
\label{eq:eqav2}
\overline{\partial}\tau_j+\nu(\fiav,\aav) \partial \tau_j + \sum_{k=1}^Q A_j^k \sigma_k + \sum_{k=1}^Q B_j^k \tau_k =0,
\end{equation}
with $A_j^k, B_j^k$ bounded.

Observe that singularities of order $Q$ in $C$ have the property that all the branches coincide at those points, therefore they are zeros of the multigraph $\{\sigma_j, \tau_j\}$.
Assume by contradiction the existence of a sequence of singular points in $Sing^Q$ accumulating to $0$. Then we can take $N$ points $q_n \in F = \pi(Sing^Q)$  which lie in $D_r$, with $N$ as large as we want and $r<R$ arbitrarily small and $\{\sigma_j, \tau_j\}(q_n)=0$, $n=1,...,N$. In the estimates to come, one should always pay attention to the fact that the constants obtained must not depend on the chosen $N$ and $r$, unless otherwise specified. 

Define the function
\[g(z):= \Pi_{i=1}^N (w(z)-w(q_i)),\]
with the $w$ obtained in the previous lemma. Then $g$ is a $C^1$, $W^{1,2}$ function and it solves on $D_R$
\[\overline{\partial}g +\nu(\fiav,\aav) \partial g =0.\]
Take $F: \mathbb{C}^3 \rightarrow \mathbb{C}^3$
\[F(z,\zeta, a) =\left(z,\chi_r(z)\frac{\zeta}{g(z)},\chi_r(z)\frac{a}{g(z)}\right),\]
where $\chi_r$ is a radial, smooth cut-off, $1$ on $B_r$, $0$ on the complement of $B_{2r}$, with gradient bounded by $\frac{K}{r}$;
we are going to analyse the pushforward $F_*(G_*(C))$. First observe that, on any set of the form $\displaystyle D_R \setminus \cup_{i=1}^N B_\delta(q_i)$ for $\delta$ as small as we want, $F$ is a $C^1$, Lipschitz and proper function. Thus, on 
\[A_\delta := (D_R \setminus \cup_{i=1}^N B_\delta(q_i)) \times \mathbb{C} \times \mathbb{C}\]
the pushforward $\Delta_\delta := F_*(\Gamma)$ is a well defined i.m. rectifiable current with finite mass, and it can develop boundary only on $\displaystyle (\cup_{i=1}^N \partial B_\delta(q_i)) \times \mathbb{C} \times \mathbb{C}$. Now we will prove

\begin{lem}
\label{lem:partint}
Sending $\delta \to 0$, we can define the pushforward $\Delta:= F_*(G_*(C))$ on the whole of $D_R \times \mathbb{C} \times \mathbb{C}$, and $\Delta$ is a boundaryless current of finite mass.
Then we can rewrite the following relation
\[\Delta(d\zeta d\bar{\zeta}) = \partial \Delta (\zeta d\bar{\zeta})\]
as a standard integration by parts formula, where both integrals are finite:
\begin{equation}
\label{eq:partint}
\int_{B_{2r}(0)} \sum_j \left|\bar{\partial} \left( \frac{\chi_r \sigma_j}{g} \right) \right|^2 = \int_{B_{2r}(0)} \sum_j\left|\partial \left( \frac{\chi_r \sigma_j}{g} \right) \right|^2 .
\end{equation}
\end{lem}

\begin{oss}
Formula (\ref{eq:partint}) is the only thing we will need in the sequel. The finiteness of the integrals was not clear in the analogous formula used in \cite{Taubes}. The reader might skip the proof of this lemma on a first reading.
\end{oss}

\begin{oss}
\label{oss:nablag}
In this formula $\nabla \left( \frac{\chi_r \sigma_j}{g} \right)$ is understood to be $0$ on the set $F=\pi(Sing^Q)$. The reason for this will be clear during the proof. On the complement $D_R \setminus \pi(Sing^Q)$ the gradient is well-defined since the functions are smooth except at the isolated points $\pi(Sing^{\leq Q-1})$.
\end{oss}

\begin{proof}[\bf{proof of lemma \ref{lem:partint}}]

From what we said before, $\Delta$ can develop boundary only on $\displaystyle (\cup_{i=1}^N q_i) \times \mathbb{C} \times \mathbb{C}$. Moreover, $\Delta$ is described by the multigraph \[\left\{\frac{\chi_r \sigma_j}{g}, \frac{\chi_r \tau_j}{g}\right\}_{j=1}^Q.\]
From theorem \ref{thm:lipesttan}, this multigraph is bounded on $D_R$, indeed we only have to check it at the points $q_i$: on some neighbourhood of a chosen $q_k$, thanks to corollary \ref{cor:lipest}, 
\[|\sigma_j(z)|= |\sigma_j(z)-\sigma(q_k)| \leq K |z-q_k|.\]
By Lagrange's theorem, if the mentioned neighbourhood was chosen small enough (its size should be much smaller than the distances between the $q_i$-s), then $g(z)\approx \Pi_{i=1}^N (z-q_i)$; more precisely, $K_1\Pi_{i=1}^N |z-q_i| \leq |g(z)|\leq K_2\Pi_{i=1}^N |z-q_i|$ with $K_1, K_2$ close to $1$ (the perturbation is due to the perturbations $\partial w \approx 1$ and $\overline{\partial} w \approx 0$). Therefore
\[|\sigma_j(z)|\leq K_{\{q_i\}} |g(z)|.\]
Notice that the constant obtained is not independent of the choices of $r$ and the set ${\{q_i\}}$, but all that matters to us is the fact that 
\[\left\{\frac{\chi_r \sigma_j}{g}, \frac{\chi_r \tau_j}{g}\right\}_{j=1}^Q.\]
is bounded, although its $L^\infty$-norm is about $\Pi_{i\neq j} |q_i-q_j|$.
We further observe that, thanks to the equation solved by $g$, the multigraph 
\[\left\{\frac{ \sigma_j}{g}, \frac{ \tau_j}{g}\right\}_{j=1}^Q\]
satisfies, on $D_R \setminus \mathcal{F}$,
\begin{equation}
\label{eq:g1}
\overline{\partial}\left(\frac{ \sigma_j}{g}\right)+\nu(\fiav,\aav) \partial \left(\frac{ \sigma_j}{g}\right) + \sum_{k=1}^Q S_j^k \left(\frac{ \sigma_k}{g}\right) + \sum_{k=1}^Q T_j^k \left(\frac{ \tau_k}{g}\right) =0,
\end{equation}

\begin{equation}
\label{eq:g2}
\overline{\partial}\left(\frac{ \tau_j}{g}\right)+\nu(\fiav,\aav) \partial \left(\frac{ \tau_j}{g}\right) + \sum_{k=1}^Q A_j^k \left(\frac{ \sigma_k}{g}\right) + \sum_{k=1}^Q B_j^k \left(\frac{ \tau_k}{g}\right) =0,
\end{equation}
with the coefficients $A,B,S,T$ as above.

\textsl{Step 1: $\Delta$ has finite mass.}
Remark that the set $\mathcal{F}=\pi(Sing^Q)$ is included in $\{z: \; \forall i \;\sigma_i(z)=\tau_i(z)=0 \}$. The integer multiplicity rectifiable current $\Delta_\delta$ possesses a.e. on $\mathcal{F} \setminus \cup_{i=1}^N B_\delta(q_i)$ an approximate tangent plane which must be horizontal, i.e. it must be the plane $(z,0,0)$. Indeed, this is true at any point of $\{z: \; \forall i \;\sigma_i(z)=\tau_i(z)=0 \}$ of density $1$, as can be seen from the definition of tangent plane (see \cite{Giaquinta} page 92).

Let us observe the action of $\Delta_\delta$ on $d\zeta \wedge d\bar{\zeta}$. By the observation we just made, this action gives $0$ on $\mathcal{F}$, therefore we should understand, in the following computations, $\nabla \left( \frac{\chi_r \sigma_j}{g} \right)=0$ on $\mathcal{F}$ (compare remark \ref{oss:nablag}). So we get:

\[\Delta_\delta (d\zeta \wedge d\bar{\zeta})=\int_{B_{2r}\setminus \cup_{i=1}^N B_\delta(q_i)} \sum_j d\left(\frac{\chi_r \sigma_j}{g}\right) \wedge d\overline{\left(\frac{\chi_r \sigma_j}{g}\right)} =\]
\begin{equation}
\label{eq:smoothaction}
= \int_{B_{2r}\setminus \cup_{i=1}^N B_\delta(q_i)}\sum_j\left |\bar{\partial}\left(\frac{\chi_r \sigma_j}{g}\right)\right |^2- \left | \partial\left(\frac{\chi_r \sigma_j}{g}\right)\right |^2.
\end{equation}
By (\ref{eq:g1}) and the triangle inequality $|a-b|^2 \geq \frac{|b|^2}{2}-|a|^2$ we get (recall $|\nu| \leq \eps$)
\[\int \sum_j \left|\sum_{k=1}^Q S_j^k \left(\frac{ \sigma_k}{g}\right) + \sum_{k=1}^Q T_j^k \left(\frac{ \tau_k}{g}\right) \right|^2 = \int \sum_j \left| \overline{\partial}\left(\frac{ \sigma_j}{g}\right)+\nu(\fiav,\aav) \partial \left(\frac{ \sigma_j}{g}\right) \right|^2 \geq\]

\[\geq \int \sum_j \left( \frac{\left| \overline{\partial}\left(\frac{ \sigma_j}{g}\right) \right|^2}{2} - {\eps}^2 \left|\partial \left(\frac{ \sigma_j}{g}\right) \right|^2 \right)\]

\[= \int \sum_j \left(\frac{1}{2} +{\eps}^2 \right) \left| \overline{\partial}\left(\frac{\sigma_j}{g}\right) \right|^2 + {\eps}^2 \sum_j \left( \left |\bar{\partial}\left(\frac{\chi_r \sigma_j}{g}\right)\right |^2- \left | \partial\left(\frac{\chi_r \sigma_j}{g}\right)\right |^2\right)=\]

\[= \int \sum_j \left(\frac{1}{2} +{\eps}^2 \right) \left| \overline{\partial}\left(\frac{\sigma_j}{g}\right) \right|^2 + {\eps}^2 \Delta_\delta (d\zeta \wedge d\bar{\zeta})=\]

\begin{equation}
\label{eq:triangle}
= \int \sum_j \left(\frac{1}{2} +{\eps}^2 \right) \left| \overline{\partial}\left(\frac{\sigma_j}{g}\right) \right|^2 + {\eps}^2 \partial\Delta_\delta (\zeta \wedge d\bar{\zeta}).
\end{equation}

Notice that the first term at the beginning of the last chain of inequalities is finite, from the condition on the $T$-s and $S$-s, and the fact that $\frac{\sigma_k}{g},\frac{\tau_k}{g}$ is bounded.

Let us restrict to a small ball $B_{\lambda}(q_i)$: we will show that
\[\lim_{\rho \to 0}\int_{B_{\lambda}(q_i)\setminus B_\rho(q_i)} \sum_j  \left| \overline{\partial}\left(\frac{\sigma_j}{g}\right) \right|^2 \]
is finite; the global finiteness on $D_R$ will follow since the $q_i$-s are finite and there are no poles elsewhere. 
In a first moment we are going to construct a sequence $\rho_n \downarrow 0$ for which $M(\partial \Delta_{\rho_n})$ is equibounded. 
Since $\Delta_\rho = F_*(\Gamma)$ and $||\nabla F||_{L^\infty(B_{\lambda}(q_i) \setminus B_\rho(q_i))} \leq \frac{K}{\rho}$, from \cite{Giaquinta}, page 134, we get
\begin{equation}
\label{eq:slicebounding}
M(\partial \Delta_\rho)\leq \frac{K}{\rho} M(\partial \Gamma_\rho).
\end{equation}
Moreover, from slicing theory, see Prop. 2 in \cite{Giaquinta}, page 154 
\[\frac{1}{{(\lambda/n)}^2} \int_0^{\frac{\lambda}{n}} M(\partial \Gamma_\rho) d\rho= \frac{1}{{(\lambda/n)}^2} \int_0^{\frac{\lambda}{n}} M(\langle  \Gamma,|z|, \rho \rangle) \leq \frac{1}{{(\lambda/n)}^2} M(\Gamma \res (B_{\frac{\lambda}{n}}(q_i)\times \mathbb{C}^2))\]
and this is bounded as $n \to \infty$ by the monotonicity formula, since the tangent is horizontal at $(q_i,0,0)$ and the multiplicity of this point is $Q$.
Then 
\[\frac{1}{{(\lambda/n)}}\int_0^{\frac{\lambda}{n}} M(\partial \Gamma_\rho) d\rho \leq K \frac{\lambda}{n}\]
so by the mean-value theorem there is $\frac{\lambda}{4n} \leq \rho_n \leq \frac{\lambda}{n}$ such that 
\[M(\partial \Gamma_{\rho_n}) \leq 2 \frac{1}{{(\lambda/n)}}\int_0^{\frac{\lambda}{n}} M(\partial \Gamma_\rho) d\rho \leq 2K \frac{\lambda}{n}\leq 8K \rho_n.\]
Now (\ref{eq:slicebounding}) yields that $M(\partial \Delta_{\rho_n})$ are equibounded. 

As observed above, $\frac{\sigma_j}{g}$ is $L^\infty$, therefore 
the function $\zeta$ is bounded on $\Delta$, so there is some constant which bounds uniformly in $n$
\[|\partial \Delta_{\rho_n} (\zeta d\bar{\zeta})|.\]
This yields, together with the $B_{\lambda}$-version of (\ref{eq:triangle}),
\[\lim_{{\rho_n} \to 0}\int_{B_\lambda(q_i) \setminus B_{\rho_n}(q_i)} \sum_j  \left| \overline{\partial}\left(\frac{\sigma_j}{g}\right) \right|^2 <\infty,\]
and consequently
\[\lim_{\rho \to 0}\int_{B_{\lambda}(q_i) \setminus B_\rho(q_i)} \sum_j  \left| \overline{\partial}\left(\frac{\sigma_j}{g}\right) \right|^2 <\infty\]
since this integral is a monotone function of $\rho$, so the limit must exist and it is enough to check in on a sequence.
Once we have the finiteness of 
\[\int_{D_R} \sum_j  \left| \overline{\partial}\left(\frac{\sigma_j}{g}\right) \right|^2,\]
using $|\Delta_{\rho_n} (d\zeta \wedge d\bar{\zeta})|=|\partial \Delta_{\rho_n} (\zeta d\bar{\zeta})|<\infty$ again, by (\ref{eq:smoothaction}) we also get the finiteness of 
\[\int_{D_R} \sum_j  \left| \partial\left(\frac{\sigma_j}{g}\right) \right|^2.\]
This implies that the Jacobian minors of $\frac{\sigma_j}{g}$ are in $L^1$, so the finiteness of the mass can be obtained by the Area formula\footnote{Recall that it is enough to apply the Area formula to the smooth parts of the current $\Delta$ which are above $D_R \setminus \mathcal{F}$. The rest of the current lies in $\mathcal{F}$, which has finite measure.}, see \cite{Giaquinta} page 225.

\medskip

\textsl{Step 2: $\Delta$ has no boundary.} As said above, we only have to exclude boundary terms localized at the points $q_i$. As before, we restrict ourselves to $\Delta \; \res\; {B_{\lambda}(q_i)}\times \mathbb{C}^2$. During this step, we will keep denoting this current by $\Delta$. To simplify things, we will test $\partial \Delta$ only on the 1-forms $\chi_\rho(z)\zeta d\bar{\zeta}$ which is needed for the integration by parts formula (\ref{eq:partint}); the proof for other 1-forms is similar\footnote{For the reader who is familiar with the support theorem for Flat-currents (see \cite{Giaquinta} page 525), we remark that the absence of boundary can be obtained by showing, via an approximation argument, that $\p \Delta$ is a Flat 1-current. The quoted theorem then implies that $\p \Delta=0$. }. Since the possible boundary in the interior of $D_R$ is localized only in $q_i \times \mathbb{C}^2$, the result will be the same for any $\rho$.
\[\partial \Delta (\chi_\rho(z)\zeta d\bar{\zeta}) = \Delta(d \chi_\rho \wedge \zeta d\bar{\zeta}) + \Delta( \chi_\rho  d\zeta \wedge d\bar{\zeta}).\]
From the previous step, 
\[|\Delta( \chi_\rho  d\zeta \wedge d\bar{\zeta})| \leq  \int_{B_{2 \rho}(q_i)} \sum_j  \left| d\left(\frac{\sigma_j}{g}\right) \right|^2 \to 0\]
for $\rho \to 0$. Let us now analyse the first term:
\[\left|\Delta(d \chi_\rho \wedge \zeta d\bar{\zeta})\right| = \left|\int_{B_{2 \rho}(q_i) \setminus B_{\rho}(q_i) } \sum_j \partial \chi_\rho \frac{\sigma_j}{g} \bar{\partial}\left( \frac{\sigma_j}{g} \right)\right| \leq \]
\[\leq \int_{B_{2 \rho}(q_i) \setminus B_{\rho}(q_i)} \frac{K}{\rho} \left\| \frac{\sigma_j}{g}\right\|_{L^\infty} \left|d \left( \frac{\sigma_j}{g} \right) \right|\]
and by H\"older's inequality
\[\leq \left\| \frac{\sigma_j}{g}\right\|_{L^\infty} \frac{K}{\rho} 2 \rho \left( \int_{B_{2 \rho}(q_i)} \sum_j \left|d \left( \frac{\sigma_j}{g} \right) \right|^2 \right)^{\frac{1}{2}}.\]
This integral goes to $0$ as $\rho \to 0$ thanks to the previous step. So there is no boundary term at any of the $q_i$ when we test on the one form $\zeta d\bar{\zeta}$.
\end{proof}

\medskip

We are now ready to finish the proof of non accumulation started before lemma \ref{lem:partint}: recall that we assumed, by contradiction, the existence of $N$ points $q_n \in \mathcal{F} = \pi(Sing^Q)$  which lie in $D_r$, with $N$ as large as we want and $r<R$ arbitrarily small and $\{\sigma_j, \tau_j\}(q_n)=0$, $n=1,...,N$. From Leibnitz rule and (\ref{eq:g1}) 

\[\int_{B_{2r}}\sum_j\left|\bar{\partial} \left( \frac{\chi_r \sigma_j}{g} \right) \right|^2 \leq K r^{-2} \int_{B_{2r}\setminus B_r} \sum_j \left|\frac{\sigma_j}{g}\right|^2 + \int_{B_{2r}}\sum_j \left|\chi_r \right|^2 \left|\bar{\partial} \left( \frac{\sigma_j}{g} \right) \right|^2\]

\[= K r^{-2} \int_{B_{2r} \setminus B_r} \sum_j \left|\frac{\sigma_j}{g}\right|^2 + \int_{B_{2r}}\sum_j |\nu(\fiav, \aav)|^2 \left|\chi_r \right|^2 \left|\partial \left( \frac{\sigma_j}{g} \right) \right|^2 +\]\[+ K \int_{B_{2r}} \left( 1 + \sum_{i=1}^Q |D\varphi_i|^2 + \sum_{i=1}^Q |D\alpha_i|^2 \right)\sum_j \left( \left|\chi_r \frac{\sigma_j}{g}\right|^2 + \left|\chi_r \frac{\tau_j}{g}\right|^2 \right).\]

Now, using $|\nu|\leq \eps$ and (\ref{eq:partint}) (notice that the previous lemma and the fact that $\{\frac{\sigma}{g},\frac{\tau}{g}\}$ is bounded guarantee the finiteness of all terms),
\[\int_{B_{2r}}\sum_j |\nu(\fiav, \aav)|^2 \left|\chi_r \partial \left( \frac{\sigma_j}{g} \right) \right|^2 = \int_{B_{2r}}\sum_j |\nu(\fiav, \aav)|^2 \left|-\partial \chi_r \left( \frac{\sigma_j}{g} \right) + \partial \left( \frac{\chi_r \sigma_j}{g} \right) \right|^2\]

\[\leq K{\eps}^2 r^{-2} \int_{B_{2r} \setminus B_r} \sum_j \left|\frac{\sigma_j}{g}\right|^2  + {\eps}^2 \int_{B_{2r}} \left|\bar{\partial} \left( \frac{\chi_r \sigma_j}{g} \right) \right|^2.\]

Putting all together, with a further use of (\ref{eq:partint}) on the l.h.s., we get
\begin{equation}
\label{eq:sigma}
\int_{B_{2r}}\sum_j\left|d \left( \frac{\chi_r \sigma_j}{g} \right) \right|^2= 2\int_{B_{2r}}\sum_j\left|\bar{\partial} \left( \frac{\chi_r \sigma_j}{g} \right) \right|^2 \leq K r^{-2} \int_{B_{2r}-B_r} \sum_j \left|\frac{\sigma_j}{g}\right|^2 + 
\end{equation}
\[+K \int_{B_{2r}} \left( 1 + \sum_{i=1}^Q |D\varphi_i|^2 + \sum_{i=1}^Q |D\alpha_i|^2 \right)\sum_j \left( \left|\chi_r \frac{\sigma_j}{g}\right|^2 + \left|\chi_r \frac{\tau_j}{g}\right|^2 \right).\]

Similarly, from (\ref{eq:g2}), and using the analogous partial integration
\[\int_{B_{2r}(0)} \sum_j \left|\bar{\partial} \left( \frac{\chi_r \tau_j}{g} \right) \right|^2 = \int_{B_{2r}(0)} \sum_j\left|\partial \left( \frac{\chi_r \tau_j}{g} \right) \right|^2,\]
we get

\begin{equation}
\label{eq:tau}
\int_{B_{2r}}\sum_j\left|d\left( \frac{\chi_r \tau_j}{g} \right) \right|^2 \leq K r^{-2} \int_{B_{2r} \setminus B_r} \sum_j \left|\frac{\tau_j}{g}\right|^2 + 
\end{equation}
\[+K \int_{B_{2r}} \left( 1 + \sum_{i=1}^Q |D\varphi_i|^2 + \sum_{i=1}^Q |D\alpha_i|^2 \right)\sum_j \left( \left|\chi_r \frac{\sigma_j}{g}\right|^2 + \left|\chi_r \frac{\tau_j}{g}\right|^2 \right).\]

Set now $\displaystyle v:=\max_j\left \{\left|\chi_r \frac{\sigma_j}{g}\right|, \left|\chi_r \frac{\tau_j}{g}\right|\right\}$. This function is $W^{1,2}$: indeed, this is true on $D_{2r} \setminus \pi(Sing^{\leq Q})$, since it is the maximum of $W^{1,2}$ functions; then by arguments already used,
\begin{itemize}
	\item $\pi(Sing^{\leq Q-1})$ are isolated points so we can extend the $W^{1,2}$  estimate to $D_{2r} \setminus \pi(Sing^{Q})$;
	\item then we extend to $\displaystyle D_{2r} \setminus(\cup_{i=1}^N B_\delta(q_i)\cap \mathcal{F})$ for any arbitrarily small $\delta$, thanks to the fact that $v=0$ on on $Sing^Q \setminus \{q_i\}$;
	\item finally, sending $\delta \to 0$, to the whole of $D_R$ since the $q_i$ are isolated.
\end{itemize}

Also observe that, by the Cauchy-Schwarz inequality, $|d(|v|)| \leq |dv|$, so 
\[\int_{B_{2r}}|dv|^2 \leq \int_{B_{2r}}\sum_j\left|d \left( \frac{\chi_r \sigma_j}{g} \right) \right|^2 + \sum_j\left|d \left( \frac{\chi_r \tau_j}{g} \right) \right|^2,\]
so (\ref{eq:sigma}) and (\ref{eq:tau}) imply
\[\int_{B_{2r}}|dv|^2 \leq K r^{-2} \int_{B_{2r} \setminus B_r} \sum_j \left(\left|\frac{\sigma_j}{g}\right|^2+\left|\frac{\tau_j}{g}\right|^2\right)+\]\[ +K \int_{B_{2r}} \left( 1 + \sum_{i=1}^Q |D\varphi_i|^2 + \sum_{i=1}^Q |D\alpha_i|^2 \right) v^2.\]
Recall that $\left( 1 + \sum_{i=1}^Q |D\varphi_i|^2 + \sum_{i=1}^Q |D\alpha_i|^2 \right)$ is $L^1$ by theorem \ref{thm:holderestimate} (H\"older estimate); then, by lemma 5.4.1. in \cite{Morrey}, we get the existence of $\delta >0$ such that the last term can be bounded by
\[\int_{B_{2r}} \left( 1 + \sum_{i=1}^Q |D\varphi_i|^2 + \sum_{i=1}^Q |D\alpha_i|^2 \right) v^2 \leq K r^\delta \int_{B_{2r}}|dv|^2,\]
so we can write
\[r^2 \int_{B_{2r}}|dv|^2 \leq K \int_{B_{2r} \setminus B_r} \sum_j \left(\left|\frac{\sigma_j}{g}\right|^2+\left|\frac{\tau_j}{g}\right|^2\right);\]
now, since $v \in W^{1,2}_0 (B_{2r})$, by Poincar\'e's inequality
\[\int_{B_{2r}}v^2 \leq K \int_{B_{2r} \setminus B_r} \sum_j \left(\left|\frac{\sigma_j}{g}\right|^2+\left|\frac{\tau_j}{g}\right|^2\right).\]
Since $\displaystyle \sum_j \left(\left|\chi_r \frac{\sigma_j}{g}\right|^2+\left|\chi_r\frac{\tau_j}{g}\right|^2\right) \leq 2Q v^2$ by definition of $v$, and $\chi_r = 1$ on $B_r$, the last inequality implies the following Carleman-type estimate
\begin{equation}
\label{eq:carleman}
\int_{B_{r/4}} \sum_j \left(\left|\frac{\sigma_j}{g}\right|^2+\left|\frac{\tau_j}{g}\right|^2\right) \leq K \int_{B_{2r} \setminus B_r} \sum_j \left(\left|\frac{\sigma_j}{g}\right|^2+\left|\frac{\tau_j}{g}\right|^2\right)
\end{equation}
with $K$ independent of $r$ and the cardinality $N$ of the set $\{q_i\}$.
Assume that the $\{q_i\}$ were chosen much inside $D_r$, say in $D_{r/4}$. Then, from the definition of $g$, if $r$ was chosen small enough (which doesn't influence $K$), on the l.h.s. of (\ref{eq:carleman}) $g\leq \left(\frac{3r}{4}\right)^N$, while on the r.h.s. $g \geq \left(\frac{3|z|}{4}\right)^N$, so we get
\[\int_{B_{r/4}} \sum_j |\sigma_j|^2+|\tau_j|^2 \leq K \int_{B_{2r} \setminus B_r} \left(\frac{r}{|z|}\right)^{2N}\sum_j |\sigma_j|^2+|\tau_j|^2;\]
letting $N$ go to infinity, we can make the r.h.s. as small as we wish, which implies 
\[\int_{B_{r/4}} \sum_j |\sigma_j|^2+|\tau_j|^2 = 0,\]
i.e. all the branches of the multigraph describing our original current must agree with the average on a neighbourhood of $0$. But then this average must be itself a Special Legendrian counted $Q$ times, therefore it must be smooth in this neighbourhood thanks to the basic step of the induction. We have therefore completed the proof of $\partone$:

\begin{thm}
\label{thm:partone}
Let $B^5$ be a ball in which the highest multiplicity for the Special Legendrian cycle $C$ is $Q$. Assume that $Sing^{\leq Q-1}$ is made of isolated points in $(C \res B^5) \setminus Sing^Q$. Then the set $Sing^Q$ is made of isolated points in $B^5$.
\end{thm}

\section{Proof of $\parttwo$: non-accumulation of lower-order singularities}

To complete the proof of the inductive step, we have to exclude the possibility of accumulation of points in $Sing^{\leq Q-1}$ to a singularity of order $Q$. 

Let $x_0 \in Sing^Q$; from theorem \ref{thm:partone} (and recalling the monotonicity formula) we can assume that we work in a ball $B^5$ centered at $x_0$ such that all the points of $C$ in this ball are of multiplicity at most $Q$ and 
\[B^5 \cap Sing^Q = \{x_0\}.\]
By the inductive assumption, the other singularities in $B^5$ are isolated and of multiplicity $\leq Q-1$. 

Thus we can take local coordinates about $x_0$ in such a way that $C$ is given by a $Q-$valued graph over $D^2$ that we denote by
\[
\{(\varphi_j(z),\al_j(z))\}_{j=1\cdots Q}\quad,
\]
where $z=x+iy$ is the coordinate in the Disk $D^2$, $\varphi_i\in{\C}$ and $\al_i\in{\R}$ and where for all $j=1\cdots Q$ $(\varphi_j(0),\al_j(0))=(0,0)$. 

\medskip

\textbf{Assumption on the multiplicity}.
In order to simplify the exposition, we assume that all smooth points of $C \res B^5$ have multiplicity exactly $1$. The following argument shows that there is no loss of generality in doing so\footnote{This assumption is not really needed to perform the proof presented in this last section, however it makes it less technical.}. 

If a smooth point $p$ has multiplicity $M\geq 2$, it must have a neighbourhood all made of smooth points of equal multiplicity $M$. Take the maximal of such neighbourhoods and denote it by $\mathcal{U}$. This smooth submanifold, counted once, constitutes an i.m. current $U$ in $B^5$, whose smooth points have multiplicity $1$, possibly having singularities located at the same points where the singular points of $C$ were. 

We claim that $U$ is a boundaryless current. Let us prove it. Let $\{q_i\}$ be the at most countable singularities of $C$ of order $\leq Q-1$, possibly accumulating onto $0$. First of all, from the maximality of $\mathcal{U}$ we can deduce that the topological boundary $\p \mathcal{U}$ inside the smooth 2-dimensional submanifold $\displaystyle (\mathcal{C} \setminus \{0\}) \setminus \cup q_i$ is empty. This implies that $\p U$ must be supported at the singularities. Thanks to this, we can localize $U$ to a neighbourhood $V^5_i$ of each isolated singularity and we can exclude the presence of boundary at each $q_i$ as follows. By abuse of notation we keep denoting by $U$ the localized current.

We will write $B_\lambda$ for the ball $B^5_\lambda(q_i)$. For almost any choice of $\lambda >0$, the slice of $U$ with $\p B_\lambda$ exists as a 1-dimensional rectifiable current of finite mass and it is the same current, with opposite sign, as the boundary of $U_\lambda:=U \res (V_i^5 \setminus B_\lambda)$. Moreover from slicing theory we have $$\int_0^{\overline{\lambda}} M(\p U_\lambda) d\lambda \leq M(U \res B_{\overline{\lambda}}) \leq M(C \res B_{\overline{\lambda}}).$$ 
From the monotonicity formula and by the mean value theorem, we get the existence of a sequence $\{\lambda_n \}\to 0$ of positive real numbers such that 
$$M(\p U_{\lambda_n}) \leq K \lambda_n,$$
which implies that $\p U_{\lambda_n} \rightharpoonup 0$. On the other hand, $U_{\lambda_n} \rightharpoonup U$ since $M(U - U_{\lambda_n})=M(U \res B_{\lambda_n}) \to 0$, therefore $\p U_{\lambda_n} \rightharpoonup \p U$ and we get $\p U=0$.

Once we have excluded the presence of boundary located at the singularities $q_i$, we can perform the same argument to exclude boundary located at $0$. So $U$ is boundaryless\footnote{An alternative argument to exclude boundary located at the singular set, is to use an analogous approximation $U_n$ of $U$ obtained by "cutting out" smaller and smaller balls around the singular set and show that $\p U_n$ is a Cauchy sequence in the Flat-norm, therefore obtaining that $\p U$ is a Flat 1-dimensional current. The support theorem (see \cite{Giaquinta} page 525) tells us that a non-zero Flat $1$-current cannot be supported on a set of $0$-Hausdorff dimension, therefore $\p U =0$. }.

The current $C - (M-1)U$ is thus still a Special Legendrian cycle and has exactly the same singularities as $C$; it is therefore enough to prove the result about non accumulation for this Special Legendrian ``subcurrent'', in order to get in for $C$. Starting now from $C-(M-1)U$, we can inductively repeat the argument and get to the desired assumption of having multiplicity $1$ at all smooth points.

\medskip

We still denote by $\pi$
the map on $C$ which assigns the coordinate $z$.
The singularities of order $\leq Q-1$ are located exactly at the points $\pi^{-1}(z_l)$ for which $z_l\neq 0$ and
\be
\label{I.1}
\exists j\ne k\quad\mbox{ s.t. } \quad (\varphi_j(z_0),\al_j(z_0))=(\varphi_k(z_0),\al_k(z_0))\quad .
\ee
As recalled at the beginning of this section, we are working under the assumption that the points in (\ref{I.1}) form a discrete set in $D^2 \setminus {0}$, therefore at most countable. Away from them, each branch $j$ of the multiple valued graph satisfies a system\footnote{These are the equations we derived in (\ref{eq:realcrsystem}) and (\ref{eq:ceq}). With respect to the notations in sections \ref{PDEaverage} and \ref{uniquecont}, we are changing here the signs of the functions $\nu$ and $\mu$.} of
the form
\be
\label{I.1a}
\left\{
\begin{array}{l}
\ds\partial_{\ov{z}}\varphi_j=\nu((\varphi_j,\al_j),z)\ \partial_{z}\varphi_j+\mu((\varphi_j,\al_j),z)\\[5mm]
\ds\nabla\al_j=h((\varphi_j,\al_j),z)
\end{array}\right.
\ee
where $\nu$ and $\mu$ are smooth complex valued functions on ${\R}^5$ such that $\nu(0)=\mu(0)=0$ and
$h$ is a smooth ${\R}^2-$valued map on ${\R}^5$.

\medskip

To complete the proof of the main result we need to show

\begin{thm}
\label{th-I.1}
With the previous notations, let $0$ be a singular point of multiplicity $Q$ of the Special Legendrian cycle. If we are working under the (inductive) assumption that all the other singularities are of order $\leq Q-1$ and are isolated in $B^5 \setminus \{0\}$, then there is no accumulation at $0$ of singularities of the form (\ref{I.1}).
\end{thm}
 
The proof of the theorem~\ref{th-I.1} we are giving below is inspired by the homological type argument in \cite{Taubes}, pages 83-84. In view of this, we are now going to analyse the structure of the Special Legendrian current in a neighbourhood of an isolated singular point $q$. 

\medskip

\textbf{The structure of an isolated singularity.}
Recalling our assumption on multiplicities, given an isolated singular point $q$ in $C$, for a small enough radius $\rho$, $C \res B^5_{\rho}(q)$ can be represented as
\[C \res B^5_\rho(q) = \oplus_{i=1}^N L_i,\]
where each $L_i$ is either a smooth Special Legendrian embedded disk, or an immersed one branched at $q$; $N$ is bounded by the multiplicity of $q$ in $C$ and $L_i \neq L_j$ if $i \neq j$.

We give a brief description of the reason why this is true. Consider the slice $\langle C, |p-q|=\rho \rangle $: this is a smooth, one-dimensional, boundaryless current $\gamma$, so it is made of several smooth simple closed curves $\gamma_i$, each one counted with multiplicity $1$.

Each $\gamma_i$ can be obtained as the image of a circle $(\rho \cos t, \rho \sin t )\subset \mathbb{R}^2 \equiv \mathbb{C}^2$ through a smooth simple map. By the smoothness assumption on all points of $\gamma_i$, we can get a smooth parametrization from an annulus in $\mathbb{C}^2$ to a subset of $C$ contained in a corresponding annulus. Take the maximal extension: since there are no other singularities, this must be a smooth simple map from $B_\rho \setminus \{0\}$ into $(C \res B_\rho) \setminus \{ q \}$.

By a removable singularity theorem, this map can be extended smoothly in $0$. There is no real need to invoke such a theorem: the extension to $0$ is obviously continuous, and it is indeed smooth by standard elliptic theory. Thus get a smooth map from  $B_\rho$ into $C \res B_\rho$; repeat the same argument for all connected components $i$-s. A mass comparison shows that this procedure must cover the whole of $C \res B_\rho$.

The following discussion is needed to understand the behaviour of the difference functions $\varphi_i - \varphi_j$ and $\alpha_i - \alpha_j$ for $i \neq j$ in a small neighbourhood of an isolated singularity $q$; let $z_l=\pi(q)$. Take the neighbourhood of the cylindrical form $B^2_\rho \times B^3_\rho$ and denote by
\[\{(\varphi_j(z),\alpha_j(z))\}_{j=1\cdots M}\quad\]
the multivalued graph describing $C \res (B^2_\rho \times B^3_\rho)$ above $B^2_{\rho}(z_l)$, where $M$ is the multiplicity of the singularity $q$. Remark that $(\varphi_j(z_l),\alpha_j(z_l))$ coincide for all $j=1,...,M$, while for $z \neq z_l$ we have $(\varphi_j(z),\alpha_j(z)) \neq (\varphi_i(z),\alpha_i(z))$ whenever $i \neq j$ (this follows from the assumption on multiplicities taken at the beginning of this section).

Above any $z \in B^2_{\rho}(z_l)$, consider the difference vector $((\varphi_i-\varphi_j)(z), (\alpha_i-\alpha_j)(z)) \in \R^3$ for any choice of $i \neq j$. The tail and head of this vector will belong respectively to some $L_k$ and $L_l$, possibly with $k=l$. Observe that, moving this vector by continuity, this condition will be preserved with the same $k$ and $l$; remark that if $k=l$, the difference vector is joining two points of the same branched disk, while if $k \neq l$ it is joining points belonging to different disks. 

For any fixed choice of $(k,l) \in \{1,..., N\} \times \{1,..., N\}$, we are now going to analyse the functions $\varphi_i-\varphi_j$ and $\alpha_i-\alpha_j$ for $i\neq j$ s.t. 
\be
\label{eq:branches}
\text{$(\varphi_i, \alpha_i)$ belongs to a branch of $L_k$ and $(\varphi_j, \alpha_j)$ to a branch of $L_l$.}
\ee
From the second equation of the Special Legendrian system (\ref{I.1a}), taking differences, we get locally
\be
\label{eq:FG}
\nabla (\alpha_i - \alpha_j) = F \cdot (\varphi_i - \varphi_j) + G \cdot (\alpha_i - \alpha_j),
\ee
where $F,G$ are bounded functions of $(z,\varphi_i(z),\varphi_j(z),\alpha_i(z),\alpha_j(z))$ depending on the derivatives of $A$; they satisfy $|F|, |G| \leq K_0$. Take a positive $\overline{t} < \frac{1}{4 K_0}$. In the ball $\{ |z-z_l| \leq \overline{t})\}$, consider the point $w$ where $|\alpha_i - \alpha_j|$ realizes its maximum, with $i\neq j$ satisfying (\ref{eq:branches}). It makes sense to integrate the equation above along the segment $I$ joining $z_l$ to $w$ and get
\[(\alpha_i - \alpha_j)(w) = \int_0^t (F|_I)(s) (\varphi_i - \varphi_j)(s) ds + \int_0^t (G|_I)(s) (\alpha_i - \alpha_j)(s) ds\]
for $t=|w| \leq \overline{t}$. Thus, for $i \neq j$ as in (\ref{eq:branches}) we have
\[||\alpha_i - \alpha_j||_{L^\infty}= |\alpha_i - \alpha_j|(t) \leq K_0 t ||\varphi_i - \varphi_j||_{L^\infty} + K_0 t ||\alpha_i - \alpha_j||_{L^\infty}\]
which implies
\[||\alpha_i - \alpha_j||_{L^\infty (B^2(z_l, \overline{t}))} \leq \frac{1}{2}||\varphi_i - \varphi_j||_{L^\infty (B^2(z_l, \overline{t}))},\]
with $i$ and $j$ as prescribed in (\ref{eq:branches}).
Choosing $\overline{t}$ smaller at the beginning, we can get an arbitrarily small constant instead of $\frac{1}{2}$: therefore
\be
\label{eq:branchesl}
\frac{||\alpha_i - \alpha_j||_{L^\infty (B^2(z_l, t))}}{||\varphi_i - \varphi_j||_{L^\infty (B^2(z_l,t))}} \to 0 \text{ as } t \to 0,
\ee
for $i,j$ as in (\ref{eq:branches}).

For each fixed choice of $i \neq j$ as in (\ref{eq:branches}), we introduce the following multivalued graph on $\{|z|\leq 1 \}$, with $\rho>0$:
\[\left( \Theta^\rho_{ij}(z), \Xi^\rho_{ij}(z) \right ) = \left (\frac{(\varphi_i - \varphi_j)(z_l+\rho z)}{||\varphi_i - \varphi_j||_{L^\infty (B^2(z_l, \rho))}}, \frac{(\alpha_i - \alpha_j)(z_l+\rho z)}{||\varphi_i - \varphi_j||_{L^\infty (B^2(z_l, \rho))}}\right ).\]
Thanks to (\ref{eq:branchesl}), both $\Theta^\rho_{ij}$ and $\Xi^\rho_{ij}$ are smaller or equal than $1$ in modulus; more precisely $\Xi^\rho_{ij}$ goes uniformly to $0$ as $\rho \to 0$ and $|\Theta^\rho_{ij}|$ always realizes the value $1$ by definition.
From (\ref{I.1a}) and (\ref{eq:FG}), the branches of this multivalued graph solve locally on $\{0<|z|\leq 1 \}$ equations of the following type:
\be
\label{eq:systemsingp}
\left \{ \begin{array}{c} 
\overline{\partial}\Theta^\rho_{ij}(z) + \nu(z_l+\rho z)\partial \Theta^\rho_{ij}(z) + \rho S(\rho z) \Theta^\rho_{ij}(z)+ \rho T(\rho z) \Xi^\rho_{ij}(z)=0 \\
\nabla \Xi^\rho_{ij}(z) = \rho F(\rho z) \Theta^\rho_{ij}(z) + \rho G(\rho z) \Xi^\rho_{ij}(z), \end{array} \right .
\ee
with $F, G \in L^\infty$ and $S,T \in L^2$.
All the multivalued graphs of the sequence are pinched at $0$. By an argument similar to the one used in theorem \ref{thm:holderestimate}, we can deduce a uniform H\"older estimate on $\left( \Theta^\rho_{ij}(z), \Xi^\rho_{ij}(z) \right )$ independent of $\rho$. 
By Ascoli-Arzel\`a's theorem, as $\rho \to 0$, we can extract a subsequence converging uniformly to a m.-v. graph $\left( \Theta_{ij}(z), \Xi_{ij}(z) \right )$. Sending the equations in (\ref{eq:systemsingp}) to the limit as $\rho \to 0$, we also get that the limiting $\left( \Theta_{ij}(z), \Xi_{ij}(z) \right )$ must solve locally on $\{0<|z|\leq 1 \}$
\be
\label{eq:systemsingh}
\left \{ \begin{array}{c} 
\overline{\partial}\Theta_{ij}(z) + \nu(z_l)\partial \Theta_{ij}(z)=0, \text{ with } |\nu(z_l)|<<1,\\
\nabla \Xi_{ij}(z) = 0.\end{array} \right .
\ee
Therefore, since $\left( \Theta_{ij}(0), \Xi_{ij}(0) \right ) = (0,0)$, from the second equation we recover once again $\left( \Theta_{ij}(z), \Xi_{ij}(z) \right )$ must be of the form $\left(\Theta_{ij}(z), 0 \right)$. Consider now the equation for $\Theta_{ij}$: with the linear change of complex variable $z \to w$, 
\[w= \sqrt{\frac{1}{1+|\nu(z_l)|^2}}\;\;z + \nu(z_l)\sqrt{\frac{1}{1+|\nu(z_l)|^2}}\;\;\overline{z}\]
we can deduce that $\Theta_{ij}$ solves 
\[\frac{\partial}{\partial \overline{w}}\Theta_{ij}(w)=0 ;\]
thus $\Theta_{ij}$ is holomorphic w.r.t. the variable $w$. We will also say that it is almost-holomorphic in $z$; more precisely, it can be shown that $\Theta= (\lambda z + \mu \overline{z})^{\tau}$, with $|\mu| << |\lambda|$, $ \Q \ni \tau >0$.

\begin{lem}
\label{lem:lemma1}
Fix $(k,l) \in \{1,..., N\} \times \{1,..., N\}$; for $i\neq j$ s.t. $(\varphi_i, \alpha_i)$ belongs to a branch of $L_k$ and $(\varphi_j, \alpha_j)$ to a branch of $L_l$ (possibly with $k=l$), the following holds: for any $\delta >0$ there is $\rho$ small enough, s.t. 
\[\frac{|\alpha_i - \alpha_j|^2}{|\varphi_i - \varphi_j|^2+|\alpha_i - \alpha_j|^2}(z) < \delta\]
for $|z|< \rho$.
\end{lem}

\begin{proof}[\bf{proof of lemma \ref{lem:lemma1}}]
All the possible uniform limits of sequences 
\[\left( \Theta_{ij}^{\rho}(z), \Xi_{ij}^{\rho}(z) \right )\]
as $\rho \to 0$, must be of the form $\left(\Theta_{ij}(z), 0 \right)$ with $\Theta$ almost-holomorphic, of modulus $1$ and satisfying the H\"older-estimate (in terms of the explicit form of $\Theta$, this estimate is reflected upon $\tau$). Therefore, for any $\delta >0$ we can choose $\rho$ small enough so that 
\[\frac{|\alpha_i - \alpha_j|^2}{|\varphi_i - \varphi_j|^2+|\alpha_i - \alpha_j|^2}= \frac{|\Xi_{ij}|^2}{|\Theta_{ij}|^2+ |\Xi_{ij}|^2} < \delta.\]
\end{proof}

\begin{lem}
\label{lem:lemmadegree}
Fix $(k,l) \in \{1,..., N\} \times \{1,..., N\}$; by lemma \ref{lem:lemma1},  for $\rho$ small enough and for $i\neq j$ s.t. $(\varphi_i, \alpha_i)$ belongs to a branch of $L_k$ and $(\varphi_j, \alpha_j)$ to a branch of $L_l$ (possibly with $k=l$), it makes sense to compute the degree of 
\[\frac{\varphi_i - \varphi_j}{|\varphi_i - \varphi_j|}\]
on the closed curve $\gamma= L_l \cap \pi^{-1}\{|z-z_l|=\rho\}$.
This degree is strictly positive.
\end{lem}

\begin{proof}[\bf{proof of lemma \ref{lem:lemmadegree}}]
$\gamma$ is a closed, connected curve; orient it so that its projection on $D^2$ winds positively. Fix then on it any determination of the vector $\varphi_i - \varphi_j$ and let it evolve along $\gamma$ in the given direction, keeping its tail on the curve; meanwhile, its head will move along a closed curve in $L_k$, which could be either the same or a different one. In the former case we are staying inside the same branched disk $L_l$, in the latter we are dealing with two different disks $L_k$ and $L_l$. In any case, the vector will eventually come back to the initial one after having run over the whole of $\gamma$; it makes then sense to consider the degree of the $S^1$-valued map $\frac{\varphi_i - \varphi_j}{|\varphi_i - \varphi_j|}$ on $\gamma$. 
Introduce the m.v. graph $\varphi_i - \varphi_j$ for $i,j$ in the $L_k$ and $L_l$ involved. This m.v.graph will in both cases have a unique connected component. By the blowing-up argument above, 
\[\frac{\varphi_i - \varphi_j}{|\varphi_i - \varphi_j|}(z_l+\rho z)= \frac{\Theta_{ij}^{\rho}}{|\Theta_{ij}^{\rho}|}(z) \rightarrow \frac{\Theta_{ij}}{|\Theta_{ij}|}(z)\]
must contribute with a strictly positive degree on $\gamma$ if $\rho$ was small enough, since $\Theta_{ij}$ is almost-holomorphic.
\end{proof}

\textbf{Proof of the non-accumulation.}
Denote by $\pi^\ast C$ the following subset of ${\R}^3\times {\R}^3\times D^2$ :
\[
\pi^\ast C:=\left\{
\begin{array}{c}
\ds\xi=(\zeta_1,\zeta_2,z)\in {\R}^3\times {\R}^3\times D^2\quad\mbox{s.t. }\quad\exists j,k\in\{1\cdots Q\} \mbox{ satisfying }\\[5mm]
\ds\zeta_1=(\varphi_j(z),\al_j(z))\quad\mbox{ and }\quad\zeta_2=(\varphi_k(z),\al_k(z))
\end{array}
\right\}
\]
By an abuse of notation we will also write $\zeta_1=(\varphi_1,\al_1)$ and $\zeta_2=(\varphi_2,\al_2)$, moreover\footnote{Here $\zeta_i$ ($i\in\{1,2\}$) will always be an element of $\R^3$ of the form $(\varphi_j(z),\al_j(z))$; it should not be confused with the complex coordinate $\zeta$ in $\C_z \times \C_\zeta \times \R_a$ used in sections \ref{PDEaverage} and \ref{uniquecont}, which will anyway not appear in this section.} we denote $z=\pi(\xi)$ - i.e. $\pi$ is extended naturally to $\pi^\ast C$.

Observe that $C\subset \pi^\ast C$ as the result of the identification of $C$ with the points $(\zeta_1,\zeta_2,z)$ such that $\zeta_1=\zeta_2$. Away
from these points, $\pi^\ast C\setminus C$ realizes a smooth 2-dimensional oriented submanifold of ${\R}^3\times {\R}^3\times D^2$
with local chart given by $z$.

On $\pi^\ast C$ we define the function 
\[
d(\xi)=|\zeta_1-\zeta_2|=\sqrt{|\al_1-\al_2|^2+|\varphi_1-\varphi_2|^2}\quad.
\]
which is smooth and non-zero on $\pi^\ast C\setminus C$ and on $\pi^\ast C\setminus C$ we define 
\[
\Delta(\xi)=\frac{|\al_1-\al_2|^2}{|\al_1-\al_2|^2+|\varphi_1-\varphi_2|^2}\quad.
\]
Let $\phi$ be a smooth non negative compactly supported function satisfying 
\[
\phi(s)=\left\{
\begin{array}{l}
1\quad\mbox{ for }s<1\\[5mm]
0\quad\mbox{ for }s>2
\end{array}
\right.
\]
 For $1>\delta>0$ we denote $\phi_\delta(\cdot)=\phi(\cdot/\delta)$. 
 
 \medskip
 
Let $\delta<1$ be a regular value of the function $\Delta$ on $\pi^\ast C\setminus C$ we define a {\it stretching-contracting} map
\[
S_\delta\ :\ {\R}^3\longrightarrow {\R}^3
\]
in the following way : $S_\delta$ is axially symmetric about the $z-$axis, $|S_\delta(x,y,z)|=|(x,y,z)|$ and the following conditions
are satisfied 
\[
S_\delta(x,y,z)=\left\{
\begin{array}{l}
\ds S_\delta(x,y,z)=sgn(z)\ (0,0,\sqrt{x^2+y^2+z^2})\quad\mbox{ if }\quad\frac{z^2}{x^2+y^2+z^2}>\delta\\[5mm]
\ds S_\delta(x,y,z)=(x,y,z)\quad\mbox{ if }\quad\frac{z^2}{x^2+y^2+z^2}<\frac{\delta}{2}
\end{array}
\right.
\]
Denote $N$ the following 3-dimensional manifold:
\[
N:=\lf\{(\xi,t)\in (\pi^\ast C\setminus C)\times {\R}\rg\}\quad .
\]
denote 
$$
D=D_\delta=\frac{1}{\sqrt{\frac{1}{\delta}-1}}\quad .
$$
Observe that $D>0$ has been chosen in particular in such a way that
\be
\label{I.2}
D^{-1}\ |\al_1-\al_2|\le |\varphi_1- \varphi_2|\quad\Longleftrightarrow\quad\Delta(\xi)\le\delta\quad\Longleftrightarrow\quad \phi_\delta(\Delta(\xi))=1\quad .
\ee

At this stage we are going to make a short digression to choose a suitable value for $\delta < 1$, which will be kept throughout the rest of the section.

Let $R$ be the radius of $D^2$. Denote by $B_r$, for $r \leq R$, the part of $\pi^\ast C\setminus C$ above the set $\{ |z|<r\}$. For any $\delta < 1$, express the set $\{\Delta > \delta\}$ as the union of its connected components, i.e. $ \displaystyle \{\Delta > \delta\} = \cup_i A_{\delta}^i$.
We are going to prove the following claim: there exist $\delta <1$ and $\ovr < R$ s.t. 
\be
\label{eq:propdelta}
\forall i \;\;\;\text{ and } \forall r\leq \ovr \;\;\;\  A_{\delta}^i \cap \partial B_r \neq \emptyset \Rightarrow A_{\delta}^i \cap \partial B_R = \emptyset.
\ee

To prove the claim, we argue by contradiction: assume the existence of sequences $\delta_n \to 1$, $r_n \to 0$ for which we can always find a connected component intersecting both $\partial B_{r_n}$ and $\partial B_R$. Then we can choose $C^1$ curves $\gamma_n$, parametrized by arc length, joining $\partial B_{r_n}$ to $\partial B_R$ and staying inside the corresponding connected component.
Up to a subsequence, by the Ascoli-Arzel\`a's theorem, we can assume the existence of a uniform limit curve $\gamma$, joining $0$ to $\partial B_R$.
The function $\Delta$ is greater than $\delta_n$ on the image of $\gamma_n$, therefore
\[\delta_n \to 1 \Rightarrow \Delta \circ \gamma \equiv 1 \Rightarrow |\varphi_1 - \varphi_2|\to 0 \text{ as } n \to \infty.\]
The limit curve $\gamma$ could a priori be merely continuous and not $C^1$.
We can write, from (\ref{eq:FG}), for any $n$ and for any $t$ in the domain of $\gamma_n$:
\[|\alpha_1 - \alpha_2|(\gamma_n(t)) \leq |\varphi_1 - \varphi_2|(\gamma_n(0)) + K_0 \int_0^t |\alpha_1 - \alpha_2|(\gamma_n(s))\; ds.\]
Sending to the limit as $n \to \infty$
\[|\alpha_1 - \alpha_2|(\gamma(t)) \leq  K_0 \int_0^t |\alpha_1 - \alpha_2|(\gamma(s))\; ds,\]
thus $\alpha_1 - \alpha_2$ is identically $0$ on the curve $\gamma$; here $\varphi_1 - \varphi_2$ also vanishes and therefore the image of $\gamma$ is a line of singularities, contradiction. Thus the claim is proved. Of course we can also choose $\delta$ to be a regular value for $\Delta$, since almost all values are as such. End of the digression.

Now, for the $\delta$ given by the claim, take any positive $r \leq \ovr$ arbitrarily small and such that $\pi^{-1}(\partial B^2_r(0))$ does not intersect
the set of $z_l$ satisfying (\ref{I.1}). Let 
$$
\ep_0:=\inf\left\{\frac{d(\xi)}{\sqrt{1+D^2}}\ ;\  \xi\in (\pi^\ast C\setminus C)\cap\pi^{-1}(\p B^2_r(0))\right\}\quad .
$$
By assumption $\ep_0>0$.

Let $\ep>0$ be a regular value less than $\ep_0$ for the function $|\varphi_1-\varphi_2|$. Denote by $g$ the following function on $\pi^\ast C\setminus C$ :
\[
g(\xi):=\frac{\varphi_1-\varphi_2}{\max\{|\al_1-\al_2|, D\ep\}}\quad .
\]
Observe that since $(C\res B^5)\setminus\{0\}$ is assumed to be a smooth Special Legendrian curve and since $(\varphi_j(0),\al_j(0))=(0,0)$ for all $j$,
$|\varphi_1-\varphi_2|^{-1}(\{\ep\})$ is a smooth compact  curve in $\pi^\ast C\setminus C$ for any regular value $\ep>0$.
Observe moreover that since $\ep<\ep_0$ we have that
\be
\label{I.3}
\lf[\pi^\ast C\setminus\lf\{\xi\ ;\ \Delta(\xi)>\delta \rg\}\cap |\varphi_1-\varphi_2|^{-1}(\{\ep\})\rg]\quad\cap\quad\p B^2_r(0)=\emptyset\quad .
\ee

Define the open set $U$ made of the connected components of $\{\Delta > \delta\}$ that intersect $B_r$ (and therefore not $\partial B_R$ thanks to (\ref{eq:propdelta})). 

For any fixed $r \leq \ovr$, choose $\eps$ small enough as follows: firstly, $\eps < \eps_0$; secondly, take 
\[\eps < \min\left\{|\varphi_1- \varphi_2|(\xi): \xi \in \overline{\partial\left( U \cap (B_R - B_r) \right) - \partial B_r} \subset \p U\right\}.\]
The minimum on the r.h.s. is strictly positive. Indeed, if it were $0$, then either we would have a singular point that realizes it, or a smooth point where $\Delta=1$. In the former case, lemma \ref{lem:lemma1} tells us that there is a neighbourhood of the singularity where $\{\Delta < \frac{\delta}{2}\}$, therefore it cannot be a boundary point of $U$, since in $U$ we have $\Delta > \delta$. In the latter case there ought to be a neighbourhood where $\{\Delta > \delta\}$, so it could not be a boundary point. 

Finally define the open set 
\[\Sigma_{r, \eps} = \left (\{|\varphi_1 - \varphi_2| < \eps \} \cap  B_r \right ) \cup U .\]

$\Sigma_{r, \eps}$ has the following properties: 
\begin{description}
	\item[(i)]\[z_l \in \pi(\Sigma_{r, \eps}) \Rightarrow z_l \in \pi(B_r), \text{ since there are no singularities in $U$}\] \[\text{due to lemma \ref{lem:lemma1};}\]
	\item[(ii)]\[p \in \partial \Sigma_{r, \eps} \Rightarrow \left \{ \begin{array}{c}
|\varphi_1 - \varphi_2|(p) = \eps \\
\Delta (p) \geq \delta
\end{array}
\right. \text{ or } \left \{ \begin{array}{c}
|\varphi_1 - \varphi_2|(p) \geq \eps \\
\Delta (p)= \delta
\end{array}
\right.\]
\[\text{ so }\;\; |g|\equiv \sqrt{\frac{1}{\delta} - 1}=D^{-1} \text{ on } \partial \Sigma_{r, \eps}.\]
\end{description}

Thus $\delta$ and $\ep$ have been chosen in such a way that $\p \Sigma_{r,\ep}$ is a closed smooth compact curve in $\pi^\ast C\setminus C$
which is included in the level set $|g|^{-1}(\{D^{-1}\})$. Remark that $\p \Sigma_{r,\ep}$ is obtained by homotopy from the loop $\pi^{-1}\{|z|=r\}$ without crossing any singularity of $C \subset \pi^\ast C$.

On $N$ we define the map $v$ given by
\[
\begin{array}{rcl}
v\ :\ N&\longrightarrow& S^2\\[5mm]
(\xi,t) &\longrightarrow & \ds
\frac{(g(\xi),\al_1-\al_2+t\,\phi_\delta\circ\Delta(\xi))}{\sqrt{|g(\xi)|^2+|\al_1-\al_2+t\,\phi_\delta\circ\Delta(\xi)|^2}}
\end{array}
\]
Observe that $|g(\xi)|^2+|\al_1-\al_2+t\,\phi_\delta\circ\Delta(\xi)|^2=0$ implies that $|\varphi_1-\varphi_2|=0$. If $\al_1-\al_2\ne 0$
then $\phi_\delta\circ\Delta(\xi)=0$ and hence we would have $|\al_1-\al_2|=0$ which is a contradiction. Hence $v$ is well defined smooth map on $N$.
Finally define the $S^2$-valued map $u$  by
\[
u=S_\delta\circ v\quad .
\]
On the complement of $\Sigma_{\ep,r}$ $v$ simplifies to
\be
\label{I.4a}
v(\xi)=\frac{(g(\xi),\al_1-\al_2+t)}{\sqrt{|g(\xi)|^2+|\al_1-\al_2+t|^2}}\quad .
\ee
From the definition of $S_\delta$, for any two form $\om$ on $S^2$ we have hence that, on $N\setminus (\Sigma_{\ep,r}\times {\R})$,
$(S_\delta\circ v)^\ast \om=0$ for $|t|>1/\ep$ (Assuming without loss of generality that $d(\xi)$ is  bounded by 1 on $\pi^\ast C$).
Hence the degree of $u$ restricted to any closed compact curve in the complement of $\Sigma_{\ep,r}$ times ${\R}$ is well defined since in 
$N\setminus (\Sigma_{\ep,r}\times {\R})$ we have $u^\ast\om\ne0$ only on a compact set.

The rest of the section is occupied with the proof of the following two lemmas, which will imply by a simple homotopy argument that can be found at the end of the section, that
the number of $z_l$ is uniformly bounded and theorem~\ref{th-I.1} will be proved.

\begin{lem}
\label{lem:lemmadegree2}
For any $\xi_l=(\varphi,\al,\varphi,\al,z_l)\in C\cap \ov[\pi^\ast C\setminus C]$ for $\rho>0$ small enough
\be
\label{I.5}
\int_{\pi^{-1}( \p B^2_\rho(z_l))\times {\R}}u^\ast\om\ge 1
\ee
where $\om$ is an arbitrary $2-$form on $S^2$ such that $\int_{S^2}\om=1$.
\end{lem}

\begin{lem}
\label{lem:I.1}
Under the previous notations, there exists a constant $K\in {\R}^+$ independent of $r$ and $\ep$ such that
\be
\label{I.5a}
\int_{\p \Sigma_{\ep,r}\times{\R}}u^\ast \sum_{i=1}^3 x^j\ dx^{j+1}\wedge dx^{j-1}\ge -K\quad.
\ee
\end{lem}

\begin{proof}[\bf{proof of lemma \ref{lem:I.1}}] 
This constitutes the core of the proof of theorem \ref{th-I.1}. 

Observe that $|g(\xi)|\equiv D^{-1}$ on $\p \Sigma_{\ep,r}$. Denote $\la$ the following function on $\Sigma_{\ep,r}\times{\R}$ 
\[
\la(\xi,t):=\sqrt{D^{-2}+(\al_1-\al_2+t)^2}\quad .
\]
We additionally denote by $w$ the following $\C \times \R$-valued map\footnote{Sometimes we will also look at $w$ as a $\R^3$-valued map.} on ${\Sigma}_{\ep,r}\times {\R}$ :
\[
w(\xi,t):=\frac{(g(\xi),\al_1-\al_2+t)}{\la}
\]
Observe that $w=u$ on $\p {\Sigma}_{\ep,r}\times {\R}$.

First we claim that
\be
\label{I.6}
\int_{\Sigma_{\ep,r}\times{\R}}|(S_\delta\circ w)^\ast (dx^1\wedge dx^2\wedge dx^3)|\ d{\mathcal H}^2\ dt\res N<+\infty\quad.
\ee 
We now prove the claim (\ref{I.6}). Let $\xi\in\Sigma_{\ep,r}$, one has
\[
\lf\{\begin{array}{l}
\ds\Delta(\xi)>\delta\quad,\\[5mm]
\ds\mbox{ or }\\[5mm]
\ds |\varphi_1-\varphi_2|<\ep
\end{array}
\rg.\quad\Longleftrightarrow\quad
\lf\{\begin{array}{l}
\ds\frac{|\varphi_1-\varphi_2|}{|\al_1-\al_2|}<D^{-1}\quad,\\[5mm]
\ds\mbox{ or }\\[5mm]
\ds \frac{|\varphi_1-\varphi_2|}{D\,\ep}<D^{-1}
\end{array}
\rg.
\]
which clearly implies that
\be
\label{I.7}
|g(\xi)|<D^{-1}\quad\mbox{ in }\quad{\Sigma}_{\ep,r}\quad \Longrightarrow\quad |w(\xi,t)|\le 1\quad\mbox{ in }\quad{\Sigma}_{\ep,r}\quad . 
\ee
We write on one hand
\[
S_\delta^\ast (dx^1\wedge dx^2\wedge dx^3)=det(\nabla S_\delta)(y)\ dy^1\wedge dy^2\wedge dy^3\quad 
\]
and locally on the other hand
\be
\label{I.8}
\begin{array}{l}
\ds w^\ast(dy^1\wedge dy^2\wedge dy^3)= \la^{-3} df^1\wedge df^2\wedge d(\al_1-\al_2+t) \\[5mm]
\ds\quad\quad+\la^{-2}\  d\la^{-1}\wedge\lf(f^1\, df^2-f^2\, df^1\rg)\wedge d(\al_1-\al_2+t)\\[5mm]
\ds\quad\quad+\la^{-2}\  (\al_1-\al_2+t)\ df^1\wedge df^2\wedge d\la^{-1}
\end{array}
\ee
where\footnote{$g^1$ and $g^2$ denote respectively the real and imaginary part of $g$.} locally $f(z):=g^1(\xi(z))+ig^2(\xi(z))$. Observe now that the following 3 and 2-forms are zero
\be
\label{I.9}
df^1\wedge df^2\wedge d(\al_1-\al_2)\equiv 0\quad\mbox{ and }\quad d\la^{-1}\wedge d(\al_1-\al_2+t)\equiv 0\quad.
\ee
Hence (\ref{I.8}) becomes, from the definition of $\lambda$,
\be
\label{I.9a}
\begin{array}{l}
w^\ast(dy^1\wedge dy^2\wedge dy^3)= \la^{-3} df^1\wedge df^2\wedge dt\\[5mm]
\ds\quad\quad -\la^{-5}(\al_1-\al_2+t)^2\ df^1\wedge df^2\wedge dt\\[5mm]
\ds\quad\quad=\la^{-5}\ D^{-2}\ df^1\wedge df^2\wedge dt
\end{array}
\ee
We rewrite
\be
\label{I.9b}
w^\ast(dy^1\wedge dy^2\wedge dy^3)=\frac{i}{2}\la^{-5}\ D^{-2}\lf[|\p_{z}f|^2-|\p_{\ov{z}}f|^2\rg]\ dz\wedge d\ov{z}\wedge dt\quad.
\ee
We first estimate the following integral :
\be
\label{I.10}
\int_{-\infty}^{+\infty}det(\nabla S_\delta)(w(\xi,t))\ \la^{-5} dt\le D^{-2}\ C_\delta\ \int_{-\infty}^{+\infty}\frac{d\tau}{(D^{-2}+\tau^2)^\frac{5}{2}}\le C_{\delta}
\ee
Observe that 
\be
\label{I.11}
|\nabla f|\le \ep^{-1}D^{-1}|\nabla(\varphi_1-\varphi_2)|+\ep^{-2} D^{-2}|\varphi_1-\varphi_2| |\nabla(\al_1-\al_2)| \quad .
\ee
Since $\int_{D^2}\sum_{j=1}^Q|\nabla \varphi_j|^2+|\nabla \al_j|^2<+\infty$ combining (\ref{I.9}), (\ref{I.10}) and (\ref{I.11}) we obtain the claim (\ref{I.6}).

\medskip

We now establish the lower bound (\ref{I.5a}). To that purpose we compute an equation for $f$.

From the equations in (\ref{I.1a}) we deduce that locally
\be
\label{I.12}
\lf\{
\begin{array}{rl}
\ds\p_{\ov{z}}(\varphi_1-\varphi_2)&\ds=\nu(\varphi_2,\al_2)\ \p_{z}(\varphi_1-\varphi_2)+\lf[\nu(\varphi_1,\al_1)-\nu(\varphi_2,\al_2)\rg]\ \p_z \varphi_1\\[5mm]
&\quad+\mu(\varphi_1,\al_1)-\mu(\varphi_2,\al_2)\\[5mm]
\ds\nabla(\al_1-\al_2)&\ds=h(\varphi_1,\al_1)-h(\varphi_2,\al_2)
\end{array}
\rg.
\ee
We have that
\be
\label{I.13}
\p_{\ov{z}}f=\frac{\p_{\ov{z}}(\varphi_1-\varphi_2)}{\max\{|\al_1-\al_2|,D\ep\}}-f\ {\mathbf 1}_{|\al_1-\al_2|>D\ep}\ \frac{\p_{\ov{z}}|\al_1-\al_2|}{\max\{|\al_1-\al_2|,D\ep\}}
\ee
where ${\mathbf 1}_{|\al_1-\al_2|>D\ep}$ is the characteristic function of the set where $|\al_1-\al_2|>D\ep$.
Inserting now (\ref{I.12}) in (\ref{I.13}) we obtain
\be
\label{I.14}
\begin{array}{l}
\ds\p_{\ov{z}}f=\nu(\varphi_2,\al_2)\ \frac{\p_z(\varphi_1-\varphi_2)}{\max\{|\al_1-\al_2|,D\ep\}}+\frac{\lf[\nu(\varphi_1,\al_1)-\nu(\varphi_2,\al_2)\rg]}{\max\{|\al_1-\al_2|,D\ep\}}\ \p_za_1\\[5mm]
\ds\quad+\frac{\mu(\varphi_1,\al_1)-\mu(\varphi_2,\al_2)}{\max\{|\al_1-\al_2|,D\ep\}}-f\ {\mathbf 1}_{|\al_1-\al_2|>D\ep}\ \frac{\p_{\ov{z}}|\al_1-\al_2|}{\max\{|\al_1-\al_2|,D\ep\}}
\end{array}
\ee
From which we deduce
\be
\label{I.15}
\begin{array}{l}
\ds\p_{\ov{z}}f=\nu(\varphi_2,\al_2)\ \p_zf+\nu(\varphi_2,\al_2)\ f\ {\mathbf 1}_{|\al_1-\al_2|>D\ep}\ \frac{\p_{z}|\al_1-\al_2|}{\max\{|\al_1-\al_2|,D\ep\}}\\[5mm]
\ds\quad+\frac{\lf[\nu(\varphi_1,\al_1)-\nu(\varphi_2,\al_2)\rg]}{\max\{|\al_1-\al_2|,D\ep\}}\ \p_za_1+\frac{\mu(\varphi_1,\al_1)-\mu(\varphi_2,\al_2)}{\max\{|\al_1-\al_2|,D\ep\}}\\[5mm]
\ds\quad-f\ {\mathbf 1}_{|\al_1-\al_2|>D\ep}\ \frac{\p_{\ov{z}}|\al_1-\al_2|}{\max\{|\al_1-\al_2|,D\ep\}}\quad.
\end{array}
\ee
Using now the second equation in (\ref{I.12}) we obtain the existence of a constant $K_0>0$ such that
\be
\label{I.16}
|\nabla(\al_1-\al_2)|\le K_0\ \lf[|\varphi_1-\varphi_2|+|\al_1-\al_2|\rg]\quad.
\ee
This later fact gives
\be
\label{I.17}
\lf|\frac{\nabla(\al_1-\al_2)}{\max\{|\al_1-\al_2|,D\ep\}}\rg|\le K_0\ \lf[|f|+1\rg]\quad.
\ee
Combining (\ref{I.15}) and (\ref{I.17}) we obtain the following bound : There exists $K_1>0$ and $K_2>0$ such that
\be
\label{I.18}
\lf|\p_{\ov{z}}f-\nu(\varphi_2,\al_2)\ \p_zf\rg|\le K_1\lf[|f|+1\rg]\ |\p_z \varphi_1|+K_2\ \lf[|f|^2+1\rg]\quad .
\ee
From (\ref{I.9b}) we have that 
\be
\label{I.19}
\begin{array}{l}
\ds\int_{\Sigma_{\ep,r}\times{\R}}(S_\delta\circ w)^\ast dx^1\wedge dx^2\wedge dx^3\\[5mm]
\ds\quad=\left(\int_{\pi(\Sigma_{\ep,r})}D^{-2}\lf[|\p_{z}f|^2-|\p_{\ov{z}}f|^2\rg]\ \frac{i}{2}dz\wedge d \overline{z}\right) \left( \int_{-\infty}^{+\infty}
det(\nabla S_\delta)\circ w\ \la^{-5}\ dt \right) .
\end{array}
\ee
 Since $det(\nabla S_\delta)(y)\ge 0$ on ${\R}^3$, 
$$
\eta(z):=\int_{-\infty}^{+\infty}
det(\nabla S_\delta)\circ w\ \la^{-5}\ dt \ge 0\quad .
$$
Moreover we also have the following bound given by (\ref{I.10})
\be
\label{I.19a}
\eta\le C_{D}=C_{\delta}\quad .
\ee
Using (\ref{I.18}) we then deduce the following lower bound
\be
\label{I.20}
\begin{array}{l}
\ds\int_{\Sigma_{\ep,r}\times{\R}}(S_\delta\circ w)^\ast dx^1\wedge dx^2\wedge dx^3\\[5mm]
\ds\quad\ge\int_{\Sigma_{\ep,r}}D^{-2}\ \lf[1-\nu^2(\varphi_2,\al_2)|\p_{z}f|^2\rg]\ \eta\ \frac{i}{2} dz\wedge d \overline{z}\\[5mm]
\ds\quad - \ti{C}_\delta\ \int_{\Sigma_{\ep,r}} \lf[4(K_1)^2(|f|+1)^2\ |\p_z \varphi_1|^2+4(K_2)^2\ (|f|^2+1)^2\rg]\ \ \frac{i}{2} dz\wedge d \overline{z}\quad .
\end{array}
\ee
Using the fact that $|f(z)|=|g(\xi)|\le D^{-1}$ on $\Sigma_{\ep,r}$,  and that, for $r$ small enough $|\nu(\varphi_2,\al_2)|<1/2$, we obtain the existence of a constant $K_\delta$ such that
\be
\label{I.21}
\begin{array}{l}
\ds\int_{\Sigma_{\ep,r}\times{\R}}(S_\delta\circ w)^\ast dx^1\wedge dx^2\wedge dx^3\\[5mm]
\ds\quad\ge-K_\delta\int_{D^2}\sum_{j=1}^Q\ \lf[|\nabla \varphi_j|^2+1\rg]\ \frac{i}{2} dz\wedge d \overline{z} \geq -K,\quad.
\end{array}
\ee
with $K>0$ independent of $r$ and $\eps$. 

Recall now that $w=u$ on $\p \Sigma_{r,\eps}\times \R$. Then by Stokes theorem
\[\int_{\Sigma_{\ep,r}\times{\R}}(S_\delta\circ w)^\ast dx^1\wedge dx^2\wedge dx^3 = \int_{\p \Sigma_{\ep,r}\times{\R}}(S_\delta\circ w)^\ast \sum_{i=1}^3 x^j\ dx^{j+1}\wedge dx^{j-1} =\]\[= \int_{\p \Sigma_{\ep,r}\times{\R}}(S_\delta\circ u)^\ast \sum_{i=1}^3 x^j\ dx^{j+1}\wedge dx^{j-1}.\]

This is the desired lower bound (\ref{I.5a}) and lemma \ref{lem:I.1} is proved.

\end{proof}

\begin{proof}[\bf{proof of lemma \ref{lem:lemmadegree2}}]
The result follows straight from lemma \ref{lem:lemmadegree}. Observe that, by lemma \ref{lem:lemma1} and by homotopy, the degree computed there is the same as the degree of the function
\[\frac{\varphi_i - \varphi_j}{D \varepsilon}= \frac{\varphi_i - \varphi_j}{\max{|\alpha_i - \alpha_j|, D \varepsilon}}= g\] 
on the loop $\{|\phi_i - \phi_j|= \varepsilon\}$ around $z_l$.
By the same computation performed in (\ref{I.19}) (we can take without loss of generality $\omega=\sum_{i=1}^3 x^j\ dx^{j+1}\wedge dx^{j-1}$), since the degree of $g$ is exactly $\int_{\pi(\Sigma_{\ep,r})}D^{-2}\lf[|\p_{z}f|^2-|\p_{\ov{z}}f|^2\rg]\ \frac{i}{2} dz\wedge d \overline{z}$,  we get that the degree of $S_{\delta} \circ w$ is strictly positive.
\end{proof}

\begin{proof}[\bf{proof of theorem \ref{th-I.1}}]
We argue by contradiction. If we had countably many singularities of the form (\ref{I.1}) accumulating onto $0$, around each such singular point, on $\pi^{-1}( \p B^2_\rho(z_l))\times {\R}$, we would have a strictly positive degree for $u$, thanks to lemma \ref{lem:lemmadegree2}. Let us observe, however, the degree of $u$ on $\p B_r \times \R$; this is the same as the degree of $u$ on $\p \Sigma_{r,\eps} \times \R$, since these two $2$-surfaces are homotopic and we do not cross any singularity during this homotopy (see $(ii)$ on page 64 and recall that $u$ is smooth out of the singularities). Choosing $r$ smaller and smaller, we must then have, under the contradiction assumption, that the degree of $u$ on $\p B_r \times \R$ goes to $- \infty$ as $r \to 0$, which contradicts lemma \ref{lem:I.1}. 
\end{proof}

\end{document}